\title{\large{\textbf     {DECOUPLING INEQUALITIES FOR THE\\ 
 GINZBURG-LANDAU $\nabla \varphi$ MODELS 
}}}

\date{}

\documentclass[12pt, leqno]{article}
\usepackage[english]{babel}
\usepackage{amsmath}
\usepackage{amssymb}
\usepackage{amsfonts}
\usepackage{amsthm}
\usepackage[a4paper,left=2.4cm,right=2.4cm,top=2.7cm,bottom=2.7cm]{geometry}
\usepackage{amsfonts, amsmath, amssymb, amsthm, mathrsfs}
\usepackage{graphicx}
\usepackage{color}

\usepackage{color}

\usepackage{comment}
\usepackage[hang,flushmargin]{footmisc}

\numberwithin{equation}{section}

\newtheorem{thm}{Theorem}[section]
\newtheorem{lem}[thm]{Lemma}
\newtheorem{proposition}[thm]{Proposition}

\newtheorem{cor}[thm]{Corollary}

\theoremstyle{remark}

\theoremstyle{definition}
\newtheorem{remark}[thm]{Remark}

\addtocounter{section}{-1}

\begin{document}

\maketitle

\begin{center}
\vspace{-1.2cm}
Pierre-Fran\c cois Rodriguez$^1$ 

\vspace{1.3cm}
Preliminary draft\\
\vspace{1.3cm}
\textbf{Abstract}
\end{center}
\noindent 
We consider a class of of massless gradient Gibbs measures, in dimension greater or equal to three, and prove a decoupling inequality for these fields. As a result, we
obtain detailed information about their geometry, and the percolative and non-percolative phases of their level sets, thus generalizing results obtained in~\cite{RoS13}, to the non-Gaussian case. Inequalities of similar flavor have also been successfully used in the study of random interlacements~\cite{Sz12a},~\cite{PT12}. A crucial aspect is the development of a suitable sprinkling technique, which relies on a particular representation of the correlations in terms of a random walk in a dynamic random environment, due to Helffer and Sj\"ostrand. The sprinkling can be effectively implemented by studying the Dirichlet problem for the corresponding Poisson equation, and quantifiying in how far a change in boundary condition along a sufficiently ``small'' part of the boundary affects the solution. Our results allow for uniformly convex potentials, and extend to non-convex perturbations thereof.
\thispagestyle{empty}

\vspace{6cm}

\begin{flushleft}

$^1$Department of Mathematics \hfill December 2016 \\
University of California, Los Angeles \\
520, Portola Plaza, MS 6172\\
Los Angeles, CA 90095 \\
\texttt{rodriguez@math.ucla.edu}
\end{flushleft}

\vspace{2cm}

\newpage
\mbox{}
\thispagestyle{empty}
\newpage

\section{Introduction}

A considerable effort has recently gone into understanding various strongly correlated occupation fields and their associated percolation phase transition, originating in the works \cite{Sz10} and \cite{Sz12a}, see also \cite{PT12}, on random interlacements, leading up to a series of recent articles \cite{DRS14.2}, \cite{DRS14.3} \cite{PRS15}, and \cite{Sa14}, which provide a very detailed picture of the geometry in the (strongly) supercritical regime, for a broad class of models and under a rather general set of assumptions. A common feature among all of these works is their crucial reliance on the availability of a so-called \textit{decoupling inequality}, which is a certain way to quantify the decay of correlations, for models with long-range dependence, with far-reaching consequences. The purpose of this work is to derive such a correlation inequality for convex gradient interface models (and non-convex perturbations thereof), in all dimensions $d\geq 3$, see for instance \cite{Fu05}, \cite{Ve06} for an introduction to the subject, thus generalizing results obtained in~\cite{RoS13}, see also \cite{PR13}, to the anharmonic case. 

Let us explain the essence of the decoupling technique more precisely. Suppose that $\mu$ is a gradient Gibbs measure on $\mathbb{Z}^d$, $d\geq3$, obtained as the weak limit of a sequence of finite volume approximations, with $0$-boundary conditions (i.e., zero tilt), and $\varphi$ the corresponding canonical field, see Section \ref{S:1} for precise definitions. For now, the reader may think of the interaction potential $V= V(\nabla \varphi)$ as being a uniformly convex two-body interaction. Formally, $\mu$ is a probability measure on $\Omega= \mathbb{R}^{\mathbb{Z}^d}$, with
\begin{equation}\label{eq:intro1}
d\mu(\varphi) \propto e^{-\beta \sum_{x\sim y}V(\varphi_x-\varphi_y)} d\varphi, \text{ with } c \leq V'' \leq c', \text{ for some $c,c' \in (0,\infty)$}.
\end{equation}
We will typically set $\beta=1$ and omit it from the notation, but our setup will allow more generally for (finite-range) multi-body interactions, subject to a suitable random walk representation condition, in order to eventually handle the case of non-convex perturbations (for which $\beta$ becomes relevant). For a massless $\mu$ as in \eqref{eq:intro1}, it has been known at least since the work of Naddaf and Spencer \cite{NS97} that correlations have slow polynomial decay, i.e. $\mathbb{E}_{\mu}[\varphi_x\varphi_y] \sim |x-y|^{2-d}$, as $|x-y| \to \infty$. More generally, if $f=f((\varphi_x)_{x\in S})$ and $g((\varphi_y)_{y\in S'})$ are two local observables, with $f,g \in L^{\infty}(\Omega)$, $S = B(x,L)$, the $\ell^{\infty}$-box around $x$ and $S'= B(x', L)$, with $|x-x'| = RL$, for some large $R$, then one cannot hope for a better bound than
\begin{equation}\label{eq:intro2}
\textnormal{Cov}_{\mu}(f,g) \sim R^{-(d-2)}
\end{equation}
(see for instance Prop.~1.1 in \cite{PR13} for a precise statement in the Gaussian case $V(\eta)=\eta^2$). In particular, if one thinks of $R$ as a being large, but fixed, and sending $L \to \infty$, this does not decay at all. 

As we now explain, the situation can become drastically different, if at least one of the observables depends \textit{monotonically} on some external parameter, and one is willing to adjust this parameter a bit, an idea often referred to as \textit{sprinkling} in the literature. Specifically, suppose that $f = f^h  = 1_{A^h}$, where $A^h(\varphi)$ is an increasing event, measurable with respect to the occupation variables $(1\{\varphi_x \geq h \}_{x\in S})$ at level $h$, for some $h \in \mathbb{R}$, then our main results, cf. Theorem \ref{P:dec_ineq} and \ref{T:dec_ineq2} below, imply that
\begin{equation}\label{eq:intro3}
\mathbb{E}_{\mu}[f^h  g] \leq  \mathbb{E}_{\mu}[f^{h-\varepsilon}] \cdot \mathbb{E}_{\mu}[ g] + \Vert g\Vert_{L^{\infty}} \cdot\delta_{S,S'}(\varepsilon)
\end{equation}
and the error term $\delta_{S,S'}(\cdot)$ can for instance be made as small as $e^{-L^{\alpha}}$, if $\varepsilon > R^{-\beta}$, for some $\alpha, \beta > 0$ (under the assumptions preceding \eqref{eq:intro2}). Note that \eqref{eq:intro3} is quite sharp, for if $g$ is increasing in $\varphi$, then the left-hand side of \eqref{eq:intro3} is bounded from below by $\mathbb{E}_{\mu}[f^h]\cdot \mathbb{E}_{\mu}[g]$, due to the FKG-inequality (which holds for $\mu$ as in \eqref{eq:intro1}).

We now describe the mechanism behing our sprinkling technique. The inequality \eqref{eq:intro3} amounts to saying that the conditional law $\mathbb{E}_{\mu}[f^h |\mathcal{F}_{S'} ]$, where $\mathcal{F}_{S'}=\sigma(\varphi_x, x\in S')$, is suitably close to $\mathbb{E}_{\mu}[f^{h-\varepsilon}]$. In the Gaussian case, see \cite{RoS13}, and also \cite{PR13}, this comparison is made reasonably straightforward by harnessing the fact that the Gaussian free field (GFF) in $\Lambda \subseteq \mathbb{Z}^d$ with $\Lambda \supset S'$ admits the decomposition
\begin{equation}\label{eq:harmonic_ext}
\begin{split}
&\text{GFF on $\Lambda \setminus S'$ with b.c. $\psi$ } \\
&\qquad \stackrel{\text{law}}{=} \text{GFF on $\Lambda \setminus S'$ with $0$ b.c. + harmonic extension of $\psi$ to $\Lambda \setminus S'$, }
\end{split} 
\end{equation}
which readily yields an explicit formula for conditional distributions (due to the Gibbs nature of $\mu$, conditioning on $\mathcal{F}_{S'}$ manifests itself as a boundary condition). A considerable effort is devoted to finding a suitable replacement for this Markovian structure of the GFF (incidentally, similar issues were faced in \cite{Mi11} for the analysis of fluctuations in a bounded domain in dimension $2$, cf. in particular Theorem 1.2 therein). One is naturally led to wonder how close the GFF and the anharmonic model really are, see Remark \ref{R:final}, 2), for more on this.

Our approach for general $\mu$, developed in Section \ref{S:sprinkle}, see in particular Proposition \ref{L:monot}, is based on an interpolation argument, by which the field $\varphi_{|{\partial_{\text{i}} S'}}$ ($\partial_{\text{i}} S'$ stands for the interior boundary of $S'$), felt in $\Lambda\setminus S'$ as a boundary condition upon conditioning on $\mathcal{F}_{S'}$, is progressively replaced by an independent copy $\widetilde{\varphi}_{|{\partial_{\text{i}} S'}}$. This interpolation has undesirable properties, since $\widetilde{\varphi}$ is signed, thus blending it in will generically not act monotonically towards producing a desired (upper, in our case) bound for $\mathbb{E}_{\mu}[f^h |\mathcal{F}_{S'} ]$. The idea is that this problem can be controlled by introducing a \textit{balancing} effect, by which the field on $\partial \Lambda$, the outer boundary of $\Lambda$, is given a slight ``push'' upward by $\varepsilon$, which is carried through $\Lambda$ due to the gradient nature of $\varphi$. Crucially, the competing influence of the varying boundary conditions along $\partial_{\text{i}}S'$ and $\partial\Lambda$ can be quantified using the Helffer-Sj\"ostrand representation~\cite{HS94}, which has already enjoyed great success in the analysis of this model in the past, cf. for instance \cite{NS97}, \cite{FS97}, \cite{DGI00}, and \cite{GOS01}. This representation allows to effectively rephrase the problem as studying the effect of varying boundary data on the solution of a certain discrete elliptic problem in $\Lambda' = \Lambda \setminus S'$ (which the reader should regard as the analogue of the harmonic extension in the Gaussian case), and, representing this solution probabilistically, one can achieve the desired balance by controlling a single quantity, $\Sigma_{\Lambda}= \Sigma_{\Lambda}(S,S')$, which we call (probabilistic) \textit{cross-section} of $S'$, viewed from $S$, defined as
\begin{equation}\label{eq:intro4}
\Sigma_{\Lambda}(S,S')= \sup_{x \in S} \sup_{\varphi,\xi }\frac{\mathbf{P}_{x,\varphi}^{\mathcal{G}_{\Lambda'},\xi}[H_{S'} < H_{\Lambda^c} ]}{1- \mathbf{P}_{x,\varphi}^{\mathcal{G}_{\Lambda'}, \xi} [H_{S'} < H_{\Lambda^c} ]},
\end{equation} 
where $\mathbf{P}_{x,\varphi}^{\mathcal{G}_{\Lambda'},\xi}$ denotes the annealed law of the associated random walk (which is a jump process among random time-varying conductances, cf. Section \ref{S:1}) started at $x$ (with initial field configuration $\varphi$ and boundary data $\xi$, which determine the evolution of the environment) and $H_K$ denotes the entrance time of the walk in $K\subset\subset \mathbb{Z}^d$.
The cross section $\Sigma_{\Lambda}(S,S')$ can be seen as the correct way to measure the size of the boundary $\partial_{\text{i}}S'$ relative to $\partial \Lambda$, as seen from $S$, and it admits suitable bounds as $\Lambda \nearrow \mathbb{Z}^d$ in case the environment is uniformly elliptic, which follows from our assumptions on $V$ in \eqref{eq:intro3}. While  convenient, this assumption could probably be relaxed, as discussed below in Remark~ \ref{R:final},~1). 

As it turns out, the interpolation becomes unbalanced, and thus the sprinkling by $\varepsilon$ ineffective, roughly when either $\varphi_{|_{\partial_{\text{i}}S'}}$ or $\widetilde{\varphi}_{|_{\partial_{\text{i}}S'}}$ are of the order $\varepsilon \Sigma_{\Lambda}(S,S')^{-1}$, which is at the origin of the error term $\delta_{S,S'}$ appearing in \eqref{eq:intro3}. Note that sharp upper bounds on the cross-section $\Sigma_{\Lambda}$, as $\Lambda \nearrow \mathbb{Z}^d$, are key here because one typically thrives for a regime where
$\varepsilon \Sigma_{\Lambda}(S,S')^{-1} \gg 1$ (to make the sprinkling effective), while keeping $\varepsilon$ as small as possible.

 Bounding $\delta_{S,S'}$ suitably is the subject of Section \ref{S:error}, and the Brascamp-Lieb inequality~\cite{BL76} provides a well-needed concentration estimate to this effect (note that controlling the mean is not an issue because we have zero tilt). For many applications, including those discussed below, the control on the error term $\delta_{S,S'}$ needs to be relatively tight, which will follow from sufficiently careful comparison estimates between hitting distributions under $\mathbf{P}_{x,\varphi}^{\mathcal{G}_{\Lambda'},\xi}(\cdot)$ and those of simple random walk. Our corresponding results are in fact quenched, but annealed estimates would suffice.

The inequality \eqref{eq:intro3}, cf. Theorem \ref{P:dec_ineq}, has a host of applications, and in particular, far-reaching consequences regarding the geometry of the level sets of $\varphi$. Notably, see Theorems~ 
\ref{T:subcrit}, \ref{T:chem_dist} and \ref{T:RW_perc} below, we identify $h_+$ and $h_-$ satisfying $-\infty < h_- \leq h_+ < \infty$, such that
\begin{equation}\label{eq:intro5}
\begin{split}
&\text{for all $h> h_+$, the connectivity function $\mathbb{P}_{\mu}[x\stackrel{\geqslant h}{\longleftrightarrow}y]$, referring to the probability}\\
&\text{that $x$ and $y$ are connected by a nearest-neighbor path along which $\varphi \geq h$, has}\\
&\text{stretched exponential decay in $|y-x|$},
\end{split}
\end{equation}
and, assuming that $\mu$ is translation invariant (which does not seem a-priori clear),
\begin{equation} \label{eq:intro6}
\begin{split}
&\text{for all $h< h_-$, the set $\{ x; \varphi_x \geq h\}$ has a unique infinite cluster, on which large}\\
&\text{balls obey a shape theorem, and on which simple random walk satisfies}\\
&\text{a quenched invariance principle}
\end{split}
\end{equation}
(we refer the reader to Section \ref{S:apps} for precise statements). Results as in \eqref{eq:intro5} are by now routinely obtained by pairing inequality \eqref{eq:intro3} with a suitable static renormalization scheme. This circle of ideas goes back to \cite{Sz10} and \cite{Sz12a}. Moreover, \eqref{eq:intro6} follows from our decoupling inequality (actually an improved version, cf. Theorem \ref{T:dec_ineq2} below), in conjunction with the recent works \cite{DRS14.2}, \cite{DRS14.3} \cite{PRS15}, and \cite{Sa14}.
The thresholds $h_+$ and $h_-$ are conjectured to be equal (and corresponding to the critical point for percolation), which is an open problem, even for $V(\eta)=\eta^2$.

Finally, we extend some of our results, and in particular \eqref{eq:intro5}, see Theorem \ref{T:noncon}, to the case where $V$ is perturbed by a sufficiently small non-convex two-body potential $g$ with compact support (see \eqref{eq:nonconvex1}, \eqref{eq:noncon2} for precise definitions) and at zero tilt, i.e. one sets $V =U + g$ in \eqref{eq:intro1}, with $U$ uniformly convex. Examples satisfying our assumptions include $e^{-V(\eta)}= (\eta^2 + \frac12)e^{-\eta^2}$, which has a double well, and $e^{-V(\eta)}= \int \rho(d\kappa)e^{-\kappa \eta^2}$, a `log-mixture' of Gaussians, with $\rho$ compactly supported on $(0,\infty)$, at zero tilt, which was studied in \cite{BK07}, \cite{BS11} to show in particular that for suitable choice of $\rho$, the gradient Gibbs states on $\mathbb{Z}^2$ are not uniquely characterized by their tilt, in contrast to the convex case~\cite{FS97}. 

The key towards obtaining a decoupling inequality for non-convex $\mu \equiv \mu_{\beta}$, which is a basic observation in many renormalization-group type arguments, and which was used in \cite{DC12} in the context of Ginzburg-Landau models, is that one typically gains convexity by integration. Specifically, the restriction $\tilde{\mu}_{\beta}$ of $\mu_{\beta}$ to, say, the even sublattice ($\mathbb{Z}^d$ is bipartite), can be described by a \textit{convex} Hamiltonian, if $g$ is sufficiently small (as parametrized by $\beta>0$, the inverse temperature). While still retaining a gradient nature in the even spins, the effective Hamiltonian describing the even field is not given by a two-body potential anymore, as integrating an odd variable (note that these are conditionally independent given the even field) ``mixes'' the interaction between all neighboring even spins. However, our methods are sufficiently robust to handle the resulting many-body potentials, and we obtain a decoupling inequality for $\widetilde{\mu}_{\beta}$ (at sufficiently small $\beta$).

\bigskip

This paper is organized as follows. In Section \ref{S:1}, we introduce our setup and collect some useful tools regarding Gradient Gibbs measures, including in particular, the Helffer-Sj\"ostrand representation, see Lemma \ref{L:HS}. Sections \ref{S:sprinkle} and \ref{S:error} are devoted to the implementation of our sprinkling technique, and the estimate of the arising error term, respectively. The central results are Proposition \ref{L:monot} and Lemma \ref{L:error}. Together, the two readily imply our first version of the decoupling inequality, Theorem \ref{P:dec_ineq}. Section~\ref{S:apps} discusses some applications to the geometry of the level sets of $\mu$. First, we set up a renormalization scheme, much in the spirit of \cite{Sz12a}, which leads to stretched exponential decay of the connectivity function at large heights, see Theorem \ref{T:subcrit}. This result is mainly included to later allow for a slightly different error term in our decoupling inequality, see Theorem \ref{T:dec_ineq2}, thus making it amenable to the setup of the recent works \cite{DRS14.2}, \cite{DRS14.3} \cite{PRS15}, and \cite{Sa14}, and yielding a number of results in the ``strongly supercritical regime,'' i.e. when $h$ is sufficiently small, see Theorems \ref{T:chem_dist} and \ref{T:RW_perc}. Finally, Section \ref{S:nonconvex} deals with extensions to the non-convex case. 

A final note regarding our convention with constants: $c,c', c'',\dots$ denote positive constants, possibly depending on $d$, the dimension of the space, which can change from place to place. Numbered constants are defined where they first appear and stay fixed throughout.

\bigskip

\noindent \textbf{Acknowledgments.} The author thanks Marek Biskup for useful discussions, and
the FIM at ETH Zurich and Alain-Sol Sznitman for their hospitality during the summer of 2016, during which part of this research was completed.

\section{Preliminaries}\label{S:1}

We now introduce the measures of interest, along with some notation. We formulate a somewhat general set of conditions for gradient Hamiltonians which guarantee a Helffer-Sj\"ostrand random walk representation in finite volume, akin to the one developed e.g. in \cite{DGI00}, \cite{GOS01}, see also the monograph \cite{Fu05}, for nearest-neighbor two-body potentials. In particular, the specifics of the Hamiltonian are completely irrelevant for the construction so long as it has the required continuous symmetry with respect to shifts (i.e. a gradient nature), and satisfies a suitable convexity assumption.
Our setup allows for many-body interactions satisfying a suitable random walk condition, which will prove useful when treating non-convex perturbations in Section \ref{S:nonconvex}, see also Remark \ref{R:gensetup} below.

We consider the lattice $\mathbb{Z}^d$, and assume tacitly throughout that $d \geq 3$ . We write $\Lambda \subset \subset \mathbb{Z}^d$ to denote a finite subset. We assume that there exists $\Gamma \subset \subset \mathbb{Z}^d$ with the property that $0 \notin \Gamma$, $ \Gamma =  -\Gamma \, (=\{ -x; x \in \Gamma \})$, and require that all $x \in \mathbb{Z}^d$ can be written as $x=\sum_{1\leq i\leq n}y_i$, for some $n \geq 1$ and suitable $y_i \in \Gamma$. The set $\mathbb{Z}^d$ is endowed with the (oriented) edge set $\mathscr{E}=\{ (x,x+y); x \in \mathbb{Z}^d, y \in \Gamma \}$, which satisfies $\mathscr{E}=-\mathscr{E}$. Given an edge $e\in \mathscr{E}$, we write $x(e)$ and $y(e)$ for its endpoints, such that $e=(x(e),y(e))$. We often use $x\sim y$ instead of $(x,y)\in \mathscr{E}$. We write $\mathcal{G}=(\mathbb{Z}^d,  \mathscr{E})$, endowed with its graph distance $|\cdot|$, and balls $B(x,L)= \{ y \in \mathcal{G}; |y-x|\leq L \}$, for $x\in \mathbb{Z}^d$, $L \geq 1$. 
We also use $\vert\cdot \vert_p$ to denote the usual $\ell^p$-distance on $\mathbb{Z}^d$. For a set $\Lambda \subset \mathbb{Z}^d$, we define its outer vertex boundary $\partial \Lambda = \{ z \in \mathbb{Z}^d \setminus \Lambda; \exists x \in K \text{s.t. } (x,z)\in \mathscr{E}  \}$, let $\overline{\Lambda}= \Lambda \cup \partial \Lambda$ and $\partial_{\text{i}}\Lambda = \partial \Lambda^c$, where $\Lambda^c= \mathbb{Z}^d \setminus \Lambda$. Given $\Lambda\subset \subset \mathbb{Z}^d$, we write $\mathcal{G}_{\Lambda}= ( \Lambda \cup \partial \Lambda, \mathscr{E}_{\Lambda} )$ for the induced subgraph where $\mathscr{E}_{\Lambda}\subset \mathscr{E}$ consists of all edges having at least one endpoint in $\Lambda$. 

Let $\Omega_{\Lambda}= \Omega_{\Lambda}$, for $\Lambda \subset \subset \mathbb{Z}^d$, and
$\Omega = \mathbb{R}^{\mathbb{Z}^d}$, endowed with their canonical $\sigma$-algebras $\mathcal{F}_{\Lambda}$, resp. $\mathcal{F}$, and canonical coordinates $\varphi_x: \Omega \to \mathbb{R}$, $\omega \mapsto \omega (x)$, for $x\in \mathbb{Z}^d$. We write
\begin{equation}\label{eq:gradphi}
\nabla \varphi(e)=\varphi_y - \varphi_x, \text{ if } e=(x,y), \, e \in \mathscr{E}
\end{equation}
for the corresponding discrete gradients. We will also consider $\widehat{\Omega}_{\Lambda} =\{0,1 \}^{\Lambda}$, $ \widehat{\Omega} = \widehat{\Omega}_{\mathbb{Z}^d}$, endowed with canonical $\sigma$-algebras $\widehat{\mathcal{F}}_{\Lambda}$, $\widehat{\mathcal{F}}$, and canonical coordinates $Y_x$, $x \in \mathbb{Z}^d$.

We consider a family $\{ V_X\}_{X \in \mathcal{B}}$ of potentials indexed by unit balls in $\mathcal{G}$, i.e. $\mathcal{B}= \{ B(x,1); x \in \mathbb{Z}^d\}$, and ${V}_X: \mathbb{R}^{\mathscr{E}(X)} \to \mathbb{R}_+$, where $\mathscr{E}(X)= \{e \in \mathscr{E}; \, x(e)=x  \}$, if $X=B(x,1)$, are the edges emanating from $x$. For each $\Lambda \subset \subset \mathbb{Z}^d$ a Hamiltonian $H_{{\Lambda}}(\varphi)$ is specified in terms of the potentials $\{ V_X\}_{X \in \mathcal{B}}$ as
\begin{equation}\label{eq:H_pots}
H_{{\Lambda}}(\varphi)= \sum_{X \in \mathcal{B}: X\cap \Lambda \neq \emptyset} V_X((\nabla \varphi(e))_{e\in \mathscr{E}(X)}).
\end{equation}
Our interactions are subject to the following conditions. For $a \in \mathbb{Z}^d$, let $X+a=\{x+a; x \in X \}$ and $\tau_a(V_X): \mathbb{R}^{\mathscr{E}(a+X)}\to \mathbb{R}_+$ be defined as $\tau_a(V_X)((\nabla \varphi(e))_{ e\in \mathscr{E}(a+X)})= V_X((\nabla \varphi(e+a))_{ e\in \mathscr{E}(X)})$, with $e+a=(x+a,y+a)$ if $e=(x,y)$. Denoting by $C^{2,\alpha}$ the space of $C^2$-functions with $\alpha$-H\"older second derivatives, and abbreviating $\partial_x = \partial / \partial \varphi_x$, $\partial^2_{x,y} = \partial^2 / \partial \varphi_x \partial \varphi_y$  we require that there exist $c_0 \in [1,\infty)$ and $\alpha > 0$ such that, for all $X \in \mathcal{B}$,
\begin{align}
&\text{smoothness: } V_X \in C^{2,\alpha}(\mathbb{R}^{\mathscr{E}(X)}, \mathbb{R}_+), \label{eq:smooth}\\
&\text{symmetry: }V_X(\eta)= V_X(-\eta),\, \eta\in \mathbb{R}^{\mathscr{E}(X)},  \label{eq:sym} \\
&\text{translation invariance: } V_{X+a}= \tau_a(V_X), \text{ }a \in \mathbb{Z}^d,  \label{eq:trans_inv}\\ 
&\begin{array}{l}
\text{uniform convexity: } c_0^{-1} \leq - \displaystyle \sum_{X \in \mathcal{B}: X \supset \{x,y\}}  \partial^2_{x,y}V_X((\nabla \varphi(e))_{e\in \mathscr{E}(X)}) \leq c_0, \, (x,y) \in \mathscr{E}, \\
\text{and } \partial^2_{x,y} V_X((\nabla \varphi(e))_{e\in \mathscr{E}(X)}) =0, \text{ for all }x \neq y \text{ with } (x,y) \notin \mathscr{E}.
\end{array}\label{eq:elliptic}
\end{align}

\begin{remark} The potential $V_X$ allows for a joint interaction between the gradients of $\varphi$ along all edges joining $x$ to its neighbors in $\mathcal{G}$. Choosing $\Gamma = \{x; \vert x\vert_1 =1 \}$ (with $|\cdot|_1$ the usual $\ell^1$-norm on $\mathbb{Z}^d$), and defining, for $X=B(x,1)$,
\begin{equation}\label{eq:pot2body}
V_X((\nabla \varphi(e))_{e \in \mathscr{E}(X)})= \frac12 \sum_{e: x(e)=x} V(\nabla\varphi(e)),
\end{equation} for symmetric $V \in C^{2,\alpha}(\mathbb{R}, \mathbb{R}_+)$ with $c_0^{-1}\leq V'' \leq c_0$, yields a potential satisfying \eqref{eq:smooth}-\eqref{eq:elliptic}, a special case the reader may wish to focus on at first reading. In view of \eqref{eq:H_pots}, \eqref{eq:pot2body} yields the usual two-body Hamiltonian considered in the introduction, cf. \eqref{eq:intro1} and \eqref{eq:mu_Lambda} below. \hfill $\square$
\end{remark}

From \eqref{eq:elliptic}, \eqref{eq:H_pots}, one sees that for $\Lambda \subset \subset \mathbb{Z}^d$, $x \in \Lambda$,
\begin{equation}\label{eq:elliptic000}
c_0^{-1} \leq-\partial^2_{x,y}H_{\Lambda}(\varphi) \leq c_0, \text{ if $(x, y) \in \mathscr{E}_{\Lambda}$,  and $\partial^2_{x,y}H_{\Lambda}(\varphi)=0$,  if $(x, y) \notin \mathscr{E}$, $x\neq y$}. 
\end{equation}
Moreover, the gradient nature of $H_{\Lambda}$ implies that, for any $x$, $H_{\Lambda}(\varphi)= H_{\Lambda}((\varphi_y - \varphi_x)_{y})$, and thus, differentiating with respect to $\varphi_x$, that $\sum_{x} \frac{\partial H_{\Lambda}(\varphi)}{\partial \varphi_x}=0$, hence, in view of \eqref{eq:elliptic000}, for all $x \in \Lambda$,
\begin{equation}\label{eq:RW_cond}
\partial^2_{x,x} H_{\Lambda}(\varphi)= - \sum_{y: y \sim x} \partial^2_{x,y} H_{\Lambda}(\varphi). 
\end{equation}
In particular, \eqref{eq:RW_cond} yields that, for any $f: \Omega \to \mathbb{R}$ with $f(x)= 0$, for $x\notin \Lambda$,
\begin{equation}\label{eq:unifcon0001}
\langle f, \partial^2 H_{\Lambda}(\varphi) f\rangle_{\ell^2}= \frac{1}{2} \sum_{x \neq y} (-\partial^2_{x,y}H_{\Lambda}(\varphi))(f(y)-f(x))^2.
\end{equation}
We now introduce the measures of interest. Given $\xi \in \Omega$, we define $H_{\Lambda}^{\xi}: \Omega_{\Lambda}\to \mathbb{R}$ as 
\begin{equation}\label{eq:H_xi}
H_{\Lambda}^{\xi}(\varphi_{\Lambda})= H_{\Lambda}(\varphi)_{|_{\varphi=\xi \text{ on }\Lambda^c}}.
\end{equation}
We will often omit the subscript $\Lambda$ in $\varphi_{\Lambda}$, minding that $H_{\Lambda}^{\xi}$ is viewed as a function on $\Omega_{\Lambda}$.
Associated to $H_{\Lambda}^\xi$ is a (Gibbs) probability measure $\mu_{\Lambda, \beta}^{\xi}$ on $(\Omega_{\Lambda},\mathcal{F}_{\Lambda})$ at inverse temperature $\beta > 0$ and with boundary condition $\xi$, defined as 
\begin{equation}\label{eq:mu_Lambda}
\mu_{\Lambda, \beta}^{\xi}(d\varphi)= \frac{1}{ {Z}_{\Lambda, \beta}^{\xi} } e^{ - \beta H_{\Lambda}^{\xi}(\varphi)} \prod_{x \in \Lambda} d\varphi_x 
\end{equation}
where $d\varphi_x$ denotes Lebesgue measure on $\mathbb{R}$ and ${Z}_{\Lambda, \beta}^{\xi}$ is a suitable normalizing constant (that the relevant integral converges follows from uniform convexity of $H_{\Lambda}^{\xi}(\varphi)$, see \eqref{eq:unifcon0001} and \eqref{eq:elliptic000}, together with Taylor's formula, implies that 
\begin{equation}\label{eq:quadraticinfty}
\liminf_{\vert\varphi\vert_2\to \infty}|H_{\Lambda}^{\xi}(\varphi)|/ \vert\varphi \vert_2^2 > 0.
\end{equation}
 With the exception of Section \ref{S:nonconvex}, where $\beta > 0$ will be a perturbative parameter, we will set $\beta=1$ and omit it from the notation. We write $\mathbb{E}_{\mu_{\Lambda, \beta}^{\xi}}$ for expectation with respect to $\mu_{\Lambda, \beta}^{\xi}$ and denote by $\langle \cdot, \cdot \rangle_{\Lambda,\beta}^{\xi}$ the scalar product in $L^2(\mu_{\Lambda, \beta}^{\xi})$. Occasionally, we will tacitly identify $\mu_{\Lambda, \beta}^{\xi}$ with the corresponding measure $\nu_{\Lambda, \beta}^{\xi}$ on $(\Omega,\mathcal{F})$, defined such that $\mu_{\Lambda, \beta}^{\xi}(A)= \nu_{\Lambda, \beta}^{\xi}(A\times \xi_{\Lambda^c})$, for $A \in \mathcal{F}_{\Lambda}$.

We recall several useful properties of the measures $\mu_{\Lambda, \beta}^{\xi}$, which will be used throughout. From the Gibbsian nature of $\mu_{\Lambda, \beta}^{\xi}$, cf. \eqref{eq:H_pots} and \eqref{eq:mu_Lambda}, one infers that, for $\Lambda \subset \subset \mathbb{Z}^d$ and  $S' \subset \Lambda$, recalling that $\mathcal{F}_{S'}= \sigma(\varphi_x, x\in S')$, 
\begin{equation}\label{eq:CONDEXP}
\mathbb{E}_{\mu_{{\Lambda}, \beta}^{\xi}} (\cdot | \mathcal{F}_{S'} ) = \mathbb{E}_{\mu_{\Lambda \setminus S', \beta}^{\xi \vee \varphi}}  (\cdot), \quad \mu_{{\Lambda}, \beta}^{\xi}\text{-a.s.},
\end{equation}
where 
\begin{equation}\label{eq:bc_concat}
\xi \vee \varphi = \begin{cases}
\xi_x, & x \notin \Lambda\\
\varphi_x, & x \in S'.
\end{cases}
\end{equation}
The right-hand side of \eqref{eq:CONDEXP} defines a regular conditional probability for the measure $\mu_{\Lambda, \beta}^{\xi}$ conditioned on $\varphi_x$, $x \in S'$, denoted by $\mu_{\Lambda, \beta}^{\xi} (\cdot | \varphi_x, x\in S' )$. Moreover, $\mu_{\Lambda, \beta}^{\xi}$ satisfies the following domain Markov property. Namely, with $\partial_{\text{i}}^{(2)} S' = \{ x \in S'; d_{\mathcal{G}}(x,S'^c) \leq 2 \}$, using \eqref{eq:H_pots} and the fact that $V_{X} \in \mathcal{F}_X$, one has that 
\begin{equation}\label{eq:DOMAINMARKOV}
\begin{split}
&\text{the fields $(\varphi_x)_{x\in \Lambda\setminus S'}$ and $(\varphi_x)_{x\in S'\setminus \partial_{\text{i}}^{(2)} S'}$ are (conditionally) }\\
&\text{independent under $\mu_{{\Lambda}, \beta}^{\xi} (\cdot | \varphi_x, x\in   \partial_{\text{i}}^{(2)} S' )$}.
\end{split}
\end{equation}
We now review the Helffer-Sj\"ostrand representation for the covariances under the measure $\mu_{\Lambda, \beta}^{\xi}$. One naturally associates to $\mu_{\Lambda, \beta}^{\xi}(d\varphi)$ the second order differential operator
\begin{equation}\label{eq:hs1}
L_{\Lambda}^{\xi}\stackrel{\text{def.}}{=}e^{  \beta H_{\Lambda}^{\xi}(\varphi)}\sum_{x \in \Lambda}\frac{\partial}{\partial \varphi_x}\Big[ e^{ - \beta H_{\Lambda}^{\xi}(\varphi)} \frac{\partial}{\partial \varphi_x} \Big] \ 
\end{equation}
with domain
\begin{equation}\label{eq:dom_L}
\text{Dom}(L_{\Lambda}^{\xi})= \{ f \in C^2(\Omega_{\Lambda}, \mathbb{R}); \sup_{\varphi\in \Omega_{\Lambda}} |\partial f(x,\varphi)|e^{-\varepsilon \sum_x |\varphi_x|}< \infty, \text{ for some }\varepsilon > 0\},
\end{equation}
 with the notation $\partial f(x,\varphi) = \partial f(\varphi)/ \partial \varphi_x$, for $x \in \Lambda$. We also set the $\partial f(x,\varphi)= 0$, for $f \in C^1(\Omega_{\Lambda}, \mathbb{R})$, $x\notin \Lambda$, which will be used throughout the paper. The point of the definition \eqref{eq:hs1} is that $L_{\Lambda}^{\xi}$ is symmetric with respect to $\mu_{\Lambda, \beta}^{\xi}$, i.e., 
\begin{equation}\label{eq:Lsym}
\langle f, -L_{\Lambda}^{\xi}g\rangle_{\Lambda,\beta}^{\xi}= \sum_{x \in \Lambda} \langle \partial f(x,\varphi), \partial g(x, \varphi)\rangle_{\Lambda,\beta}^{\xi} = \langle -L_{\Lambda}^{\xi} f, g\rangle_{\Lambda,\beta}^{\xi}, \text{ for } f,g \in \text{Dom}(L_{\Lambda}^{\xi}), 
\end{equation}
which follows from integration by parts and \eqref{eq:dom_L}, \eqref{eq:quadraticinfty}. Rewriting $L_{\Lambda}^{\xi}=  \sum_{x \in \Lambda} [\frac{\partial^2}{\partial \varphi_x^2} - \frac{\partial H_{\Lambda}^{\xi}(\varphi)}{\partial \varphi_x} \frac{\partial}{\partial \varphi_x}]$, cf. \eqref{eq:hs1}, one may view $L_{\Lambda}^{\xi}$ as the generator of the finite-dimensional diffusion process described by the stochastic differential equation
\begin{equation}\label{eq:Langevin}
\begin{cases}
d\Phi_t(x) = - \partial H_{\Lambda}^{\xi}(\Phi_t,x) dt + \sqrt{2} dB_t(x), & x \in \Lambda\\
\Phi_t(x)= \xi_x, & x \notin \Lambda,
\end{cases}
\end{equation}
where $ (B_t(y))_{t \geq 0}$, $y \in \Lambda$, are independent Brownian motion. Due to assumption \eqref{eq:elliptic}, the drift coefficients in \eqref{eq:Langevin} satisfy the growth assumptions of classical SDE theory (see for instance \cite{KS91}, Ch.5, Thm. 2.9), which yields that \eqref{eq:Langevin} has a (unique) strong solution, for given initial condition. We denote by $P^{\Lambda, \xi}_\varphi$ the canonical law on $C^0(\mathbb{R}_+, \Omega \cap \{ \varphi_{\vert_{\Lambda^c}}= \xi  \} )$ of the solution to \eqref{eq:Langevin} with starting point $\Phi_0 = \varphi$, and by ${E}^{\Lambda, \xi}_\varphi$ the corresponding expectation.

Let $F,G \in \text{Dom}(L_{\Lambda}^{\xi}) (\subset L^2(\mu_{\Lambda, \beta}^{\xi}))$. We wish to compute $\textnormal{Cov}_{\mu_{\Lambda, \beta}^{\xi}}(F,G)= \mathbb{E}_{\mu_{\Lambda, \beta}^{\xi}}[F\widetilde{G}]$, with $\widetilde{G}= G- \mathbb{E}_{\mu_{\Lambda, \beta}^{\xi}}[{G}] $. Proceeding as in the proof of Lemma 2.1 in \cite{DGI00}, but under the general assumption \eqref{eq:smooth}, \eqref{eq:elliptic} above replacing (2.2), (2.3) in \cite{DGI00} (the precise form of the Hamiltonian $H_{\Lambda}$ is irrelevant in the proof), one deduces that the equation 
\begin{equation}\label{eq:poisson1}
-L_{\Lambda}^{\xi}U = \widetilde{G}
\end{equation}
has a unique classical solution $U\in C^{3,\alpha} \cap  \overline{\text{Dom}}(L_{\Lambda}^{\xi})$ satisfying $ \mathbb{E}_{\mu_{\Lambda, \beta}^{\xi}}[U]=0$, where $\overline{\text{Dom}}(L_{\Lambda}^{\xi})$ is the domain of the closed self-adjoint extension of $-L_{\Lambda}^{\xi}$ to $ L^2(\mu_{\Lambda, \beta}^{\xi})$. Thus, 
\begin{equation}\label{eq:HScov0}
\textnormal{Cov}_{\mu_{\Lambda, \beta}^{\xi}}(F,G)= \mathbb{E}_{\mu_{\Lambda, \beta}^{\xi}}[F ( -L_{\Lambda}^{\xi}U)]= \sum_{x \in \Lambda} \mathbb{E}_{\mu_{\Lambda, \beta}^{\xi}}[\partial F(x,\varphi) \partial U(x, \varphi)],
\end{equation}
where the last step follows from \eqref{eq:hs1} and integration by parts. One seeks a probabilistic representation for $ \partial U(x, \varphi)$. Note that, defining 
\begin{equation}\label{eq:thea}
a_{x,y}^{\xi}(\varphi) = -[\partial_x \partial _y H_{\Lambda}(\varphi) ]_{|_{\varphi=\xi \text{ on }\Lambda^c}}, \text{ for }(x,y) \in \mathscr{E}_{\Lambda},
\end{equation}
and with the convention $ \partial U(x, \varphi) =0$, $x \notin \Lambda$, one has, for $x \in \Lambda$,
\begin{equation}\label{eq:comm}
\begin{split}
&\partial_x (-L_{\Lambda}^{\xi}) U = - \partial_x \sum_{y \in \Lambda} [\partial^2_y U- (\partial _y H_{\Lambda}^{\xi})(\partial_yU)] \\
&= -L_{\Lambda}^{\xi}(\partial_x U) + \sum_{y \in \Lambda} (\partial_x \partial _y H_{\Lambda}^{\xi})(\partial_yU) \stackrel{\eqref{eq:RW_cond}}{=}  -L_{\Lambda}^{\xi}(\partial_x U) - \sum_{y:(x,y)\in \mathscr{E}_{\Lambda}} a_{x,y}^{\xi}(\varphi) (\partial_yU - \partial_xU) .
\end{split}
\end{equation}
Thus, introducing the operator $Q_{\Lambda}^{\xi, \varphi}$, acting on functions $f:\Lambda \cup \partial\Lambda \to \mathbb{R}$ as
\begin{equation}\label{eq:defQ}
(Q_{\Lambda}^{\xi, \varphi}f)(x)=    \sum_{y: (x,y)\in \mathscr{E}_{\Lambda}}  a_{x,y}^{\xi}(\varphi)(f(y)-f(x)), 
\end{equation}
one obtains from \eqref{eq:poisson1} and \eqref{eq:comm} that $u(x,\varphi) =  \partial U(x, \varphi)$ solves the boundary value problem  
\begin{equation}\label{eq:BVP}
\begin{cases}
(-\mathscr{L}_{\Lambda}^{\xi})u(x, \varphi)= \partial\widetilde{G}(x,\varphi), &  x \in  \Lambda, \varphi \in\Omega_{\Lambda},\\
u(x,\varphi) =0, &  x \in  \Lambda^c, \varphi \in\Omega_{\Lambda},
\end{cases} 
\end{equation}
where we have defined
\begin{equation}\label{eq:L}
\mathscr{L}_{\Lambda}^{\xi} f (x, \varphi) = ({L}_{\Lambda}^{\xi}f(x, \cdot))(\varphi) + (Q_{\Lambda}^{\xi, \varphi}f(\cdot, \varphi))(x), \quad x \in  \Lambda, \varphi \in\Omega_{\Lambda}
\end{equation}
with domain
\begin{equation}\label{eq:L_dom}
\text{Dom}(\mathscr{L}_{\Lambda}^{\xi} )= \{ f: \mathbb{Z}^d \times \Omega_{\Lambda} \to \mathbb{R}; \, f(x, \cdot) \in \text{Dom}({L}_{\Lambda}^{\xi}), x \in \Lambda,\, f(y,\cdot)= 0 , y \in \Lambda^c\}.
\end{equation}
The operator $Q_{\Lambda}^{\xi, \varphi}$ is the generator of a pure jump process on $\mathcal{G}_{\Lambda}$, since $c_0^{-1} \leq a_{x,y}^{\xi}(\varphi)= a_{x,y}^{\xi}(\varphi) \leq c_0$ for all $(x,y)\in \mathscr{E}_{\Lambda}$, as follows from \eqref{eq:defQ}, \eqref{eq:thea} and \eqref{eq:elliptic000}. Hence, using the symmetry of $a_{x,y}^{\xi}(\varphi)$, one deduces from, \eqref{eq:L}, \eqref{eq:Lsym} and \eqref{eq:defQ} that for all $f, g \in \text{Dom}(\mathscr{L}_{\Lambda}^{\xi} )$,
\begin{equation}\label{eq:Dirichletform}
\begin{split}
\mathcal{E}(f,g)& \stackrel{\text{def.}}{=}\sum_{x \in  \Lambda} \langle f(x,\cdot), ( - \mathscr{L}_{\Lambda}^{\xi} g) (x, \cdot) \rangle_{\Lambda,\beta}^{\xi} \\
&= \sum_{x \in  \Lambda} \langle \partial f(x,\varphi), \partial g(x, \varphi)\rangle_{\Lambda,\beta}^{\xi} + \frac12 \sum_{ e\in \mathscr{E}_{\Lambda}}a_{x(e),y(e)}^{\xi}(\varphi)\nabla f(e,\varphi) \nabla g(e,\varphi)
\end{split}
\end{equation}
where $\nabla f(e,\varphi) = f(y, \varphi)-f(x,\varphi)$ if $e =(x,y)$, for $f \in \text{Dom}(\mathscr{L}_{\Lambda}^{\xi} )$, induces a Dirichlet form on the Hilbert space $L^2(\kappa_{\Lambda} \otimes \mu_{\Lambda, \beta}^{\xi})$, with $\kappa_{\Lambda}$ the counting measure on $\Lambda$. Thus, see for instance \cite{Fu11}, the closure of $\mathscr{L}_{\Lambda}^{\xi}$ in $L^2(\kappa_{\Lambda} \otimes \mu_{\Lambda, \beta}^{\xi})$ is the generator of a symmetric Markov process $(X_t, \Phi_t)_{ t\geq0}$ on $\overline{\Lambda} \times \Omega_{\Lambda}^\xi$, where $ \Omega_{\Lambda}^\xi = \Omega \cap \{ \varphi_{\vert_{\Lambda^c}}= \xi  \}$. We denote by $\mathbf{P}_{x,\varphi}^{\mathcal{G}_{\Lambda},\xi}$ the canonical law of this process with starting point $(X_0, \Phi_0) = (x, \varphi) \in  {\Lambda} \times \Omega_{\Lambda}^\xi$ on $D(\mathbb{R}_+, \overline{\Lambda})\times C^0(\mathbb{R}_+,  \Omega_{\Lambda}^\xi )$, endowed with its canonical $\sigma$-algebra, where $D(\mathbb{R}_+, \overline{\Lambda})$ denotes the space of right-continuous trajectories on $\overline{\Lambda}$. Because of the specific form of the generator $\mathscr{L}_{\Lambda}^{\xi} $, cf. \eqref{eq:L}, one can construct the process $(X_t, \Phi_t)_{t \geq 0}$ by first generating the ``environment'' $(\Phi_t)_{t}$, solution of the Langevin equation \eqref{eq:Langevin}, and then constructing the jump process $(X_t)_t$ with time-dependent transition rates 
\begin{equation}\label{eq:RATES}
a_{\Lambda}(\Phi)= \{ a_{x,y}^{\xi}(\Phi_t);  t\geq 0, \, (x,y)\in \mathscr{E}_{\Lambda}\},
\end{equation}
with $a_{x,y}^{\xi}(\cdot)$ as defined in \eqref{eq:thea}. We use
\begin{equation}\label{eq:stoppingtimes}
H_K =  \inf\{ t > 0 ; X_t \in K\}, \, \tau_K =H_{K^c}, \text{ for }K \subset \mathbb{Z}^d,
\end{equation}
to denote the entrance time in $K$, resp. exit time from $K$. Having identified $\mathscr{L}_{\Lambda}^{\xi}$ as a generator, returning to the boundary value problem \eqref{eq:BVP}, we deduce that $u$ admits the following probabilistic representation. Recall the convention $\partial f(y, \varphi)=0$ for $y \notin \Lambda$ and 
$f:\mathbb{Z}^d \times \Omega_{\Lambda}\to \mathbb{R}$.

\begin{lem}[Helffer-Sj\"ostrand representation formula] $\quad$ \label{L:HS}

\medskip
\noindent The solution $u: \mathbb{Z}^d \times \Omega_{\Lambda}\to \mathbb{R}$ of \eqref{eq:BVP} can be expressed as
\begin{equation}\label{eq:HS1}
u(x,\varphi)= \mathbf{E}_{x,\varphi}^{\mathcal{G}_{\Lambda},\xi}\Big[ \int_0^{\tau_{\Lambda}} \partial\widetilde{G}(X_t,\Phi_t) dt \Big], \quad x \in \overline{\Lambda}, \varphi \in \Omega_{\Lambda}.
\end{equation}
Moreover, for any $F, G \in \textnormal{Dom}(L_V^\xi)$, one has
\begin{equation}\label{eq:HS2}
\textnormal{Cov}_{\mu_{\Lambda, \beta}^{\xi}}(F,G)= \sum_{x \in \Lambda} \int_0^\infty dt E_{x, \mu_{\Lambda, \beta}^{\xi}}[\partial F(x,\Phi_0) \partial G(X_t, \Phi_t) 1\{ t < \tau_{\Lambda}\}  ],
\end{equation}
where $P_{x, \mu_{\Lambda, \beta}^{\xi}} [\cdot]= \int \mu_{\Lambda,\beta}^{\xi}(d\varphi) \mathbf{P}_{x,\varphi}^{\mathcal{G}_{\Lambda},\xi}[\cdot]$ is the law of $(X_t, \Phi_t)_t$ with initial distribution $(X_0, \Phi_0) \sim \delta_x \otimes  \mu_{\Lambda,\beta}^{\xi}$.
\end{lem}

\begin{proof}
The representation \eqref{eq:HS1} follows from an application of the optional stopping theorem, mimicking the proof of Prop. 2.2. in \cite{DGI00}, and \eqref{eq:HS2} is an immediate consequence of \eqref{eq:HS1}, \eqref{eq:HScov0}, recalling that $G-\widetilde{G}= \mathbb{E}_{\mu_{\Lambda, \beta}^{\xi}}[{G}]$, and applying Fubini's theorem.
\end{proof}
An immediate consequence of Lemma \ref{L:HS} is a comparison estimate for moments of $\varphi$ under $\mu_{\Lambda, \beta}^{\xi}$ with moments of the corresponding Gaussian free field on $\mathcal{G}$. Let $\mu_{\Lambda, \beta}^{\xi,\ast}$ be obtained from \eqref{eq:mu_Lambda} by setting $H_{\Lambda}(\varphi)= \frac{1}{2c_0}\sum_{e \in \mathscr{E}_{\Lambda}} (\nabla \varphi(e))^2$ in \eqref{eq:H_xi}, \eqref{eq:mu_Lambda}, with $c_0$ from \eqref{eq:elliptic}. 
\begin{lem}[Brascamp-Lieb inequality for exponential moments] \label{L:BL}$\quad$

\medskip
\noindent For all $\Lambda \subset \subset \mathbb{Z}^d$, $\nu \in \Omega_{\Lambda}$, setting $\hat{\varphi} =  \varphi - \mathbb{E}_{\mu_{\Lambda, \beta}^{\xi}}[\varphi]$, one has 
\begin{equation}\label{eq:BL}
\mathbb{E}_{\mu_{\Lambda, \beta}^{\xi}}\Big[e^{\langle \nu, \hat{\varphi}\rangle_{\ell^2(\Lambda)}}\Big] \leq e^{\frac12 \textnormal{var}_{\mu_{\Lambda, \beta}^{\xi,\ast}}(\langle \nu, \varphi \rangle_{\ell^2(\Lambda)})}
\end{equation}
\end{lem}

\begin{proof}
On account of \eqref{eq:HS2}, the proof of \eqref{eq:BL} is analogous to that of $(4.13)$ in \cite{Fu05}. We omit the details. 
\end{proof}

\begin{remark}
The Gibbsian specification of $H_{\Lambda}$ in terms of $\{ V_X \}_X$ will enter when discussing the limit $\Lambda \nearrow \mathbb{Z}^d$ below. But Lemma \ref{L:HS} continues to hold if one replaces \eqref{eq:H_pots} and conditions \eqref{eq:smooth}-\eqref{eq:elliptic} by requiring that $H_{\Lambda}: \Omega \to \mathbb{R}$ be measurable with respect to $\mathcal{F}_{\Lambda \cup B}$, where $B \supset \partial \Lambda$ is a finite set, and satisfy the following requirements: $H_{\Lambda} \in C^{2,\alpha}$, for some $\alpha > 0$ (smoothness), $H_{\Lambda}(\varphi)= H_{\Lambda}(\varphi + C)$, for all $C \in \mathbb{R}$ (gradient structure), 
where, with hopefully obvious notation, $\varphi + C$ is obtained from $\varphi$ by shifting all coordinates by $C$
and $c_0^{-1} \leq-\partial^2_{x,y}H_{\Lambda}(\varphi) \leq c_0$, if $(x, y) \in \mathcal{G}_{\Lambda}$, $\partial^2_{x,y}H_{\Lambda}(\varphi) = 0$ if $ (x,y) \notin \mathscr{E}$, $x \neq y$  (convexity). \hfill $\square$
\end{remark}

We now introduce a suitable class of infinite-volume measures (or equilibrium states, in the jargon of statistical mechanics). We will consider weak limits of measures $\mu_{{\Lambda}, \beta}^{\xi}$ as $\mathcal{G}_{\Lambda} \nearrow \mathcal{G} = (\mathbb{Z}^d, \mathscr{E})$.
Henceforth, in writing $(\Lambda_n)_{n \geq 0} \nearrow \mathbb{Z}^d$, we always refer to an increasing sequence of finite subsets of $\mathbb{Z}^d$ whose union is $\mathbb{Z}^d$.
Let $\mathcal{M}^1(\Omega)$ be the set of probability measures on $\Omega$, and define $\mathscr{W}_{\beta}=\mathscr{W}_{\beta}(\{V_X\}_X)$ as
\begin{equation}\label{eq:W}
\begin{split}
\mathscr{W}_{\beta} = \{ 
&\mu  \in \mathcal{M}^1(\Omega); \, 
\text{there exists $ (\Lambda_n)_{n \geq 0} \nearrow \mathbb{Z}^d$,}\\
&\text{with $\mathcal{G}_{\Lambda_n}$ connected, for all $n \geq 0$ s.t. } \mu_{\Lambda_n, \beta}^{0} \stackrel{w}{\rightarrow} \mu\}.
\end{split}
\end{equation}
To see that $\mathscr{W}_{\beta} \neq \emptyset$, note that, denoting by $g^{*}(x,y)$ the Green function of (continuous-time, with exponential holding times of parameter $1$) simple random walk on $\mathcal{G}= (\mathbb{Z}^d, \mathscr{E})$, see \eqref{eq:GFstar} below, and $g^{*}_{{\Lambda}}(x,y)$ that of simple random walk on $\mathcal{G}_{\Lambda}$ killed upon hitting $\partial \Lambda$, one has for instance, using \eqref{eq:BL}, minding that $\mathbb{E}_{\mu_{\mathcal{G}_{\Lambda}, \beta}^{0}}[\varphi_x] = 0$, for all $x$, due to \eqref{eq:sym}, that
\begin{equation} \label{eq:tight}
\sup_{\Lambda} \sup_{x} \mathbb{E}_{\mu_{{\Lambda}, \beta}^{0}}[e^{\varphi_x}] \leq e^{\frac c2 g^{*}_{{\Lambda}}(x,x)} \leq e^{\frac c2 g^{*}(0,0)} < \infty
\end{equation}
(since $d\geq 3$, and $\mathcal{G}$ has bounded geometry), from which one easily deduces that the family $\{ \mu_{{\Lambda}, \beta}^{0}; \, \Lambda \subset \subset \mathbb{Z}^d \}$ (tacitly viewed as measures on $(\Omega, \mathcal{F})$) is tight.

We conclude this section with some elements of potential theory for simple random walk on $\mathcal{G}$. We denote by $P^*_x$ the law of continuous-time simple random walk on $\mathcal{G}$, started at $x$, write $(Z_t)_{t\geq 0}$ for the corresponding canonical process and 
\begin{equation}\label{eq:GFstar}
g^*(x,y)=E_x^\ast \Big[ \int_0^\infty dt 1\{ Z_t =y\}\Big], \, x,y \in \mathbb{Z}^d,
\end{equation} 
for its Green function. For $U \subset \subset \mathbb{Z}^d$, we define the equilibrium measure of $U$ as 
\begin{equation}\label{eq:equil_meas}
e^*_U(y)= P_y^*[\widetilde{H}_U = \infty]1_{y \in U},
\end{equation}
with $\widetilde{H}_U= \inf\{ t >0; Z_t \in U \text{ and $ \exists \, s\in (0,t)$ s.t. $Z_s \neq Z_0$} \}$ the hitting time of $U$, and the capacity of $U$
\begin{equation}\label{eq:cap_def0}
 \text{cap}^*(U) 
 =\sum_{y \in U} e^*_U(y),
\end{equation}
which satisfies the variational principle (see for instance \cite{Sz12b}, Prop.1.9)
\begin{equation}\label{eq:cap2}
 \text{cap}^\ast(U) = \frac{1}{\inf \{ E^*(\nu); \nu \text{ a prob. meas., } \text{supp}(\nu) \subset U  \}}
 \end{equation}
where
\begin{equation}\label{eq:en1}
\text{ where } E^*(\nu)= \sum_{x,y}\nu(x)g^*(x,y) \nu(y)= \langle \nu, G^* \nu \rangle_{\ell^2(\mathcal{G})},
\end{equation}
with $G^*\nu(x)= \sum_y G^*(x,y)\nu (y)$ the energy associated to the measure $\nu$. 
The unique minimizer in \eqref{eq:cap2} is $\nu  = \bar{e}_U^* = e_U^*/ \text{cap}^*(U)$, i.e.
\begin{equation}\label{eq:cap3}
\text{cap}^*(U) = E^*(\bar{e}_U^*)^{-1}.
\end{equation}

\section{Sprinkling} \label{S:sprinkle}

We proceed to the first and main result of this paper, which is a decoupling inequality for measures $\mu \in \mathscr{W}_{\beta}$, cf. \eqref{eq:W}, stated below in Theorem \ref{P:dec_ineq}. This result is proved over the next two sections. The main issue, as explained in the introduction, see also Remark \ref{R:idea_dec} below, consists of finding a suitable way to ``sprinkle the field,'' and is presented in this section, cf. Proposition \ref{L:monot}.

We begin by introducing a key quantity that will eventually control the entire sprinkling technique. Recall the definition of $\mathbf{P}_{x,\varphi}^{\mathcal{G}_{\Lambda},\xi}$ above \eqref{eq:RATES}. Given $\Lambda \subset \subset \mathbb{Z}^d$ and a target set $S' \subset \subset \Lambda$, setting $\Lambda' = \Lambda \setminus S'$, 
we define the (probabilistic) \textit{cross-section of $S'$ with respect to $x$ inside $\Lambda$} as
\begin{equation}\label{eq:invisibility_cond}
\begin{split}
&\Sigma_\Lambda(x,S') = \sup_{\xi \in \Omega } \sup_{\varphi \in \Omega_{\Lambda'}^{\xi}} \frac{\mathbf{P}_{x,\varphi}^{\mathcal{G}_{\Lambda'},\xi}[H_{S'} < H_{\Lambda^c} ]}{1- \mathbf{P}_{x,\varphi}^{\mathcal{G}_{\Lambda'},\xi}[H_{S'} < H_{\Lambda^c} ]}, \text{ for } x\in \Lambda, \text{ and } \\
&\Sigma_\Lambda(S,S')= \sup_{x \in S} \Sigma_\Lambda(x,S'), \text{ for $S \subset \Lambda$}
\end{split}
\end{equation}
(recall that $H_K$, $K \subset \mathbb{Z}^d$, refers to the entrance time in $K$ for the random walk $X_{\cdot}$, cf. \eqref{eq:stoppingtimes}). The killing outside $\Lambda$ will soon be removed by letting $\Lambda \nearrow \mathbb{Z}^d$, and the resulting quantity $\Sigma(S,S')$, see \eqref{eq:def_sigma_inf} below, will later play a pivotal role. Note that $\Sigma_\Lambda(S,S')< \infty$ whenever $S\cap S' =\emptyset$.

Henceforth, in writing $\mu \in \mathscr{W}$, we mean that $\mu \in \mathscr{W}_{\beta=1}(\{V_X\}_X)$, cf. \eqref{eq:W}, for some family $\{ V_X\}_X$ of potentials satisfying \eqref{eq:smooth} - \eqref{eq:elliptic}. The parameter $\beta=1$ will be omitted from all notation. Finally, if $A^h \subset \Omega$ satisfies $A^h \in  \sigma(1\{ \varphi_x \geq h\} ;  x\in S)$, $h \in \mathbb{R}$, then there exists a unique event $A\subset \widehat{\Omega}_S$ such that $A^h =\{ (1\{ \varphi_x \geq h\})_{x\in S} \in A \}$, for all $h \in \mathbb{R}$. Conversely, if $A \subset \widehat{\Omega}_S $ is given, we define $A^h = \{ (1\{ \varphi_x \geq h\})_{x\in S} \in A \}$, $h \in \mathbb{R}$.

\begin{thm} \label{P:dec_ineq} $(\varepsilon > 0, \, h \in \mathbb{R}, \, \mu \in \mathscr{W})$

\medskip
\noindent 
Let $(\Lambda_n)_{n \geq 0}$ be an increasing sequence of finite subsets of $\mathbb{Z}^d$ such that $\mu_{{\Lambda_n}}^0 \stackrel{w}{\rightarrow} \mu$, and suppose $ \emptyset \neq S, S' \subset \subset \mathbb{Z}^d$ are disjoint. Then, with
\begin{equation}\label{eq:def_sigma_inf}
\Sigma =\Sigma(S,S', (\Lambda_n)_{n\geq 0}) \stackrel{\textnormal{def.}}{=} \liminf_{n \to \infty} \Sigma_{\Lambda_n}(S,S'),
\end{equation}
one has, for all increasing $A^h \in \sigma(1\{ \varphi_x \geq h\} ;  x\in S)$ and all bounded continuous functions $f:\Omega \to [0,\infty)$ satisfying $f \in\sigma(\varphi_x; x \in S')$,
\begin{equation}\label{eq:dec_ineq}
\mathbb{E}_{\mu}[1_{A^h}\cdot f]\leq \mu(A^{h- \varepsilon})\cdot \mathbb{E}_{\mu} [f] + \delta_{S, S'}( \varepsilon, \Sigma) \cdot \Vert {f}\Vert_{L^{\infty}(\Omega)},
\end{equation} 
where
\begin{equation}\label{eq:eta_error}
\delta_{S, S'}( \varepsilon, \Sigma) =
c_1 |S| \exp \bigg\{ | \partial_{\textnormal{i}} S'| -c_2  \bigg[ \frac{ \textnormal{cap}^*(S')}{\sigma^*(S,S')^2 |\partial_{\textnormal{i}}S'| }\cdot \frac{\varepsilon}{\Sigma}\,\bigg]^2\bigg\},
\end{equation}
and
\begin{equation}\label{eq:sigma_ref*}
\sigma^*(S,S')= \sup_{x\in S}  \frac{\sup_{y \in S'} g^*(x,y)}{\inf_{y \in S'} g^*(x,y)}.
\end{equation}
\end{thm}
 
\begin{remark}\label{R:dec_ineq1} 1) (Sharpness of \eqref{eq:dec_ineq}). In case $f$ is increasing in $\varphi$, a companion lower bound to  \eqref{eq:dec_ineq} can be obtained from the FKG-inequality, which is an immediate consequence of \eqref{eq:HS2}, yielding $\mathbb{E}_{\mu}[1_{A^h}\cdot f] \geq  \mu(A^{h})\cdot \mathbb{E}_{\mu} [f]$ (note that $\mu(A^{h}) \leq \mu(A^{h-\varepsilon})$).\\
2) (Error term). The size of $\delta_{S, S'}( \varepsilon, \Sigma)$, cf. \eqref{eq:eta_error}, hinges on a careful balance between the sprinkling parameter $\varepsilon > 0$, the geometry of $S$, $S'$, and their relative \text{size}, as measured by the cross section $\Sigma$, which enters as an ``invisibility'' condition for the walk $\mathbf{P}_{x,\varphi}^{\mathcal{G}_{\Lambda'}}$, see \eqref{eq:invisibility_cond}: the smaller $\Sigma$ is, the better the error term $\delta_{S, S'}$. This is reminiscent of the Gaussian case, cf. the discussion following \eqref{eq:harmonic_ext}, and also \cite{RoS13}, (2.30) and (2.31), where one could roughly afford to choose $\varepsilon > 0$ as
$$
\varepsilon \approx \sup_{x\in S} E_x^*[\varphi_{Z_{H_{S'}}} 1\{ H_{S'} < \infty\}].
$$
3) A somewhat different error term, more desirable for certain applications, is derived below in Theorem \ref{T:dec_ineq2}. Its proof \textit{relies} however on Theorem \ref{P:dec_ineq}.\\
4) An analogue of \eqref{eq:dec_ineq} holds if $A^h$ is replaced by a \textit{decreasing} event $B^h$ satisfying the same measurability assumptions as $A^h$. In that case, defining the `flipped' event $\overline{B}= \{ \omega^f; \, \omega \in B \}$, where $\omega^f$ is obtained from $\omega \in \{ 0,1\}^S$ by changing all the entries, noting that $\mu(B^h)=\mu(\overline{B}^{-h})$, which follows because $\varphi$ and $-\varphi$ have the same law under $\mu$, cf. \eqref{eq:sym}, and applying Theorem \ref{P:dec_ineq} to the increasing event $\overline{B}^{-h}$, one obtains 
\begin{equation}\label{eq:dec_ineq_bis}
\mathbb{E}_{\mu}[1_{B^h}\cdot f]\leq \mu(B^{h+ \varepsilon})\cdot \mathbb{E}_{\mu} [f] + \delta_{S, S'}( \varepsilon, \Sigma) \cdot \Vert {f}\Vert_{L^{\infty}(\Omega)}.
\end{equation} 
\end{remark}

The proof of Theorem \ref{P:dec_ineq} hinges on two key results, Proposition \ref{L:monot} and Lemma \ref{L:error}. The former, of independent interest, contains the gyst of the decoupling argument for conditional distributions of gradient measures. Its proof is the main object of this section, and makes crucial use of the Helffer Sj\"ostrand representation. The decoupling gives rise to a certain error term, which eventually leads to \eqref{eq:eta_error} above, and needs to be controlled. This is done in Section \ref{S:error}. Together, the two results will readily imply Theorem \ref{P:dec_ineq}. Its proof is found at the end of Section \ref{S:error}.

We now investigate certain conditional distributions of $\mu \in \mathscr{W}$. Throughout the remainder of this section, we assume that $S,S' ,\Lambda \subset \subset \mathbb{Z}^d$ satisfy
\begin{equation}\label{eq:dec:A_sets}
\begin{split}
&S \cup {S'}  \subset \Lambda, \, S\cap S' =\emptyset, \text { $A \in \sigma(Y_x; x \in S)$ is increasing, and $h \in \mathbb{R}$}.
\end{split}
\end{equation}
Consider
\begin{equation} \label{eq:dec10}
\begin{split}
Z^h(\varphi_{S'}) \stackrel{\text{def.}}{=} 
 \mu_{\Lambda \setminus S'}^{0 \vee \varphi}  (A^h)\end{split},
\end{equation}
where $0 \vee \varphi$ specifies the boundary condition for $\Lambda' = \Lambda \setminus S'$, cf. \eqref{eq:bc_concat}, which vanishes on $\Lambda^c$ and equals $\varphi$ on $S'$. 
As explained around \eqref{eq:CONDEXP}, $Z^h(\varphi_{S'})$ represents a choice of regular conditional distribution, i.e.
\begin{equation} \label{eq:dec10.1}
\begin{split}
 Z^h(\varphi_{S'}) = \mathbb{E}_{\mu_{\Lambda}^{0}}(A^h | \,  \varphi_x, \, x \in S').  
 \end{split}
\end{equation}
Assume henceforth that $\omega=(\varphi,\widetilde{\varphi}) \in \Omega\times \Omega$ is distributed under $Q_{\Lambda}^{0} (= \mu_{\Lambda}^{0} \otimes \widetilde{\mu}_{\Lambda}^{0})$ as two independent copies of the field $\varphi$ under $\mu_{\Lambda}^{0}$. Roughly speaking, we aim to show that for a suitably defined (good) event $G = G(\varepsilon, S,S',\Lambda) \subset \Omega \times \Omega $, one has 
$
\{ Z^h(\varphi_{S'})  >  Z^{h-\varepsilon}(\widetilde{\varphi}_{S'}) \}  \subset G^c, 
$
 in such a way that $G$ has high probability whenever $S'$ is 
 sufficiently invisible, i.e. whenever the cross-section $\Sigma_{\Lambda}(S,S')$ is sufficiently small. This is essentially the content of Proposition \ref{L:monot} below. The control for the probability of $G^c$ in terms of $\Sigma$ is deferred to the next section.
 
 The proof involves an interpolation argument relying on the Helffer-Sj\"ostrand representation, which we set up next. For $t\in [0,1]$, we define the (random) interpolated boundary condition $\xi^\omega(t)=( \xi^{\omega}_x(t))_{x}$, for $\omega= (\varphi, \widetilde{\varphi})$, $x \in (\Lambda')^c =   \Lambda^c \cup S'$, and $\varepsilon >0$, as
\begin{equation}\label{eq:dynamic_bc}
 \xi^{\omega}_x(t)=
\begin{cases} 
t\varepsilon, & x \in  \Lambda^c \\
(1-t)\varphi_x + t(\widetilde{\varphi}_x+\varepsilon), & x \in S',
\end{cases}
\end{equation}
noting that $\xi^\omega(0)= 0 \vee \varphi$ and $\xi^\omega(1)= (0 \vee \widetilde{\varphi}) + \varepsilon$ (where, with hopefully obvious notation, one adds $\varepsilon$ to each component of $\xi^\omega(1)$), and consider the function, for $\omega = (\varphi, \widetilde{\varphi})$
\begin{equation}\label{eq:def_F}
\mathscr{F}_{\Lambda}^{\omega}: [0,1] \to \mathbb{R}, \ t \mapsto \mu_{{\Lambda\setminus S'}}^{\xi^{\omega}(t)}(A^{h}).
\end{equation}
To keep the notation light, we will often omit the superscript $\omega$ from $\xi^{\omega}(t)$.
\begin{remark}\label{R:idea_dec} It is plain to see that $\mathscr{F}_{\Lambda}^{\omega} \in C^1[0,1]$ and we will soon show, cf. \eqref{eq:F_monot}, that it is in fact increasing for ``most'' $\omega$ (in a sense to be made precise). This is not at all evident since the possible decrease in value of the boundary condition along $S'$ in \eqref{eq:dynamic_bc} acts \textit{against} the desired monotonicity ($A^{h}$ is \textit{increasing} in $ \varphi$). The idea is that this can typically, i.e. for ``most'' $\omega$, be controlled by the following \textit{balancing} mechanism: since the boundary condition $\xi^{\omega}(\cdot)$ only decreases along a sufficiently \textit{small} portion of the boundary, $S'$, as will be quantified in terms of the cross-section $\Sigma$, the corresponding error can be absorbed by a slight global upward push of the field, parametrized by $\varepsilon$. \hfill $\square$
\end{remark}

Recall the definition of $\mathbf{P}_{x,\varphi}^{\mathcal{G}_{\Lambda},\xi}$, cf below \eqref{eq:Dirichletform}. For $M>0$, let
\begin{equation}\label{eq:dec15}
\begin{split}
&G_{\Lambda,M}^c \equiv G_{\Lambda,M,S,S'}^c(\omega)= \bigcup_{t\in \mathbb{Q} \cap [0,1]} \bigcup_{(x,\varphi') \in S \times \mathbb{Q}^{\Lambda'}} \Big\{  \mathbf{E}_{x, \varphi'}^{\mathcal{G}_{\Lambda'}, \xi^{\omega}(t)} \big[{\varphi}_{X_{\tau_{\Lambda'}}} - \widetilde{\varphi}_{X_{\tau_{\Lambda'}}}  \big|X_{\tau_{\Lambda'}} \in \partial_{\text{i}}S'\big]  > M    \Big\}, 
\end{split}
\end{equation}
(part of $\Omega \times \Omega$), where $\tau_{\Lambda'}= H_{(\Lambda')^c}$, cf. \eqref{eq:stoppingtimes}, and
\begin{equation}\label{eq:decM}
M_{\varepsilon, \Sigma_{\Lambda}} =  \varepsilon( \Sigma_{\Lambda}^{-1}(S,S')+1), \quad \text{cf. \eqref{eq:invisibility_cond} for the definition of $\Sigma_{\Lambda}$}.
\end{equation}
Note that $M_{\varepsilon, \Sigma_{\Lambda}} > 0$ by \eqref{eq:dec:A_sets}. The following proposition comprises the desired ``quenched'' monotonicity statement for the function $\mathscr{F}_{\Lambda}^{\omega}$ defined in \eqref{eq:dynamic_bc},  \eqref{eq:def_F}. 
\begin{proposition} \label{L:monot}$(\text{\eqref{eq:dec:A_sets}}, \varepsilon > 0, M \leq M_{\varepsilon, \Sigma_{\Lambda}})$,
\begin{equation}\label{eq:F_monot}
\frac{d\mathscr{F}_{\Lambda}^{\omega}(t)}{dt} \geq 0, \text{ for all $ t\in [0,1]$ and $\omega\in G_{\Lambda,M}$.}
\end{equation}
Moreover,
\begin{equation}\label{eq:sprinkling}
\{ Z^h(\varphi_{S'})  >  Z^{h-\varepsilon}(\widetilde{\varphi}_{S'}) \}  \subset G_{\Lambda,M}^c.
\end{equation}
\end{proposition} 
\begin{proof}
First, note that \eqref{eq:sprinkling} is an immediate consequence of \eqref{eq:F_monot}. Indeed, the latter implies that $\mathscr{F}_{\Lambda}^{\omega}(0) \leq \mathscr{F}_{\Lambda}^{\omega}(1)$ for all $\omega \in G_{\Lambda,M}$. Due to the gradient nature of the Hamiltonian, cf. \eqref{eq:H_pots}, with a straightforward change of variables, one sees that $ \mu_{{\Lambda\setminus S'}}^{\psi+ \varepsilon}(A^h) = \mu_{{\Lambda\setminus S'}}^{\psi}(A^{h-\varepsilon})$, for any boundary condition $\psi$, and thus
\begin{equation}\label{eq:dec600}
\mathscr{F}_{\Lambda}^{\omega}(1) \stackrel{\eqref{eq:def_F}}{=} \mu_{{\Lambda\setminus S'}}^{\xi^{\omega}(1)}(A^h) \stackrel{\eqref{eq:dynamic_bc}}{=} \mu_{{\Lambda\setminus S'}}^{(0\vee \widetilde{\varphi})+ \varepsilon}(A^h) = \mu_{{\Lambda\setminus S'}}^{0 \vee \widetilde{\varphi}}(A^{h-\varepsilon})  \stackrel{\eqref{eq:dec10}}{=} 
Z^{h-\varepsilon}(\widetilde{\varphi}_{S'}).
\end{equation}
Similarly, \eqref{eq:dynamic_bc} and \eqref{eq:def_F} imply that $Z^{h}({\varphi}_{S'}) = \mathscr{F}_{\Lambda}^{\omega}(0)$. Thus, \eqref{eq:sprinkling} follows. 

We now show \eqref{eq:F_monot}. Since $d\mathscr{F}_{\Lambda}^{\omega}(t)/dt$ is continuous, it is enough to prove \eqref{eq:F_monot} for $t \in \mathbb{Q} \cap [0,1]$, which we assume from now on. With $\Lambda' =\Lambda \setminus S'$, write 
\begin{equation}\label{eq:H_bcxi}
{H}_{\Lambda'}(\varphi | \xi)\stackrel{\text{def.}}{=} {H}_{{\Lambda}'}((\varphi_x)_{x\in \Lambda'}, (\xi_x)_{x\notin\Lambda'}),
\end{equation}
with ${H}_{\Lambda'}$ as defined in \eqref{eq:H_pots}. We view ${H}_{\Lambda'}(\cdot | \xi)$ as a map $\Omega_{\Lambda'}\to \mathbb{R}$. Then, denoting by $\langle \cdot \rangle_{\lambda}$ integration with respect to Lebesgue measure on $\Omega_{\Lambda'}$, it follows by dominated convergence from \eqref{eq:mu_Lambda}, with $\xi(t)=\xi^{\omega}(t)$ as in \eqref{eq:dynamic_bc}, that
\begin{equation}\label{eq:dec499}
\begin{split}
\frac{d\mathscr{F}_{\Lambda}^{\omega}}{dt} &= \frac{d}{dt}\bigg[\frac{\langle 1_{A^{h}}(\varphi') e^{- {H}_{\Lambda'}(\varphi' | \xi(t))}\rangle_{\lambda(d\varphi')}}{\langle e^{- {H}_{\Lambda'}(\varphi' | \xi(t))} \rangle_{\lambda(d\varphi')}} \bigg]
= \textnormal{Cov}_{\mu_{\Lambda'}^{\xi(t)}}\Big(1_{A^{h}}(\cdot), \frac{-d {H}_{\Lambda'}(\cdot | \xi(t))}{dt} \Big).
\end{split}
\end{equation}
By convolution with a suitable mollifier, we can arrange for $(F_{\delta})_{\delta >0}$ to be a family smooth approximations of the function $F \equiv 1_{A^{h}}$, such that 
\begin{equation}\label{eq:dec4999}
\begin{split}
&F_{\delta}, \partial_x F_{\delta},  \partial^2_{x,y} F_{\delta} \in C^{\infty}(\Omega_{\Lambda'}) \cap L^{\infty}(\lambda), \, x,y \in \Lambda', \, \sup_{\delta}\Vert F_{\delta}\Vert_{\infty} \leq 1, \\
&F_{\delta} \to F \text{ ptw. as $ \delta \to 0$, and }  \partial_x F_{\delta} \geq 0, \, x \in S,\, \partial_x F_{\delta} = 0, \, x \in \Lambda' \setminus S,
\end{split}
\end{equation}
where $\partial_x$ denotes partial derivative with respect to $\varphi_x$ (regarding boundedness, taking $k_\delta$ a (suitably rescaled) standard mollifier with compact support, and setting $F_{\delta}= k_\delta * F$, one has for instance $\Vert\partial_x F_{\delta} \Vert_{\infty}= \Vert(\partial_x k_{\delta}) * F\Vert_{\infty} \leq \Vert(\partial_x k_{\delta})\Vert_1  < \infty $, using Young's inequality.) The non-negativity of 
$\partial_x F_{\delta}$ is due to the fact that $A^h(\cdot)$ is increasing. To keep the notation light, we will often write $\langle \cdot \rangle_{\Lambda'}^{\psi(t)}$ below to denote expectation with respect to $\mu_{\Lambda'}^{\xi(t)}$. Clearly, by \eqref{eq:dec4999} and dominated convergence,
\begin{equation}\label{eq:dec4999.0}
\begin{split}
\langle (F-F_\delta)^2 \rangle_{\Lambda'}^{\xi(t)} &= (Z_{\Lambda'}^{\xi(t)})^{-1}\langle e^{- {H}_{\Lambda'}(\cdot | \xi(t))} (F-F_\delta)^2 \rangle_{\lambda(d\varphi')} \to 0, \text{ as } \delta \to 0.
\end{split}
\end{equation}
We now apply the Helffer-Sj\"ostrand formula, cf. Lemma \ref{L:HS}, to the right-hand side of \eqref{eq:dec499}.
Note that $F_{\delta} \in \text{Dom}(L_{\Lambda'}^{\xi(t)})$, by \eqref{eq:dec4999}, cf. \eqref{eq:dom_L}, and also $G(\cdot)=  \frac{-d {H}_{\Lambda'}(\cdot | \xi(t))}{dt} \in \text{Dom}(L_{\Lambda'}^{\xi(t)})$, as follows from \eqref{eq:dec501} below, on account of \eqref{eq:H_pots} and \eqref{eq:elliptic}. Thus,
\begin{equation} \label{eq:dec500}
\begin{split}
\textnormal{Cov}_{\mu_{\Lambda'}^{\xi(t)}}(F_{\delta},G) &\stackrel{\eqref{eq:HS2}}{=} \sum_{x \in \Lambda'} \int_0^{\infty} ds E_{x, \mu_{\Lambda'}^{\xi(t)}}[\partial F_{\delta}(x, \Phi_0) \cdot \partial G (X_s, \Phi_s) 1\{ s < \tau_{\Lambda'} \}] \\
&\stackrel{\text{Fubini}}{=} \sum_{x \in S} \Big\langle \partial F_{\delta}(x, \varphi') \cdot\mathbf{E}_{x, \varphi'}^{\mathcal{G}_{\Lambda'}, \xi(t)}\Big[\int_0^{\tau_{\Lambda'}} ds \,  \partial G (X_s, \Phi_s) \Big]\Big\rangle_{\Lambda'}^{\xi(t)},
\end{split}
\end{equation}
where the summation is over $x \in S$ due to \eqref{eq:dec4999}, $\langle \cdot \rangle_{\Lambda'}^{\xi(t)}$ governs the field $\varphi'$, and recall that $\Phi_s = (\Phi_s(x))_{x\in \Lambda'}$, $s \geq 0$, is the (unique strong) solution of the SDE's \eqref{eq:Langevin} with boundary condition $\Phi_s(x) = \xi_x(t)$ for all $x\in \partial\Lambda'$ and initial data $(\Phi_0)_{|_{\Lambda'}}= \varphi'$, $(X_s)_{s\geq 0}$ is the jump process on $\mathcal{G}_{\Lambda'}$ with generator \eqref{eq:defQ} and dynamic conductances given by \eqref{eq:RATES} (with $\Lambda'$ in place of $\Lambda$, and $\xi=\xi(t)$), and $P_{x, \mu_{\Lambda'}^{\xi(t)}}$ denotes the joint evolution of $(X_s, \Phi_s)_{s\geq 0}$ with initial condition $X_0=x$ and $\Phi_0$ distributed according to $\mu_{\Lambda'}^{\xi(t)}$. The application of the Fubini-Tonelli theorem in \eqref{eq:dec500} is justified, for the relevant integrand is bounded in absolute value by
$$
\sup_{x' \in \Lambda'} \Vert\partial_{x'} F_{\delta}\Vert_{\infty} \cdot \partial G (X_s, \Phi_s) 1\{ s < \tau_{\Lambda'} \}, 
$$
the first supremum is finite by \eqref{eq:dec4999}, and it follows from \eqref{eq:dec502} below, with the help of \eqref{eq:elliptic}, that the iterated integral $E_{x, \mu_{\Lambda'}^{\xi(t)}}[\int_0^{\tau_{\Lambda'}} ds |\partial G (X_s, \Phi_s)| ]$ converges (note that $E_{x, \mu_{\Lambda'}^{\xi(t)}}[\tau_{\Lambda'}]$ is finite by uniform ellipticity, since $\mathcal{G}_{\Lambda'}$ is a finite graph). We next separately consider the function 
\begin{equation}\label{eq:dec:500.0}
u(x, \varphi') =  \mathbf{E}_{x, \varphi'}^{\mathcal{G}_{\Lambda'}, \xi(t)}\Big[\int_0^{\tau_{\Lambda'}} ds \,  \partial G (X_s, \Phi_s) \Big], \ x\in \Lambda', \varphi' \in \Omega_{\Lambda'},
\end{equation}
 appearing on the right-hand side of \eqref{eq:dec500}. By \eqref{eq:HS1}, \eqref{eq:BVP} (note that $\partial G(x,\varphi)= \partial \widetilde{G} (x, \varphi)$, where $\widetilde{G}:=G-\langle G \rangle_{\Lambda'}^{\psi(t)}$), $u$, extended to $0$ outside $\Lambda'$, is a classical solution to the boundary value problem
 \begin{equation}
 \label{eq:dec:500.1}
 \begin{cases}
 - \mathscr{L}_{\Lambda'}^{\xi(t)} u(x, \varphi') = \partial G(x, \varphi'), & x\in {\Lambda'}, \varphi' \in \Omega_{\Lambda'}, \\
 u(x, \varphi')= 0, & x \notin {\Lambda'}, \varphi' \in \Omega_{\Lambda'}.
\end{cases}
 \end{equation}
 The source term $\partial G(x, \varphi')$ can be computed explicitly as follows. First, observe that
\begin{equation} \label{eq:dec501}
G(\varphi')= \frac{-d {H}_{\Lambda'}(\varphi' | \xi(t))}{dt} \stackrel{\eqref{eq:H_bcxi}}{=} - \sum_{y \notin \Lambda'}\frac{\partial {H}_{\Lambda'}(\varphi' | \xi)}{\partial \xi_y}\Big|_{\xi =  \xi(t)} \cdot \dot{\xi}_y(t),
\end{equation}
where dot denotes derivative with respect to $t$. Consequently, for $x\in \Lambda'$,
\begin{equation}\label{eq:dec502}
\begin{split}
\partial G (x, \varphi') = \frac{\partial G(\varphi')}{\partial \varphi_x'} &=  - \sum_{y \notin\Lambda'}\frac{\partial^2 {H}_{\Lambda'}(\varphi' | \xi)}{\partial \varphi_x' \partial \xi_y}\Big|_{\xi =  \xi(t)} \cdot \dot{\xi}_y(t) 
\end{split}
\end{equation}
which, on account of \eqref{eq:H_pots}, vanishes unless $x \in X$ for some $X \in \mathcal{B}$ with $X \cap  (\Lambda')^c \neq \emptyset$. As we now explain, $u(x, \varphi')$, $x \in \Lambda'$, can be viewed as an $\mathscr{L}_{\Lambda'}^{\xi(t)}$-harmonic function, i.e. a solution to the bare Laplace equation $\mathscr{L}_{\Lambda'}^{\xi(t)}u=0$ (without source term), but with non-vanishing boundary condition. Due to the particular form of \eqref{eq:dec502}, and keeping in mind that $u$ vanishes on the boundary, cf. \eqref{eq:dec:500.1} (which, along with the last equality in \eqref{eq:comm}, is why the action of $Q_{\Lambda'}^{\xi(t), \varphi'}$ can be expressed as in the second line below, with a summation over $\Lambda'$ rather than $\overline{\Lambda}'$), one has, for $x \in \Lambda'$, $\varphi' \in \Omega_{\Lambda'}$,
\begin{equation}\label{eq:dec502.0}
\begin{split}
0 &\stackrel{\eqref{eq:dec:500.1}}{=} - \mathscr{L}_{\Lambda'}^{\xi(t)} u(x, \varphi') - \partial G(x, \varphi') \\
&\underset{\eqref{eq:L}}{\stackrel{\eqref{eq:defQ}}{=}} -L_{\Lambda'}^{\xi(t)}u(x,\varphi') -\bigg[\sum_{y \in {\Lambda'}} \frac{\partial^2 {H}_{\Lambda'}(\varphi')}{\partial \varphi'_x \partial \varphi'_y}\Big|_{\varphi' =  \xi(t) \,\text{on}\, (\Lambda')^c}u(y,\varphi')\bigg] - \partial G(x, \varphi')\\
&\stackrel{\eqref{eq:dec502}}{=}  -L_{\Lambda'}^{\xi(t)}u(x,\varphi') -\sum_{y \in \overline{\Lambda'}} \frac{\partial^2 {H}_{\Lambda'}(\varphi')}{\partial \varphi'_x \partial \varphi'_y}\Big|_{\varphi' =  \xi(t) \,\text{on}\, (\Lambda')^c} [ u(y,\varphi') + \dot{\xi}_y(t)1_{y\notin  \Lambda'}].
\end{split}
\end{equation}
Thus, defining 
\begin{equation}\label{eq:dec:502.2}
\tilde{u}(x,\varphi')= u(x,\varphi') + \dot{\xi}_x(t)1_{x\notin  \Lambda'}, 
\end{equation}
and observing that $ {L}_{\Lambda'}^{\xi(t)} u=  {L}_{\Lambda'}^{\xi(t)} \tilde{u}$, since ${L}_{\Lambda'}^{\xi(t), }$ is a differential operator on $\Omega_{\Lambda'}$, cf. \eqref{eq:hs1}, it follows from \eqref{eq:dec502.0} and \eqref{eq:dec:500.1} that $\tilde{u}$ satisfies
 \begin{equation}
 \label{eq:dec:502.1}
 \begin{cases}
 - \mathscr{L}_{\Lambda'}^{\xi(t)} \tilde u(x, \varphi') =0, & x\in {\Lambda'}, \varphi' \in \Omega_{\Lambda'}, \\
 \tilde u(x, \varphi')= \dot{\xi}_x(t), & x\notin{\Lambda'}, \varphi' \in \Omega_{\Lambda'}.
\end{cases}
 \end{equation}
Hence, for all  $x \in \Lambda', \, \varphi' \in \Omega_{\Lambda'}$,
\begin{equation}\label{eq:dec502.7}
u(x,\varphi') \stackrel{\eqref{eq:dec:502.2}}{=} \tilde{u}(x,\varphi') =  \mathbf{E}_{x, \varphi'}^{\mathcal{G}_{\Lambda'}, \xi(t)}[\dot{\xi}_{X_{\tau_{\Lambda'}}}(t)] 
\end{equation}
(the last equality follows for instance by verifying, on account of \eqref{eq:dec:502.1}, that $M_s = \tilde{u}(X_{s\wedge \tau_{\Lambda'}}, \Phi_{s \wedge \tau_{\Lambda'}})$, $s \geq 0$, is a uniformly integrable martingale with respect to the natural filtration associated to the process $(X_{\cdot}, \Phi_{\cdot})$, and applying the optional stopping theorem). 

Returning to \eqref{eq:dec500}, with \eqref{eq:dec:500.0}, \eqref{eq:dec502.7}, one obtains, for all $t \in \mathbb{Q}\cap [0,1]$ and $\omega \in G_{\Lambda,M}$ (recall this is the randomness governing the boundary condition $\xi(t)= \xi^{\omega}(t)$, cf. \eqref{eq:dynamic_bc}), keeping in mind that $ \partial F_{\delta}(x, \varphi')$ is non-negative, see \eqref{eq:dec4999}, and that $\Lambda' =\Lambda \setminus S$, so that $X_{\cdot}$ exits $\Lambda'$ either through $\partial \Lambda$ or $\partial_{\text{i}}S'$,
\begin{equation} \label{eq:dec510}
\begin{split}
&\textnormal{Cov}_{\mu_{\Lambda'}^{\xi(t)}}(F_{\delta},G) = \sum_{x \in S} \Big\langle \partial F_{\delta}(x, \varphi') \mathbf{E}_{x, \varphi'}^{\mathcal{G}_{\Lambda'}, \xi(t)}[\dot{\xi}_{X_{\tau_{\Lambda'}}}(t)] \Big\rangle_{\Lambda'}^{\xi(t)} \\
&\stackrel{\eqref{eq:dynamic_bc}}{=} \sum_{x \in S} \Big\langle \partial F_{\delta}(x, {\varphi}') \Big[\varepsilon \mathbf{P}_{x, \varphi'}^{\mathcal{G}_{\Lambda'}, \xi(t)}[X_{\tau_{\Lambda'}} \in \partial \Lambda] - \sum_{y \in \partial_{\text{i}}S'} (\varphi_y - \widetilde{\varphi}_y - \varepsilon) \cdot \mathbf{P}_{x, \varphi'}^{\mathcal{G}_{\Lambda'}, \xi(t)}[X_{\tau_{\Lambda'}} = y] \Big] \Big\rangle_{\Lambda'}^{\xi(t)}\\
&\stackrel{\eqref{eq:decM}
}{\geq} \sum_{x \in S} \Big\langle \partial F_{\delta}(x, {\varphi}') \mathbf{P}_{x, \varphi'}^{\mathcal{G}_{\Lambda'}, \xi(t)}[X_{\tau_{\Lambda'}} \in \partial_{\text{i}} S'] \cdot \Big[M_{\varepsilon,\Sigma_{\Lambda}} - \mathbf{E}_{x, \varphi'}^{\mathcal{G}_{\Lambda'}, \xi(t)}[ \varphi_{X_{\tau_{\Lambda'}}}  -\widetilde{\varphi}_{X_{\tau_{\Lambda'}}}  |X_{\tau_{\Lambda'}} \in \partial_{\text{i}} S'] \Big] \Big\rangle_{\Lambda'}^{\xi(t)}\\
&\stackrel{\eqref{eq:dec15}}{\geq} (M_{\varepsilon,\Sigma_{\Lambda}} - M) \sum_{x \in S} \Big\langle \partial F_{\delta}(x, {\varphi}') \mathbf{P}_{x, \varphi'}^{\mathcal{G}_{\Lambda'}, \xi(t)}[X_{\tau_{\Lambda'}} \in \partial_{\text{i}} S'] \Big\rangle_{\Lambda'}^{\xi(t)} \stackrel{\eqref{eq:dec4999}}{\geq} 0,
\end{split}
\end{equation}
since $M \leq M_{\varepsilon,\Sigma_{\Lambda}}$ by assumption. The third line in \eqref{eq:dec510} explains the specific choice of the cut-off parameter $M_{\varepsilon,\Sigma_{\Lambda}}$ in \eqref{eq:decM}. Finally, returning to $\mathscr{F}_{\Lambda}^{\omega}$, it follows that
\begin{equation} \label{eq:dec510.0}
\begin{split}
\frac{d\mathscr{F}_{\Lambda}^{\omega}}{dt} &\stackrel{\eqref{eq:dec499}}{=} \textnormal{Cov}_{\mu_{\Lambda'}^{\xi(t)}}(F,G)  \geq \textnormal{Cov}_{\mu_{\Lambda'}^{\xi(t)}}(F_{\delta},G) - \Big|\textnormal{Cov}_{\mu_{\Lambda'}^{\xi(t)}}(F- F_{\delta},G)\Big| \\
&\stackrel{\eqref{eq:dec510}}{\geq} - \Vert F-F_{\delta}\Vert_{L^2(\mu_{\Lambda'}^{\xi(t)})} \Vert G_c\Vert_{L^2(\mu_{\Lambda'}^{\xi(t)})},
\end{split}
\end{equation}
where $ G_c= G- \mathbb{E}_{\mu_{\Lambda'}^{\xi(t)}}[G]$, using Cauchy-Schwarz in the last step. The claim \eqref{eq:F_monot} follows by letting $\delta \to 0$ in \eqref{eq:dec510.0}, using \eqref{eq:dec4999.0}.
\end{proof}

\section{Comparison estimates and the error term}\label{S:error}

The successful application of Proposition \ref{L:monot} hinges on a suitable bound on the probability of $G_{\Lambda, M}^c$, cf. \eqref{eq:dec15}, as $\Lambda \nearrow \mathbb{Z}^d$, which is derived in this section, see Lemma \ref{L:error} below. It relies on suitable estimates for hitting probabilities of random walks among time-dependent conductances, in terms of the corresponding quantities for simple random walk, which we derive first, see Lemma \ref{L:hit_comparison}. This lemma is of independent interest, so we state it in reasonable generality. Its tailored version is stated thereafter in Corollary \ref{C:exit_dist}. We will also need to compare entrance {distributions}, which is done in Lemma \ref{L:hit_dist1}, in order to obtain the desired bound on the error term in Lemma \ref{L:error}. Theorem \ref{P:dec_ineq} then follows readily from Proposition \ref{L:monot} and Lemma \ref{L:error}. The proof can be found at the end of this section.

We consider a space $\mathcal{A}$ of environments on $\mathcal{G}= (\mathbb{Z}^d, \mathscr{E})$, defined as
\begin{equation}\label{eq:space_environments}
\mathcal{A}= \{ a: \mathbb{R}_+ \times \mathscr{E} ; \, a_t(e)= a_t(-e), \, c_0^{-1} \leq  a_t(e) \leq c_0, \text{ for all } t\geq 0, \, e \in \mathscr{E} \},
\end{equation}
and write $P_x^{a}$, $x \in \mathbb{Z}^d$, $a \in \mathcal{A}$, for the canonical law of the continuous-time jump process with generator 
\begin{equation}\label{eq:gen_a}
\mathscr{L}_t^a f(x) = \sum_{y:y \sim x}a_t(x,y)(f(y)-f(x)), 
\end{equation}
acting on bounded functions $f: \mathbb{Z}^d\to \mathbb{R}$. We write $(X_t)_{t \geq 0}$ for the corresponding canonical process. The total jump rate $\sum_{e:x(e)=x}a_t(e)$ out of $x$ at time $t \geq 0$ is not normalized, 
and the process $(X_t)_{t \geq 0}$ is usually referred to as \textit{variable-speed} random walk (among dynamic conductances).  Recall that $P_x^{*}$ refers to the law of (continuous-time) simple random walk on $\mathcal{G}$, which corresponds to the environment $a\in \mathcal{A}$ with $a(t,x)= |\Gamma|^{-1}$, for all $t,x$ (by possibly redefining $c_0$, we may assume that $c_0 \geq |\Gamma|$). Denoting by $p^a(x,y;t)=P_x^{a}[X_t=y]$ the corresponding heat kernel, one has from \cite{GOS01}, Props. B.3 and B.4, that for all $a \in \mathcal{A}$,
\begin{align}
&p^a(x,y;t) \geq \frac{c}{1\vee t^{d/2}}, \text{ if } |x-y| \leq \sqrt{t} \label{eq:HKLB}\\
&p^a(x,y;t) \leq \frac{c'}{1\vee t^{d/2}}e^{-\frac{|x-y|}{1\vee \sqrt{t}}}, \text{ for } t\geq 0, x,y \in \mathbb{Z}^d \label{eq:HKUB}.
\end{align}
(In fact one has Gaussian upper bounds in the long-time regime, see \cite{DD05}, Prop. 4.2.) The constants $c,c'$ depend on $d$, $\Gamma$ and $c_0$ as appearing in \eqref{eq:space_environments}, but are otherwise uniform in $a$. From these heat kernel estimates, one classically deduces for the corresponding Green function $g^a(x,y)= \int_0^{\infty} P_x^{a}[X_t =y] dt$, first by restricting the integral to $t \geq |x-y|^2$ and using \eqref{eq:HKLB} to obtain a lower bound, or splitting the integral at $t=1$ and using \eqref{eq:HKUB} to obtain an upper bound, that 
\begin{equation}\label{eq:GF_bounds}
\frac{c}{1 \vee |x-y|^{d-2}} \leq g^a(x,y) \leq \frac{c'}{ 1\vee |x-y|^{d-2}}, \text{ for } a\in \mathcal{A}, x,y \in \mathbb{Z}^d.
\end{equation}
The next lemma yields a comparison estimate for hitting probabilities of $P_x^{a}$ with those of simple random walk, and will be used in several instances below. 
The rationale of its proof is that, in spite of the nuisance due to the dynamic environment $a$, given bounds as in \eqref{eq:GF_bounds} for the Green function, one can often afford to develop an ``approximate'' potential theory, by which 
$g^*(\cdot ,\cdot)$ replaces the kernel $g^{a}(\cdot, \cdot)$. 

\begin{lem} $(a \in \mathcal{A})$
\label{L:hit_comparison}
\medskip
\noindent
There exists $c_{3} \in (1,\infty)$, such that, for all $U\subset \subset \mathbb{Z}^d$, defining
\begin{equation}\label{eq:sigma}
\sigma_x^*(U) = \frac{\sup_{y \in U} g^*(x,y)}{\inf_{y \in U} g^*(x,y)},
\end{equation}
one has, for all $x \notin U$,
\begin{align}
&\big( c_{3}\sigma_x^*(U) \big)^{-1} P_x^*[H_U < \infty] \leq P_x^a[H_U < \infty] \leq \big(c_{3}\sigma_x^*(U)\big) P_x^*[H_U < \infty] \label{eq:CE1}.
\end{align}
\end{lem}
\begin{proof}
First, note that $P_x^*[H_U < \infty]$ admits straightforward upper and lower bounds based on potential theory for simple random walk. Recalling \eqref{eq:equil_meas} and \eqref{eq:cap_def0}, by last-exit decomposition for the simple random walk, see for instance \cite{LL10}, Prop. 4.6.4, one has $P_x^*[H_U < \infty]= \sum_{y \in U}g^*(x,y) e^*_U(y)$, and therefore, in view of \eqref{eq:sigma},
\begin{equation}\label{eq:hit1}
\begin{split}
\sigma_x^*(U)^{-1}  \sup_{y \in U} g^*(x,y) \cdot \text{cap}^*(U) \leq  P_x^*[H_U < \infty] \leq \sigma_x^*(U)  \inf_{y \in U} g^*(x,y) \cdot \text{cap}^*(U).
\end{split}
\end{equation}

We now show \eqref{eq:CE1}, starting with the upper bound. Abbreviating, for $a\in \mathcal{A}$, $y \in U$, $x \notin U$,
\begin{equation}\label{eq:defH}
\nu_x^a(y)= P_x^a[H_U < \infty, X_{H_U}= y] \text{ and } \bar{\nu}^a_x(y) = \nu_x^a(x,y)/ P_x^a[H_U < \infty], 
\end{equation}
and writing, whenever $H_U < \infty$, $(\theta_{H_U}a) \in \mathcal{A}$ for the environment defined via $(\theta_{s}a)_t(e)= a_{s + t}(e)$, for $s,t \geq 0$, $e\in \mathscr{E}$, using the strong Markov property at time $H_U$, we obtain, for all $z \in U$, $x \notin U$, 
\begin{equation}\label{eq:hit10}
\begin{split}
g^a(x,z)&= E_x^a\Big[\int_0^\infty dt 1_{\{ X_t=z \}}\Big]= \sum_{y\in U} E_x^a\Big[ 1_{\{ H_U < \infty,  X_{H_U}= y \}} \cdot E_y^{\theta_{H_U}a}\Big[ \int_0^\infty dt 1_{\{ X_t=z \}}\Big] \Big]\\
&\stackrel{\eqref{eq:GF_bounds}} \geq c_{4} \sum_{y\in U} \nu_x^a(y) g^*(y,z).
\end{split}
\end{equation}
Moreover, noting that, by \eqref{eq:GF_bounds}, $g^a(x,z)\leq c_{5} g^*(x,z) \leq c_{5}\sup_{y \in U} g^*(x,y)$ for all $z \in U$, $x \notin U$, it follows from \eqref{eq:hit10}, multiplying on both sides by $ \nu_x^a(z)$ and summing over $z \in U$, recalling the definition of $E^*(\cdot)$ in \eqref{eq:en1}, that
\begin{equation}\label{eq:hit11}
c_{5} \sup_{y \in U} g^*(x,y)  P_x^a[H_U < \infty] \geq c_{4} E^*(\nu_x^a)= c_{4} E^*(\bar{\nu}_x^a) P_x^a[H_U < \infty]^2,
\end{equation}
and therefore, since $\bar{\nu}_x$ is a probability measure supported on $U$, cf. \eqref{eq:defH},
\begin{equation*}
\begin{split}
P_x^a[H_U < \infty] &\stackrel{\eqref{eq:hit11}}{\leq}\frac{c\sup_{y \in U} g^*(x,y)}{E^*(\bar{\nu}_x^a)} \leq \frac{c \sup_{y \in U} g^*(x,y)}{\inf \{ E^*(\nu); \nu \text{ a prob. meas., } \text{supp}(\nu) \subset U  \}}\\
&\stackrel{\eqref{eq:cap2}}{=} c\sup_{y \in U} g^*(x,y) \cdot \text{cap}^*(U) \stackrel{\eqref{eq:hit1}}{\leq} c\sigma_x^*(U) P_x^*[H_U < \infty].
\end{split}
\end{equation*}
This is the asserted upper bound in \eqref{eq:CE1}. For the lower bound, define the random variable $W(\nu)$, for $\nu$ any probability measure supported on $U$, as
\begin{equation}\label{eq:hit13}
W(\nu)= \sum_{y\in U} \nu(y) \int_0^{\infty} dt \frac{1\{ X_t=y\}}{g^*(x,y)}.
\end{equation}
Clearly, $W(\nu)$ is non-negative and $H_U < \infty$ if $W(\nu)>0$. Hence, by the Paley-Zygmund inequality, for $x \notin U$,
\begin{equation}\label{eq:hit14}
P_x^a[H_U < \infty ] \geq P_x^a[W(\nu)>0] \geq \frac{E_x^a[W(\nu)]^2}{E_x^a[W(\nu)^2]}.
\end{equation}
We consider the first and second moment separately. For the former, one has the lower bound
\begin{equation}\label{eq:hit15}
E_x^a[W(\nu)] = \sum_{y \in U} \nu(y)\frac{g^a(x,y)}{g^*(x,y)} \stackrel{\eqref{eq:GF_bounds}}{\geq} c,
\end{equation}
since $\nu$ is normalized by assumption. To get an upper bound on the second moment, observe that, with $\mathcal{F}_s^X = \sigma(X_t; t \leq s)$,
\begin{equation}\label{eq:hit16}
\begin{split}
 &E_x^a[W(\nu)^2] \stackrel{\eqref{eq:hit13}}{=}   E_x^a\Big[\sum_{y,z\in U} \nu(y) \nu(z) \int_0^{\infty} ds  \int_0^{\infty} dt \frac{ 1\{ X_s=y\} 1\{ X_t=z\}}{g^*(x,y)g^*(x,z)}\Big] \\
 &\qquad \begin{array}{cl}
= &  \hspace{-1ex}  \displaystyle   2\sum_{y,z\in U} \frac{\nu(y) \nu(z)}{g^*(x,y)g^*(x,z)} E_x^a \Big[ \int_0^{\infty} ds 1_{\{X_s=y\}} E_x^a \Big[ \int_s^{\infty} dt 1_{\{ X_t=z\}} \Big|  \mathcal{F}_s^X \Big]\Big]\\
  \underset{\text{Markov}}{\overset{\text{simple}}{\leq}} & \hspace{-1ex}  \displaystyle   2\sum_{y,z\in U} \frac{\nu(y) \nu(z)}{g^*(x,y)g^*(x,z)} \int_0^{\infty} ds E_x^a \Big[  1_{\{X_s=y\}} \sup_{a'\in \mathcal{A}} g^{a'} (y,z) \Big]\\
  \stackrel{\eqref{eq:GF_bounds}}{\leq} &  \hspace{-1ex} \displaystyle   c \sum_{y,z\in U} \frac{\nu(y) \nu(z)}{g^*(x,z)}g^*(y,z) \stackrel{\eqref{eq:en1}}{\leq} \frac{c E^*(\nu)}{\inf_{z\in U}g^*(x,z)}.
\end{array}
\end{split}
\end{equation}
Returning to \eqref{eq:hit14}, using \eqref{eq:hit15}, \eqref{eq:hit16} and optimizing over $\nu$ yields
\begin{equation*}
\begin{array}{rcl}
P_x^a[H_U < \infty ] & \hspace{-1ex}  \geq &  \hspace{-1ex} c' \inf_{z\in U}g^*(x,z) \sup_{\nu} E^*(\nu)^{-1} 
\stackrel{\eqref{eq:cap2}, \eqref{eq:hit1}}{\geq} c' \sigma_x(U)^{-1} P_x^*[H_U < \infty],
\end{array}
\end{equation*}
which is the desired lower bound.
\end{proof}
The following annealed estimate is tailored to our purposes. This brings into play the measure $\mathbf{P}_{x,\varphi}^{\mathcal{G}_{\Lambda},\xi}$ from the Helffer-Sj\"ostrand representation, cf. below \eqref{eq:Dirichletform}. 

 \begin{cor} \label{C:exit_dist}
$(\Lambda, S'\subset \subset \mathbb{Z}^d$, $\Lambda \supset {S}'$, $\Lambda' = \Lambda \setminus S')$

\medskip \noindent
There exist $c_5, c_6$ (independent of $\Lambda$ and $S'$) and $\Lambda_0 (x,S') \subset \subset \mathbb{Z}^d$ with $x \in \Lambda_0 (x,S')$ such that, for all $x \in \mathbb{Z}^d \setminus S'$, $\xi \in \Omega$, $\varphi \in \Omega_{\Lambda'}^\xi$, if $\Lambda \supset \Lambda_0(x,S')$,
\begin{equation}\label{eq:exit_distr}
c_5 \sigma_x^*(S')^{-1} P_x^*[H_{S'}< \infty] \leq \mathbf{P}_{x,\varphi}^{\mathcal{G}_{\Lambda'}, \xi} [
 H_{S'}< H_{\Lambda^c}] \leq c_6 \sigma_x^*(S') P_x^*[H_{S'}< \infty], 
 \end{equation}
\end{cor}

\begin{proof}
By the discussion leading to \eqref{eq:RATES}, and recalling the definition of $P_{\varphi}^{\Lambda',\xi}$ below \eqref{eq:Langevin}, we see that for bounded, measurable $f$ on $D(\mathbb{R}_+, \overline{\Lambda})$, 
\begin{equation}\label{eq:disint1}
\mathbf{E}_{x,\varphi}^{\mathcal{G}_{\Lambda'}, \xi} [f((X_t)_{ t\geq 0})] = \int dP_{\varphi}^{\Lambda',\xi}(\Phi)\, E_x^{a_{\Lambda'}(\Phi)} [f((X_t)_{ t\geq 0})],
\end{equation}
with $a_{\Lambda'}(\Phi) $ as in \eqref{eq:RATES}, and $P_x^{a_{\Lambda'}(\Phi)}$ is the law of the random walk on $\mathcal{G}_{\Lambda'}$ among the dynamic conductances $a_{\Lambda'}(\Phi) $, killed when first exiting $\Lambda'$. 
We will apply Lemma \ref{L:hit_comparison} directly to $P_x^{a_{\Lambda'}(\Phi)}$.

For $ a\in \mathcal{A}$, cf. \eqref{eq:space_environments}, denote by $P^{{\Lambda}, a}_x$, $x \in \Lambda$, the law of the variable-speed random walk on $\mathcal{G}$ killed when first exiting $\Lambda$. Then, since 
\begin{equation}\label{eq:hit-40}
P^{{\Lambda}, a}_x[H_{S'}< \infty] = P^{ a}_x[H_{S'}< \tau_{\Lambda}]  = P^{ a}_x[H_{S'}< \infty]- P^{ a}_x[ \tau_{\Lambda}\leq H_{S'} < \infty]  
\end{equation}
for all $x \in \Lambda\setminus S'$, and, by the strong Markov property at time $H_{S'} \wedge \tau_{\Lambda} (= \tau_{\Lambda'})$, 
\begin{equation}\label{eq:hit-41}
\begin{split}
&P^{ a}_x[ \tau_{\Lambda}\leq H_{S'} < \infty]
= P_x^a[ X_{\tau_{\Lambda'}} \in \partial \Lambda, H_{S'} \circ \tau_{\Lambda'} < \infty ] \\&= \sum_{y \in \partial \Lambda} E_x^a\Big[1_{\{  X_{\tau_{\Lambda'}} =y\}} P_y^{\theta_{\tau_{\Lambda'}} a}[H_{S'}< \infty] \Big]
\leq \sup_{y\in \partial \Lambda, a\in \mathcal{A}} P_y^{a}[H_{S'}< \infty] \stackrel{\eqref{eq:CE1}}{\leq} c(S') d_{\mathcal{G}}( \partial \Lambda, S')^{2-d},
\end{split}
\end{equation}
for all $\Lambda \supset \Lambda_1(S')$ (using in the last step that $\sup_{x\in \partial \Lambda} \sigma_x^*(S')< 10$ for suitable $\Lambda_1(S')$, and all $\Lambda  \supset \Lambda_1(S')$), we deduce from \eqref{eq:hit-40} and \eqref{eq:hit-41}, that for all $x \notin S$, all $\Lambda \supset \Lambda_2(x,S')$ and all $a\in \mathcal{A}$, $\frac12 \leq \frac{P^{{\Lambda}, a}_x[H_{S'}< \infty]}{P^{ a}_x[H_{S'}< \infty]} \leq 1 $. From this, it follows that \eqref{eq:CE1} has the following finite-volume version. Namely, for all $a \in \mathcal{A}$, $x \notin S'$, $\Lambda \supset  \Lambda_2(x,S')$,
\begin{equation}\label{eq:hit-42}
(2c_3 \sigma_x^*(S'))^{-1} P_x^{*}[H_{S'} < \infty] \le P_x^{\Lambda,a}[H_{S'} < \infty] \leq c_3 \sigma_x^*(S') P_x^{*}[H_{S'} < \infty].
\end{equation}
The claim \eqref{eq:exit_distr} now follows by extending $a_{\Lambda'}(\Phi)$ to an environment $a(\Phi)$ on $\mathscr{E}$ by declaring each edge in $\mathscr{E}\setminus \mathscr{E}_{\Lambda'}$ to have, say, conductance $1$ for all $t \geq 0$, thus obtaining that $a(\Phi) \in \mathcal{A}$, due to \eqref{eq:thea} and \eqref{eq:elliptic}, so \eqref{eq:hit-42} yields $P_{\varphi}^{\Lambda',\xi}(d\Phi)$-a.s. bounds on $P_x^{\Lambda,a(\Phi)}[H_{S'} < \infty] = P_x^{a_{\Lambda'}(\Phi)}[ H_{S'}< H_{\Lambda^c} ]$, for $x \notin S$ and $\Lambda \supset  \Lambda_2(x,S')$, and \eqref{eq:exit_distr} follows upon integrating over $\Phi$, using \eqref{eq:disint1}.
\end{proof}

We will also need the following comparison inequality for hitting \textit{distributions}. In what follows we write
\begin{equation}\label{eq:hit_dist0}
\nu_x^{S'}(y)= P_x^*[H_{S'}< \infty, X_{H_{S'}}=y], \text{ for } S'\subset \mathbb{Z}^d, \, x,y \in \mathbb{Z}^d.
\end{equation}
\begin{lem}\label{L:hit_dist1}
$(\Lambda, S'\subset \subset \mathbb{Z}^d$, $\Lambda \supset {S}'$, $\Lambda' = \Lambda \setminus S')$

\medskip \noindent
For all $x \in \Lambda'$, $\xi \in \Omega$, $\varphi \in \Omega_{\Lambda'}^\xi$, and $f: S' \to [0,\infty)$,
\begin{equation}\label{eq:hit_dist1}
\mathbf{E}_{x,\varphi}^{\mathcal{G}_{\Lambda'},\xi}[f(X_{\tau_{\Lambda'}})1\{X_{\tau_{\Lambda'}} \in S'\}] \leq c_{7} \langle G^*\nu_x^{S'}, f \rangle_{\ell^2(S')},
\end{equation}
where $ (G^*\nu_x^{S'})(y)= \sum_z g^*(y,z)\nu_x^{S'}(z)$, cf. below \eqref{eq:en1}.
\end{lem}
\begin{proof}
Recalling the quenched law $P_x^{a_{\Lambda'}(\Phi)}$ defined in \eqref{eq:disint1}, and using Lemma \ref{L:hit_comparison} with $U=\{ y\}$ (noting that $\sigma_x^*(\{ y\})=1$, cf. \eqref{eq:sigma}), we have, for $y\in S'$, $x \in \Lambda'$, $P_{\varphi}^{\Lambda',\xi}(d\Phi)$-a.s.,
\begin{equation}\label{eq:hit_dist2}
\begin{split}
&P_x^{a_{\Lambda'}(\Phi)}[X_{\tau_{\Lambda'}} = y ]
=P_x^{a_{\Lambda'}(\Phi)}[H_y < \infty ] \leq \sup_{a \in \mathcal{A}}P_x^{a}[H_y < \infty ] \stackrel{\eqref{eq:CE1}}{\leq} c_{3} P_x^{*}[H_y < \infty ]\\
&\quad = c_{3}(\nu_x^{S'}(y) + P_x^*[H_{S'}< H_y < \infty]) \underset{\text{Markov}}{\stackrel{\text{strong}}{=}} c_{3}\Big(\nu_x^{S'}(y) + \sum_{z\neq y}\nu_x^{S'}(z)P_z^*[H_y< \infty]\Big)\\
&\quad \leq c\sum_{z \in S'}  \nu_x^{S'}(z) g^*(z,y)= (G^*\nu_x^{S'})(y),
\end{split}
\end{equation}
where we used in the last line that $g^*(y,y) > 1$ and 
$$
P_z^*[H_y < \infty]=g^*(z,y)P_y[\widetilde{H}_y=\infty]\leq g^*(z,y),
$$
by last exit decomposition in $y$. From \eqref{eq:hit_dist2}, and since $f\geq 0$ by assumption, we deduce that $E_x^{a_{\Lambda'}(\Phi)}[f(X_{\tau_{\Lambda'}})1\{X_{\tau_{\Lambda'}} \in S'\}]$ is bounded by the right-hand side of \eqref{eq:hit_dist1}, and the claim follows by averaging over $\Phi$, see \eqref{eq:disint1}.
\end{proof}

In the next lemma, we derive a suitable (i.e. sufficiently sharp, for the purposes we have in mind, cf. Lemma \ref{L:renorm1} and Corollary \ref{C:DI1} in the next section) upper bound for the probability of the `bad' event $G_{\Lambda,M}^c$ defined in \eqref{eq:dec15}. The proof relies on a careful application of the exponential Brascamp-Lieb inequality, see Lemma \ref{L:BL}, together with the previously derived comparison estimates, Corollary \ref{C:exit_dist} and Lemma \ref{L:hit_dist1}.

In what follows let $(\Lambda_n)_{n\geq 0}$ be any increasing sequence of finite sets exhausting $\mathbb{Z}^d$, and recall that under $Q_{\Lambda}^0$, the field $\omega=(\varphi,\widetilde{\varphi})$ is distributed as two independent copies of $\varphi$ under $\mu_{\Lambda}^0$. We stress that the constants $c_2, c_3$ appearing below are independent of the choice of $(\Lambda_n)_{n\geq 0}$.

\begin{lem}\label{L:error} $(\emptyset \neq S,S' \subset \subset \mathbb{Z}^d, S\cap S' = \emptyset)$

\medskip
\noindent For any sequence $(M_n)_{n \geq 0}$, with $M_n > 0$ for all $n$, denoting $M_{\infty} = \limsup_n M_n$,
and letting
\begin{equation}\label{eq:sigma_global}
\sigma^*\equiv \sigma^*(S,S')= \sup_{x\in S}\sigma_x^*(S'), \quad (\text{cf. \eqref{eq:sigma} for the definition of $\sigma_x^*(S'))$},
\end{equation}
one has
\begin{equation}\label{eq:dec649}
\liminf_{n \to \infty}  Q_{\Lambda_n}^0[G_{\Lambda_n, M_n}^c] \leq c_1|S|2^{| \partial_{\text{i}} S'|}\exp \Big\{-c_2\Big(\frac{\textnormal{cap}^*(S') M_{\infty}}{(\sigma^*)^2 |\partial_{\textnormal{i}}S'| }\Big)^2\Big\}.
\end{equation}
\end{lem}

\begin{proof}

Let $\Lambda \supset S \cup S'$ be a finite subset of $\mathbb{Z}^d$. We will later take $\Lambda = \Lambda_n$ and let $n \to \infty$. First, in view of \eqref{eq:dec15}, note that, for $M > 0$,
$$
G_{\Lambda,M}^c \subset \bigcup_{t, (x,\varphi')} \bigg( \Big\{  \mathbf{E}_{x, \varphi'}^{\mathcal{G}_{\Lambda'}, \xi(t)}[{\varphi}_{X_{\tau_{\Lambda'}}}  |X_{\tau_{\Lambda'}} \in \partial_{\text{i}} S']  > \frac M2    \Big\} \cup  \Big\{  \mathbf{E}_{x, \varphi'}^{\mathcal{G}_{\Lambda'}, \xi(t)}[\widetilde{\varphi}_{X_{\tau_{\Lambda'}}}  |X_{\tau_{\Lambda'}} \in \partial_{\text{i}} S']  <- \frac M2    \Big\} \bigg).
$$
We first dispense with the $\omega =(\varphi,\widetilde{\varphi})$-dependence coming from the boundary condition $\xi(t)=\xi^{\omega}(t)$, cf. \eqref{eq:dynamic_bc}. Using the lower bound from \eqref{eq:exit_distr}, along with \eqref{eq:hit_dist1}, it follows that, whenever $\Lambda \supset {\Lambda}_3(S,S') = \bigcup_{x \in S} \Lambda_{0}(x,S')$, a finite subset of $\mathbb{Z}^d$, and for all $f: S' \to \mathbb{R}$, with $f^+=f\vee0$, $\ell^2 = \ell^2(\partial_{\text{i}}S')$, noting that $\mathbf{P}_{x, \varphi'}^{\mathcal{G}_{\Lambda'}, \xi(t)}[X_{\tau_{\Lambda'}} \in \partial_{\text{i}} S']= \mathbf{P}_{x, \varphi'}^{\mathcal{G}_{\Lambda'}, \xi(t)}[H_{S'}< H_{\Lambda^c}]$,
\begin{equation}\label{eq:dec650}
\mathbf{E}_{x, \varphi'}^{\mathcal{G}_{\Lambda'}, \xi(t)}[{f}_{X_{\tau_{\Lambda'}}}  |X_{\tau_{\Lambda'}} \in \partial_{\text{i}} S'] \leq \frac{c_{7} \sigma_x^*(S') \langle G^*\nu_x^{S'}, f^+ \rangle_{\ell^2}}{c_5P_x^*[H_{S'}< \infty] }   \leq c_{8} \sigma^*  \langle G^* \bar{\nu}_x^{S'}, f^{+}  \rangle_{\ell^2},
\end{equation}
where $\bar{\nu}_x^{S'}(\cdot)= {\nu}_x^{S'}(\cdot)/ P_x^*[H_{S'}< \infty] $ is the normalized entrance distribution, a probability measure for every $x$, cf. \eqref{eq:hit_dist0}. Note that \eqref{eq:dec650} is uniform in $\xi(t)$ (and hence in $t$) and $\varphi'$. Substituting \eqref{eq:dec650} for $f=\varphi$, along with the lower bound $\mathbf{E}_{x, \varphi'}^{\mathcal{G}_{\Lambda'}, \xi(t)}[{f}_{X_{\tau_{\Lambda'}}}  |X_{\tau_{\Lambda'}} \in \partial_{\text{i}} S'] \geq -c_{8} \sigma^*  \langle G^* \bar{\nu}_x^{S'}, f^{-}  \rangle_{\ell^2}$, obtained similarly, into the previous display, and using that $\widetilde{\varphi}$ has the same law as $-\widetilde{\varphi}$ under $Q_{\Lambda}^0$, which is due to the symmetry of $V_X$, cf.  \eqref{eq:sym}, and the fact that $Q_{\Lambda}^0$ imposes $0$-boundary condition, one arrives at
\begin{equation}\label{eq:dec652.0}
Q_{\Lambda}^0[G_{\Lambda,M}^c] \leq 2 |S| \sup_{x \in S} \mu_{\Lambda}^0 \big(\langle G^* \bar{\nu}_x^{S'}, \varphi^{+}  \rangle_{\ell^2} > c_{9} M/\sigma^*\big), \text{ for $\Lambda \supset {\Lambda}_3(S,S')$ },
\end{equation}
where $c_{9} = 1/2c_{8}$. Consider a fixed $x \in S$. By letting $K$ range over all possible outcomes of the set $\{ y \in \partial_{\text{i}} S' ; \, \varphi_y \geq 0\}$ and writing $(G^*\bar{\nu}_x^{S'})_{|_K}(\cdot)= 1_{\{ \cdot \in K \}}G^*\bar{\nu}_x^{S'}(\cdot)$,  one readily obtains, for $x \in S $ and $\lambda>0$, 
\begin{equation}\label{eq:dec652}
\mu_{\Lambda}^0 \big(\langle G^*\bar{\nu}_x^{S'}, \varphi^{+}  \rangle_{\ell^2} > c_{9} M/\sigma^*\big) \leq 2^{| \partial_{\text{i}} S'|} e^{-\lambda c_{9}\frac{M}{\sigma^*}} \sup_{K \subset \partial_{\text{i}} S'} \mathbb{E}_{\mu_{\Lambda}^0}\Big[ e^{\lambda  \langle (G^*\bar{\nu}_x^{S'})_{|_K}, \varphi  \rangle_{\ell^2}}\Big].
\end{equation}
The Brascamp-Lieb bound \eqref{eq:BL} (note that $\varphi$ is centered) then yields, for $x \in S$ and $K \subset \partial_{\text{i}} S'$, 
\begin{equation}\label{eq:dec653}
\mathbb{E}_{\mu_{\Lambda}^0}\Big[ e^{\lambda  \langle(G^*\bar{\nu}_x^{S'})_{|_K}, \varphi  \rangle_{\ell^2}}\Big]  \leq e^{\frac12 \lambda^2 \text{var}^{*}_{\Lambda}(\psi_K(x))}, \text{ with } \psi_K(x) = \langle (G^*\bar{\nu}_x^{S'})_{|_K}, \varphi  \rangle_{\ell^2},
\end{equation}
where $\text{var}^{*}_{\Lambda}$ denotes variance with respect to the law $\mu^{*}_{\Lambda}$ under which $\varphi$ is distributed as a centered Gaussian field with covariances $\mathbb{E}_{\mu^{*}_{\Lambda}}[\varphi_x\varphi_y]= c_{10} g^*_{\Lambda}(x,y)$, for all $x,y$, and $g^*_{\Lambda}$ is the Green function of simple random walk on $\mathcal{G}$ killed outside $\Lambda$. To bound $\text{var}^{*}_{\Lambda}(\psi_K(x))$, we will harness the fact that, for all $x,z\in \mathbb{Z}^d$, 
$$
\sum_y {\nu}_x^{S'}(y) g^*(y,z) \stackrel{\eqref{eq:hit_dist0}}{=} E_x^*[g^*(X_{H_{S'}},z)1\{ H_{S'}< \infty\}] = g^*(x,z)-g^*_{\mathbb{Z}^d\setminus S'}(x,z) \leq g^*(x,z),
$$
where the penultimate step follows from the strong Markov property at time $H_{S'}$. Since $\psi_K(x)$ is a centered Gaussian under $\mu^{*}_{\Lambda}$, we obtain, for all $x \in S$, $K \subset \partial_{\text{i}}S'$, and $\Lambda \supset \Lambda_3(S,S')$,
\begin{equation}\label{eq:varbound00}
\begin{split}
\text{var}^{*}_{\Lambda}(\psi_K(x))&= \mathbb{E}_{\mu^{*}_{\Lambda}}[\psi_K(x)^2] = c_{10}\sum_{y,z\in K} (G^*\bar{\nu}_x^{S'})(y) (G^*\bar{\nu}_x^{S'})(z) g^*_{\Lambda}(y,z)\\
&\leq c_{10} \sum_{y,z,v,w\in \partial_{\text{i}}S'} g^*(y,v)\bar{\nu}_x^{S'}(v) g^*(z,w)\bar{\nu}_x^{S'}(w) g^*_{\Lambda}(y,z)\\
&\leq  c_{10} \sum_{y,z\in \partial_{\text{i}}S'}\frac{g^*(x,y)g^*(x,z)}{P_x^*[H_{S'}< \infty]^2}g_{\Lambda}^*(y,z)  
\leq c_{11} \Big[\frac{\sigma^* |\partial_{\text{i}}S'|}{\text{cap}^*(S')}\Big]^2, 
\end{split}
\end{equation}
using in the last step that $g_{\Lambda}^*(\cdot,\cdot) \nearrow g^*(\cdot,\cdot) \leq c'$, as $\Lambda \nearrow \mathbb{Z}^d$, together with the fact that, for all $x \in S$, cf. above \eqref{eq:hit1}, 
\begin{equation*}
\sup_{y \in  \partial_{\text{i}}S'} \frac{g^*(x,y)}{P_x^*[H_{S'}< \infty]} = \sup_{y \in  \partial_{\text{i}}S'} \frac{g^*(x,y)}{\sum_{z \in  S'} g^*(x,z)e^*_{S'}(z)} \leq \frac{\sigma_x^*(S')}{\text{cap}^*(S')}.
\end{equation*}
Substituting \eqref{eq:varbound00} into \eqref{eq:dec652}, \eqref{eq:dec653}, optimizing over $\lambda$, we get, for all $x \in S$, any $(\Lambda_n)_{n\geq 0}$ as appearing in the lemma, and any sequence $M_n > 0$, for all $n$, letting $n \to \infty$,
$$
\liminf_n \mu_{{\Lambda_n}}^0 \big(\langle G^*\bar{\nu}_x^{S'}, \varphi^{+}  \rangle_{\ell^2} > c_{9} M_n/\sigma^*\big) \leq 2^{| \partial_{\text{i}} S'|} e^{-c_{12} (\frac{\text{cap}^*(S') M_{\infty}}{(\sigma^*)^2 |\partial_{\text{i}}S'| })^2},
$$
with $c_{12}=c_{9}^2/2c_{11}$. In view of \eqref{eq:dec652.0}, 
this yields \eqref{eq:dec649}.
\end{proof}

With Proposition \ref{L:monot} and Lemma \ref{L:error} at hand, we proceed to the proof of the decoupling inequality, Theorem \ref{P:dec_ineq}. 

\bigskip
\noindent \textit{Proof of Theorem \ref{P:dec_ineq}.}
Let $\Lambda \subset \subset \mathbb{Z}^d$ satisfy the conditions in \eqref{eq:dec:A_sets}. Applying Proposition \nolinebreak \ref{L:monot} and minding that $f  \geq 0$ is measurable with respect to $\mathcal{F}_{S'} = \sigma(\varphi_x, x\in S')$, we obtain, for all $M\leq M_{\varepsilon, \Sigma_{\Lambda}}$,
\begin{equation} \label{eq:dec9}
\begin{split}
&\mathbb{E}_{\mu_{\Lambda}^0}[1_{A^h}\cdot f] = \mathbb{E}_{\mu_{\Lambda}^0}\Big[ \mathbb{E}_{\mu_{\Lambda}^0}[1_{A^h}  |  \mathcal{F}_{S'}]\cdot f\Big]  \stackrel{\eqref{eq:dec10.1}}{=}\mathbb{E}_{Q_{\Lambda}^0}[ Z^h(\varphi_{S'}) \cdot f(\varphi)] \\
& \stackrel{\eqref{eq:sprinkling}}{\leq} \mathbb{E}_{Q_{\Lambda}^0}[ Z^{h-\varepsilon}(\widetilde{\varphi}_{S'})  \, 1_{G_{\Lambda, M}} \cdot f(\varphi) ] + \mathbb{E}_{Q_{\Lambda}^0}[ f(\varphi) \cdot 1_{G_{\Lambda, M}^c}] \\
&\stackrel{\eqref{eq:dec10.1}}{\leq} \mu_{\Lambda}^0(A^{h-\varepsilon} )\mathbb{E}_{ \mu_{\Lambda}^0}[ f] + \Vert f \Vert_{L^{\infty}(\Omega)} \cdot  Q_{\Lambda}^0[G_{\Lambda, M}^c],
\end{split}
\end{equation}
using independence of $\varphi$ and $\widetilde{\varphi}$ in the last step. Hence, \eqref{eq:dec_ineq} follows at once with $\Lambda=\Lambda_n$, $M=M_n = M_{\varepsilon, \Sigma_{\Lambda_n}}$, cf. \eqref{eq:decM}, upon letting $n \to \infty$ in \eqref{eq:dec9} and using \eqref{eq:dec649}, noting that $M_{\infty}= \limsup_n M_n = \varepsilon (\Sigma^{-1} +1) $, on account of \eqref{eq:decM} and \eqref{eq:def_sigma_inf}.
\hfill $\square$

\newpage

\begin{remark} Although this will not be needed below, let us point out that one can derive a  better bound in Lemma \ref{L:error}, by estimating the variance in \eqref{eq:varbound00} more carefully. Indeed, taking for simplicity $S' = B_L=  B(0,L)$, with $S_L= \partial_{\text{i}} B_L$, one has
\begin{equation*}
\begin{split}
\sum_{y,z\in S_L}  g^*_{\Lambda}(y,z) \leq c \sum_{y \in S_L}\sum_{n=1}^L|\{ z \in S_L; |z-y|=n \}| n^{-(d-2)} \leq c' \sum_{y \in S_L}\sum_{n=1}^L 1 = O(L^d), \text{ as }L \to \infty
\end{split}
\end{equation*}
which is far better than the crude bound $|\partial_{\text{i}}S'|^2 = O(L^{2d-2})$. In particular (for this choice of $S'$), since $\text{cap}^*(B_L) \geq cL^{d-2}$, this gives $\text{var}^{*}_{\Lambda}(\psi_K(x)) \leq c(\sigma^*)^2 \cdot O(L^{d-2(d-2)})$, and the last term is bounded uniformly if $d \geq 4$. \hfill $\square$
\end{remark} 

\section{Applications} \label{S:apps}

We now explore some of the consequences of Theorem \ref{P:dec_ineq}, by investigating connected components of level sets of $\varphi$, and deduce, among other things, the results \eqref{eq:intro5}, \eqref{eq:intro6}, mentioned in the introduction. We will ultimately show that our decoupling inequality \eqref{eq:dec_ineq} can be made to fit the general framework of \cite{DRS14.3}, \cite{DRS14.2}, \cite{PRS15} and \cite{Sa14}, but, as it turns out, this requires improving on the error term $\delta_{S,S'}(\varepsilon, \Sigma)$ in a certain sense. To this end, we first show in Theorem \ref{T:subcrit} below that connected components of sites where $\varphi$ exceeds a sufficiently large height decay rapidly as a function of their diameter. This result, which is interesting in its own right, involves a by now standard renormalization argument, akin to the one developed in \cite{Sz12a} or \cite{RoS13}, by which Theorem \ref{P:dec_ineq} is used as one-step renormalization, see Lemma \ref{L:renorm1}. The only issue, which is key in passing from one length scale $L_n$ to the next, is to keep the sprinkling parameter $\varepsilon$ as small as possible while retaining good control on the error term on the right-hand side of \eqref{eq:dec_ineq}, which translates into considering sets $S, S'$ with sufficiently small cross section $\Sigma$.  

The resulting decoupling inequalities are stated below in Corollary \ref{C:DI1}. These are sufficiently strong to prove the desired connectivity function of upper level-sets $E^{\geqslant h}=\{ x \in \mathbb{Z}^d; \, \varphi_x \geq h\}$, decay stretched exponentially as a function of distance, at large heights $h$. Using this supplementary information, we will retroactively ensure that, up to a small error (in terms of $|S|$, $|S'|$), one can prevent $(\varphi,\widetilde{\varphi})$ to land in $G_{\Lambda,M}^c$ in Lemma \ref{L:monot}, and thus avoid the use of Lemma \ref{L:error} altogether. The resulting inequality is stated below in Theorem \ref{T:dec_ineq2}. It provides the crucial tool which enables us to resort to \cite{DRS14.3}, \cite{DRS14.2}, \cite{PRS15} and \cite{Sa14}, to deduce a host of results regarding the geometry of $E^{\geqslant h}$, when $h$ is sufficient small, see Theorems \ref{T:chem_dist}, \ref{T:RW_perc} and Remark \ref{R:moreresults} below. 

We now introduce the renormalization scheme, and consider a sequence of length scales
\begin{equation}\label{eq:renorm1}
L_{n+1} = 100\cdot R^n, \, n \geq 0, \text{ with } R \geq 100^2.
\end{equation}
We will encounter families of events $A_m$ which are naturally indexed by the leaves $m \in T^{(n)}$ of a binary tree of depth $n$. This requires a small amount of notation. Denote by $T_n = \bigcup_{0\leq k \leq n}T^{(k)}$, $T^{(k)}= \{0,1 \}^k$ the canonical binary tree of depth $n$. For $m \in T^{(k)}$, we write $m0$ and $m1$ for the children of $m$ in $T^{(k+1)}$. A map $\tau:T_n\to \mathbb{Z}^d$ will be called \textit{proper embedding with base $x$} if $\tau(\emptyset)= x$,
\begin{equation}\label{eq:renorm2}
 \tau(m) \in L_{n-k}\mathbb{Z}^d \text{ for } m \in T^{(k)}, 0\leq k \leq n
\end{equation} 
and, with the notation $\widetilde{B}_{\tau(m)} = B(\tau(m), 10 L_{n-k})$, if $m \in T^{(k)}$, one has, for all $0\leq k < n$, $m \in T^{(k)}$ and $i=1,2$,
\begin{equation}\label{eq:renorm3}
\widetilde{B}_{\tau(mi)} \subset \widetilde{B}_{\tau(m)} \text{ and } d_{\mathcal{G}}(\widetilde{B}_{\tau(m0)}, \widetilde{B}_{\tau(m1)}) \geq L_{n-k}/100,
\end{equation} 
with $d_{\mathcal{G}}$ denoting graph distance on $\mathcal{G}$. The set of all proper embeddings of $T_n$ with base $x$ will be denoted by $\Xi_{n,x}$. A sequence $(A_{\tau, m})_{m \in T^{(n)}}$ of events (each part of $\widehat{\Omega}$) will be called $\tau$-adapted if
\begin{equation}\label{eq:renorm4}
A_{\tau, m} \in \sigma(Y_x, x \in \widetilde{B}_{\tau(m)}), \text{ for all } m \in T^{(n)}.
\end{equation} 
Finally, we suppose, cf. above Theorem \ref{P:dec_ineq}, that $\mu \in \mathscr{W}$.
A first consequence of Theorem~\ref{P:dec_ineq} and the above choices is the following
\begin{lem}[One-step renormalization]\label{L:renorm1} $(x_0 \in \mathbb{Z}^d$, $n \geq 0$, $\tau \in \Xi_{n+1,x_0}$, $h \in \mathbb{R}$, $R \geq c_{13})$

\medskip
\noindent For all $\varepsilon \geq \varepsilon_n = c_{14}[2^{2n}/ {R}^{n+1}]^{1/2}$, and all families $(A_{\tau, m})_{m \in T^{(n+1)}}$ of increasing, $\tau$-adapted events
\begin{equation} \label{eq:renorm5}
\mu\Big[ \bigcap_{m \in T^{(n+1)}}A_{\tau, m}^h \Big] \leq \mu\Big[ \bigcap_{m \in T^{(n)}}A_{\tau, 0m}^{h-\varepsilon} \Big]  \mu\Big[ \bigcap_{m \in T^{(n)}}A_{\tau, 1m}^{h-\varepsilon}  \Big] + \delta_n,
\end{equation}
where $\delta_n= e^{-c_{15}\sqrt{R}^{n+1}}$.
\end{lem}
\begin{proof} We apply Theorem \ref{P:dec_ineq} with $A^h = \bigcap_{m \in T^{(n)}}A_{\tau, 0m}^{h}$, $f= 1_{\bigcap_{m \in T^{(n)}}A_{\tau, 1m}^{h}}$, so that, in particular, $\mathbb{E}_{\mu}[1_{A^h} f] = \mu[ \bigcap_{m \in T^{(n+1)}}A_{\tau, m}^h]$. Accordingly, we set $S = \bigcup_{m \in T^{(n)}} \widetilde{B}_{\tau(0m)}$, $S' = \bigcup_{m \in T^{(n)}} \widetilde{B}_{\tau(1m)}$. By adaptedness, cf. \eqref{eq:renorm4}, $A \in \sigma(Y_x, x \in S)$, and $f(\varphi) \in \sigma (\varphi_x,  x\in S' )$. Moreover, inductively, by \eqref{eq:renorm3}, we see that $S \subset \widetilde{B}_{\tau(0)}$, while $S \subset \widetilde{B}_{\tau(1)}$, and therefore
\begin{equation} \label{eq:renorm6}
100L_{n+1} \geq d_{\mathcal{G}}(S,S') \geq L_{n+1}/100.
\end{equation}
We now derive a suitable upper bound for the cross-section $\Sigma= \Sigma(S,S', (\Lambda_k)_{k \geq 0})$ defined in \eqref{eq:def_sigma_inf}. By possibly passing to a subsequence, we may assume that $\Sigma= \lim_k \Sigma(S,S', \Lambda_k)$ and that the limit is monotone. Hence, for all $\Lambda_k \supset  \bigcup_{x \in S}\Lambda_0(x,S')$, cf. Corollary \ref{C:exit_dist}, $x \in S$, and $R \geq c$, setting $\Lambda_k'=\Lambda_k \setminus S$, noting that $\sigma_x^*(S') \leq c$ for all $x \in S$, which follows from \eqref{eq:GF_bounds}, \eqref{eq:sigma} and \eqref{eq:renorm6}, and using that $W: x \mapsto \frac{x}{1-x}$ is continuous and increasing on~$[0,1)$, we infer that for all $R \geq c$,
\begin{equation}\label{eq:renorm7}
\begin{split}
\Sigma & \stackrel{\eqref{eq:invisibility_cond}}{\leq} \lim_k \sup_{x \in S }W \Big(\sup_{\varphi, \xi} \mathbf{P}_{x,\varphi}^{\mathcal{G}_{\Lambda_k'}, \xi}[H_{S'} < H_{\Lambda_k^c}]\Big) \stackrel{\eqref{eq:exit_distr}}{\leq}  \sup_{x \in S } W\big( c_6\sigma_x^*(S') P_x^*[H_{S'}< \infty] \big)\\
&\leq W\big(c_{16} P_x^*[H_{S'}< \infty]\big) \leq 2c_{16} c_{17} \Big(\frac2R \Big)^{n+1},
\end{split}
\end{equation}
using \eqref{eq:renorm1}, \eqref{eq:renorm6} in the last step, along with the estimate $$
P_x^*[H_{S'}< \infty] \leq c_{17} d_{\mathcal{G}}(S,S')^{-(d-2)} 2^n,
$$
 which follows from \eqref{eq:hit1}, \eqref{eq:GF_bounds} and the fact that $\text{cap}(\cdot)$ is subadditive (note that this last bound also implies that $c_{16}P_x^*[H_{S'}< \infty] \leq 1/2$, whenever $R \geq c$, which is implied in the last step of \eqref{eq:renorm7}). The claim \eqref{eq:renorm5} then follows immediately from \eqref{eq:dec_ineq}, upon noticing that, for $R \geq c$, since $\Sigma^{-1} \geq c_{18} (\frac R2 )^{n+1}$ by \eqref{eq:renorm7}, and $| \partial_{\text{i}} S'| \leq |S'| = |S|= c_{19}2^n$, $\sigma^*(S,S')\leq c_{20}$, the error term $\delta_{S,S'}(\varepsilon,\Sigma)$ defined in \eqref{eq:eta_error} can be bounded, whenever $\varepsilon^2 \geq \varepsilon_n^2 = |S|100^{d-1} c_{20}/(c_2 c_{18}^2 \text{cap}^*(B_{100}) {(R/2)}^{n+1})$, as follows:
\begin{equation*}
\begin{split}
\delta_{S,S'}(\varepsilon,\Sigma)&\leq c_1 \exp \Big\{ 2|S|- c_2 \varepsilon_n^2 \frac{\text{cap}^*(B_{100})}{c_{20} 2^n 100^{d-1}}  {c_{18}^2} \Big(\frac R2 \Big)^{2(n+1)} \Big\} \\
&= c_1  \exp \Big\{ 2|S| \Big(1- \Big(\frac R4\Big)^{n+1} \Big) \Big\} \leq  c_1  \exp \Big\{ - \frac{c_{19}^2}{2} L_0^d \sqrt{R}^{n+1}\Big\} \equiv \delta_n,
\end{split}
\end{equation*}
whenever $R \geq c$.
\end{proof}
Lemma \ref{L:renorm1} can be iterated, yielding the following
\begin{cor}[Decoupling inequalities I] $(x_0\in \mathbb{Z}^d$, $n \geq 0$, $\tau \in \Xi_{n,x_0}$, $h_0 \in \mathbb{R}$, $R \geq c_{13})$
\label{C:DI1}

\medskip
\noindent For all families $(A_{\tau, m})_{m \in T^{(n)}}$ of increasing, $\tau$-adapted events, with $h_{n}^+= h_0 + \sum_{0\leq k < n} \varepsilon_n$, $h_{\infty}^+= \lim_n h_n^+$ and $ \varepsilon_n$ as defined above \eqref{eq:renorm5},
 \begin{equation}\label{eq:renorm8}
 \mu\Big[ \bigcap_{m \in T^{(n)}}A_{\tau, m}^{h_{\infty}^+} \Big] \leq   \mu\Big[ \bigcap_{m \in T^{(n)}}A_{\tau, m}^{h_{n}^+} \Big] \leq \prod_{m \in T^{(n)}}\Big( \mu[A_{\tau, m}^{h_{0}}] + \bar{\delta}(R) \Big),
 \end{equation}
 where
  \begin{equation}\label{eq:renorm9}
  \bar{\delta}(R) \stackrel{\textnormal{def.}}{=} \sum_{n \geq 0} (\delta_n)^{\frac{1}{2^{n+1}}} 
  .
 \end{equation}
\end{cor}
\begin{proof}
The first inequality in \eqref{eq:renorm8} is due to monotonicity. The second one follows from
$$
\mu\Big[ \bigcap_{m \in T^{(n)}}A_{\tau, m}^{h_{n}^+} \Big] \leq \prod_{m \in T^{(n)}}\Big( \mu[A_{\tau, m}^{h_{0}}] + \sum_{k < n}  (\delta_k)^{\frac{1}{2^{k+1}}}\Big),
$$
which is easily shown by induction over $n$, using \eqref{eq:renorm5}. The fact that $\bar{\delta}(\cdot)$ is finite is plain from the definition of $\delta_n$ below \eqref{eq:renorm5}.
\end{proof}

\begin{remark}
As for Theorem \ref{P:dec_ineq}, cf. Remark \ref{R:dec_ineq1}, 4), \eqref{eq:renorm8} has the following complement. Under the assumptions of Corollary \ref{C:DI1}, for any family $(A_{\tau, m})_{m \in T^{(n)}}$ of \textit{decreasing}, $\tau$-adapted events, setting $h_{n}^-= h_0 - \sum_{0 \leq k < n} \varepsilon_n$, $h^-_{\infty}= \lim_n h_n^-$, one has
\begin{equation}\label{eq:renorm8.99}
 \mu\Big[ \bigcap_{m \in T^{(n)}}A_{\tau, m}^{h_{\infty}^-} \Big] \leq   \mu\Big[ \bigcap_{m \in T^{(n)}}A_{\tau, m}^{h_{n}^-} \Big] \leq \prod_{m \in T^{(n)}}\Big( \mu[A_{\tau, m}^{h_{0}}] + \bar{\delta}(R) \Big).
\end{equation}
This follows by considering the flipped events $\overline{A}_{\tau, m}=\{ \omega^f; \omega \in A_{\tau, m}\}$, cf. Remark \ref{R:dec_ineq1}, 4), noting that  $\mu[ \bigcap_{m \in T^{(n)}}A_{\tau, m}^{h} ] = \mu[ \bigcap_{m \in T^{(n)}}\overline{A}_{\tau, m}^{-h} ]$, for all $n \geq 0$ and $h \in \mathbb{R}$, and applying \eqref{eq:renorm8} to the events $\overline{A}_{\tau, m}$, $m \in T^{(n)}$, with the sequence $(-h_n^-)_{n\geq 0}$. \hfill $\square$
\end{remark}

We are now ready to state our first result regarding the geometry of $E^{\geqslant h}= \{x\in \mathbb{Z}^d;\,  \varphi_x \geq h\}$. For $K,K' \subset \mathbb{Z}^d$, we denote by $\{ K \stackrel{\geqslant h}{\longleftrightarrow}K' \} \subset \Omega$ the event that $E^{\geqslant h}$ contains a nearest-neighbor path in $\mathcal{G}$ intersecting both $K$ and $K'$. 

\begin{thm} $(\mu \in \mathscr{W})$
\label{T:subcrit}
\begin{equation}\label{eq:renorm10}
\begin{split}
h^+ = h^+(\mu)  \stackrel{\text{def.}}{=}  \inf\Big\{ h \in \mathbb{R} ; \, &\text{there exist $c(h), c'(h)$ and $\rho(h)>0$ s.t. } \\
&\mu(x \stackrel{\geqslant h}{\longleftrightarrow} y ) \leq c(h)e^{-c'(h)|x-y|^{\rho}}, \text{ for all }x,y \in \mathbb{Z}^d\Big\} < \infty.
\end{split}
\end{equation}
\end{thm}
With Corollary \ref{C:DI1} at hand, the proof is similar to Theorem 4.1 in \cite{Sz12a}, using a certain set of so-called cascading events. We include the short proof for completeness.
\begin{proof}
Let $x\in \mathbb{Z}^d$. One defines the set $\Xi_{n,x}^*$, for $n \geq 0$, consisting of all proper embeddings $\tau \in \Xi_{n,x}$, cf. \eqref{eq:renorm2} and \eqref{eq:renorm3}, with the additional property that, for every $1 \leq k < n$ and $m \in T^{(k)}$, $B(\tau(m0), L_{n-k-1}) \cap B(\tau(m), L_{n-k}) \ne \emptyset$ and $B(\tau(m1), L_{n-k-1}) \cap B(\tau(m), 3L_{n-k}/2) \ne \emptyset$. On account of \eqref{eq:renorm1}, $\Xi_{n,x}^*$ is not empty, and moreover $|\Xi_{n,x}^*| \leq (c_{21}R^{d-1})^{2^n}$. Consider the events $E^h_{L,y}= \{ B_L(y) \stackrel{\geqslant h}{\leftrightarrow} \partial B_{2L}(y) \}$,  for $y \in \mathbb{Z}^d$, $L \geq 1$ and define $A_{\tau, m}^h = E^h_{L_0, \tau(m)}$, for $m \in T^{(n)}$, $\tau \in \Xi_{n,x}^*$. 
Since any path connecting $B_L(y)$ to $ \partial B_{2L}(y)$ must intersect both $\partial_{\text{i}}B_L(y)$ and also $\partial_{\text{i}}B_{3L/2}(y)$, one obtains inductively that $E^h_{L_n,x} \subset \bigcup_{\tau \in  \Xi_n^*} \bigcap_{m \in T^{(n)}} A_{\tau, m}^h$, for all $n \geq 0$, $h \in \mathbb{R}$, and thus, for $R \geq c_{13}$, $n\geq 0$, $h_0 \in \mathbb{R}$, using \eqref{eq:renorm8},
\begin{equation}
\label{eq:renorm11}
\mu[E^{h_{\infty}^+}_{L_n,x}] \leq |\Xi_{n,x}^*| \cdot \mu\Big[ \bigcap_{m \in T^{(n)}}A_{\tau, m}^{h_{n}^+} \Big] \leq  \Big[c_{21}R^{d-1}\Big(\sup_{z\in \mathbb{Z}^d}\mu[E^{h_0}_{L_0, z}] + c_{22}e^{-c_{15} \sqrt{R}}\Big)\Big]^{2^n},
\end{equation}
where we used that $\bar{\delta}(R) \leq c_{22} e^{-c_{15} \sqrt{R}}$, which follows readily from \eqref{eq:renorm9} and the definition of $\delta_n$ below \eqref{eq:renorm5}. Now, first choose $R \geq c_{13}$ such that $c_{21}c_{22}R^{d-1} e^{-c_{15} \sqrt{R}}\leq 1/4$, and then $h_0$ large enough such that $\Big(\mu[E^{h_0}_{L_0, z}] \leq (4 c_{21}R^{d-1})^{-1}$, for all $z$, by considering, for instance, the maximum of $\varphi$ in a box around $z$ of radius $L_0\, (=100)$, which has to exceed $h_0$ on the event $E^{h_0}_{L_0, z}$, applying a union bound, and using the Brascamp-Lieb inequality \eqref{eq:BL}. With these choices, and noting that $h_{\infty}^+ < \infty$, cf. above \eqref{eq:renorm8}, it follows from \eqref{eq:renorm11} that $\mu[E^{h_{\infty}^+}_{L_n,x}] \leq 2^{-2^n}$, for all $n \geq 0$, and this implies readily that $\mu(x \stackrel{\geq h_{\infty}^+}{\longleftrightarrow} y )$ has stretched exponential decay in $|x-y|$, thus yielding \eqref{eq:renorm10}.
\end{proof}

\begin{remark} (Level-sets of $|\varphi|$).
By adapting the arguments of \cite{Ro14}, Section 4, one can show an analogue of Theorem \ref{T:subcrit}, whereby $\mu(x \stackrel{\geqslant h}{\longleftrightarrow} y )$, the connectivity function of the one-sided level set $E^{\geqslant h}=\{ x \in \mathbb{Z}^d;\, \varphi_x \geqslant h\}$, is replaced by that of $\{x \in \mathbb{Z}^d; \, |\varphi_x| \geqslant h\}$. \hfill $\square$
\end{remark}
We now use the fact our decoupling inequality provides a tool to control excursion sets above high levels to avoid ``falling out of'' the event $G_{\Lambda,M}^c$ in \eqref{eq:dec15}, thus avoiding the use of Lemma \ref{L:error} altogether and obtaining a different error term in \eqref{eq:dec_ineq}. A similar procedure was already necessary in the Gaussian case, see \cite{DRS14.3}. 

\begin{thm} $(\mu \in \mathscr{W})$
\label{T:dec_ineq2}

\medskip
\noindent There exists $K, \alpha, c_{23} \in (0, \infty)$ such that, for all $\varepsilon \in (0,1/2)$, $L \geq 1$, $h \in \mathbb{R}$, $x_1, x_2 \in\mathbb{Z}^d$ satisfying 
\begin{equation}
\label{eq:renorm20}
|x_1 -x_2| \geq \varepsilon^{-K} L,
\end{equation}
all increasing events $A  \in \sigma(Y_x, x \in B(x_1,10L))$, and bounded $f \geq 0$ satisfying $f \in\sigma(\varphi_x, x \in B(x_2,10L))$, one has
\begin{equation}
\label{eq:renorm21}
\mu(1_{A^h} \cdot f) \leq \mu({A^{h-\varepsilon}} ) \cdot \mu( f) + c_{23} \Vert  f\Vert_{\infty} e^{-L^{\alpha}}
\end{equation}
\end{thm}

\begin{proof}
Let 
\begin{equation}\label{eq:phi_hat}
\widehat{\varphi} =\varphi- \widetilde{\varphi},
\end{equation}
 with $ \omega = (\varphi, \widetilde{\varphi})$ having law $\mu \otimes \mu$. Denote by $\widehat{\mathbb{P}}$ the law of $\widehat{\varphi} $. Then, defining $\widehat{E}_{L,x}^h$ to be the event that $B(x,L)$ is connected to $\partial B(x,2L)$ by a path in $\{x \in \mathbb{Z}^d; \, \widehat{\varphi}_x \geq h \}$, and mimicking the proof of Theorem \ref{T:subcrit} with these events, one deduces that the quantity $\widehat{h}_+$, defined as in \eqref{eq:renorm10}, but with $\widehat{\mathbb{P}}$ instead of $\mu$, is finite (in words, $\widehat{h}_+$ is smallest such that level sets of $\widehat{\varphi}$ have stretched exponential connectivity decay, for any height $h > \widehat{h}_+$). The proof is  identical to that of Theorem \ref{T:subcrit}, one merely needs to observe that $\widehat{\mathbb{P}}[\widehat{E}_{L_0,0}^{h_0}] \to 0$ as $h_0 \to \infty$, for any fixed value of $L_0 \in \mathbb{R}$.

Next, consider two sets $S$, $S'$ with
\begin{equation}
\label{eq:renorm22}
S=B(x_1, 10L), \ S' \subset B(x_2, 20L).
\end{equation}
We seek an upper bound for $\Sigma(S,S')$ uniform in $L\geq 1$. First note that for $\sigma^*(S,S')$ as defined in \eqref{eq:sigma_ref*}, one has, on account of \eqref{eq:GF_bounds} and \eqref{eq:renorm22}, with $x,y$ below ranging over all points in $B(x_1, 10L)$ and $B(x_2, 20L)$, respectively, using that $x\mapsto \frac{1+x}{1-x}$ is increasing in $[0,1)$ and that $\varepsilon \leq 1/2$,
\begin{equation}\label{eq:renorm23.6}
\sigma^* \leq c \frac{\sup_{x,y} |x-y|^{2-d}}{\inf_{x,y} |x-y|^{2-d}} \leq c \frac{ (|x_1-x_2| -c'L)^{2-d}}{(|x_1-x_2|+c'L)^{2-d}} \stackrel{\eqref{eq:renorm20}}{\leq} c \Big(\frac{1+ c'\varepsilon^K}{1-c'\varepsilon^K}\Big)^{d-2} \leq c''.
\end{equation}
Then, for suitable $\widehat{\Lambda} = \widehat{\Lambda}(S,x_2,L) \subset \subset \mathbb{Z}^d$ and all finite $\Lambda \supset \widehat{\Lambda}$, $x \in S$, and $\varphi \in \Omega_{\Lambda}^{\xi}, \xi \in \Omega$, one has, using last-exit decomposition for $P_x^*[H_{S'} < \infty]$, cf. above \eqref{eq:hit1}, and Corollary~\ref{C:exit_dist}, with $\Lambda' =\Lambda \setminus S$,
\begin{equation}
\label{eq:renorm23}
\begin{split}
\mathbf{P}_{x,\varphi}^{\mathcal{G}_{\Lambda'},\xi}[H_{S'} < H_{\Lambda^c} ] \underset{\eqref{eq:renorm23.6}}{\stackrel{\eqref{eq:exit_distr}}{\leq}}  c P_x^*[H_{S'} < \infty] &\leq c\cdot \text{dist}(S,S')^{-(d-2)}\text{cap}(B(x_2, 20L))\\
& \leq c' (\varepsilon^{-K}-50)^{-(d-2)}.
\end{split}
\end{equation}
Thus, for all $K$ sufficiently large, it follows from \eqref{eq:invisibility_cond}, \eqref{eq:renorm23}, in view of assumption \eqref{eq:renorm20}, that $ \sup_{\Lambda \supset \widehat{\Lambda}} \Sigma_{\Lambda}(S,S') \leq \frac{1}{10} \wedge c\varepsilon^{K(d-2)}$, for all $S,S'$ satisfying \eqref{eq:renorm22}, and we may thus arrange, by choosing $K$ large enough, that
\begin{equation}
\label{eq:renorm24}
M_{\varepsilon, \Sigma_{\Lambda}(S,S')} \stackrel{\eqref{eq:decM}}{=} \varepsilon(\Sigma_{\Lambda}^{-1}(S,S')+1) \geq c' \varepsilon^{1-K(d-2)} \geq \widehat{h}_+ + 1, \text { for all $ \varepsilon \in (0, 1/2)$, $\Lambda \supset \widehat{\Lambda} $}.
\end{equation}
Fix such $K$. With a slight abuse of notation, suppose henceforth that $\omega = (\varphi, \widetilde{\varphi})$ is distributed according to $Q_{\Lambda}^0 (= \mu_{\Lambda}^0 \otimes \mu_{\Lambda}^0)$, and assume that $\widehat{\Lambda}$ has been chosen large enough to satisfy $\widehat{\Lambda} \supset B(x_2, 100|x_2-x_1|)$. Let $\mathscr{C}_x$ denote the connected component of $x$ inside $\{ y \in \mathbb{Z}^d; {\varphi}_y - \widetilde{\varphi}_y \geq  \widehat{h}_+ + 1 \}$ (which might be empty), and define the random set
\begin{equation}
\label{eq:renorm25}
\mathscr{S}' = B(x_2,10L) \cup \bigcup_{x \in \partial_{\text{i}}B(x_2,10L) } \partial \mathscr{C}_x,
\end{equation}
(with the convention that $\partial \emptyset = \emptyset$). By definition of $\mathscr{S}' $, and for any $S'$ as in \eqref{eq:renorm22}, all $\Lambda \supset \widehat{\Lambda}$, on the event $\{\mathscr{S}' =  S' \}$, the field $\varphi - \widetilde{\varphi} $ satisfies $(\varphi - \widetilde{\varphi} )_{|_{\partial_{\text{i}}S'}} <  \widehat{h}_+ + 1 $, hence, with $\Lambda' = \Lambda \setminus S'$, 
$$
\mathbf{E}_{x, \varphi'}^{\mathcal{G}_{\Lambda'},\xi}[{\varphi}_{X_{\tau_{\Lambda'}}} - \widetilde{\varphi}_{X_{\tau_{\Lambda'}}}  |X_{\tau_{\Lambda'}} \in \partial_{\text{i}} S'] < \widehat{h}_+ + 1,
$$
for any $x \in S$, and $\varphi' \in \Omega_{\Lambda}^{\xi}, \xi \in \Omega$, and therefore, in view of \eqref{eq:renorm24}, \eqref{eq:dec15},
\begin{equation}
\label{eq:renorm26}
 \{ \mathscr{S}'  =  S'\}  \subset G_{\Lambda, \widehat{h}_+ + 1, S,S'},
\end{equation}
for any $S' \subset  B(x_2,20L)$, $\Lambda \supset \widehat{\Lambda}$. Applying Proposition \ref{L:monot}, we deduce from \eqref{eq:sprinkling} and \eqref{eq:renorm26} that $Z^h(\varphi_{S'})1_{\mathscr{S}'  =  S'} \leq Z^{h-\varepsilon}(\widetilde{\varphi}_{S'})1_{\mathscr{S}'  =  S'}$, with $Z^h(\cdot)$ as defined in \eqref{eq:dec10}. Thus, observing that $f(\varphi)\cdot 1_{\{ \mathscr{S}'  =  S'\}}$ is measurable with respect to ${F}_{S'} \equiv \sigma(\varphi_x, \widetilde{\varphi}_x, x \in S')$, cf. \eqref{eq:renorm25}, we obtain, for all $\Lambda \supset \widehat{\Lambda}$,
\begin{equation}
\label{eq:renorm27}
\begin{split}
&\mu_{\Lambda}^0(1_{A^h} \cdot f) = \mathbb{E}_{Q_{\Lambda}^0}(1_{A^h}(\varphi) \cdot f(\varphi)) \\
&\leq  \sum_{S' \subset B(x_2,20L)} \mathbb{E}_{Q_{\Lambda}^0}[\underbrace{\mathbb{E}_{Q_{\Lambda}^0}[1_{A^h}(\varphi)|\,{F}_{S'}]}_{= Z^h(\varphi_{S'}) }\cdot1_{\mathscr{S}'  =  S'} f(\varphi)] + \mathbb{E}_{Q_{\Lambda}^0}(f(\varphi) \cdot 1_{\mathscr{S}' \cap B(x_2, 20L)^c \neq \emptyset})\\
&\leq \mu_{\Lambda}^0 ({A^{h-\varepsilon}} ) \cdot \mu_{\Lambda}^0 ( f) + \Vert f \Vert_{\infty} \cdot Q_{\Lambda}^0(\mathscr{S}' \cap B(x_2, 20L)^c \neq \emptyset).
\end{split}
\end{equation}
Finally, \eqref{eq:renorm21} follows from \eqref{eq:renorm27} by taking $\Lambda \nearrow \mathbb{Z}^d$ along a suitable sequence, and observing that, by definition of $\widehat{h}_+$, cf. the discussion following \eqref{eq:phi_hat},
\begin{equation*}
\begin{split}
&Q_{\Lambda}^0(\mathscr{S}' \cap B(x_2, 20L)^c \neq \emptyset) \\
&\to \widehat{\mathbb{P}}[B(x_2,10L) \leftrightarrow \partial_{\text{i}}B(x_2,20L) \text{ inside } \{ y \in \mathbb{Z}^d; \widehat{\varphi}_y \geq \widehat{h}_+ +1 \}] \leq ce^{-L^{\alpha}},
\end{split}
\end{equation*}
for all $L \geq 1$ and suitable $\alpha >0$.
\end{proof}
With Theorem \ref{T:dec_ineq2} at hand, we have the necessary tool to apply the recent results of \cite{DRS14.3}, \cite{DRS14.2}, \cite{PRS15} and \cite{Sa14}, which allow for correlated percolation fields with suitably quantified correlations. Their common feature is their reliance on a certain set of assumptions for the occupation field, called \textbf{P1}-\textbf{P3} and \textbf{S1}-\textbf{S2}, which can be found for instance in \cite{PRS15}, p.2. Among these, crucial condition \textbf{P3}, which quantifies the correlations, will follow from Theorem \ref{T:dec_ineq2} above.

Let $\tau_z: \Omega \to \Omega$, $(\tau_z \omega)(\cdot) = \omega(\cdot +z)$, for $\omega \in \Omega$, $z \in \mathbb{Z}^d$, denote the canonical shifts. From now on until the end of this Section, we assume that
\begin{equation}\label{eq:shift_inv}
\mu \in \mathscr{W} \text{ satisfies $\mu(\tau_z^{-1}(A))=\mu(A)$, for all $A \in \mathcal{F}$}.
\end{equation}
Note that shift-invariance of $\mu$ does not seem a-priori clear. Moreover, we assume for simplicity that 
\begin{equation}\label{eq:2bodycondholds}
\begin{split}
\text{$V=\{ V_X\}_X$ is given by the two-body potentials of \eqref{eq:pot2body}.}
\end{split}
\end{equation}
(this is not necessary, but simplifies certain duality arguments we are about to make). For $h \in \mathbb{R}$ and $r \in [1,\infty]$, we denote by $\mathcal{S}^{\geqslant h}_r$ the random set consisting of all sites of $E^{\geqslant h}$ in connected components of $\ell^1$-diameter at least $r$. We say that $h \in \mathbb{R}$ is strongly supercritical for $\mu$ if there exists $\Delta(h) > 0$ such that
\begin{equation}
\label{eq:renorm28}
\mu(\mathcal{S}^{\geqslant h}_L \cap B(0,L) = \emptyset) \leq e^{-(\log L)^{1+\Delta(h)}}\\
\end{equation}
and
\begin{equation}
\label{eq:renorm29}
\begin{split}
\mu\big[&\text{there exist components in $\mathcal{S}^{\geqslant h}_{L/10} \cap B(0,L)$, which} \\
&\text{are not connected in $E^{\geqslant h} \cap B(0,2L)$}\big] \leq e^{-(\log L)^{1+\Delta(h)}}
\end{split}
\end{equation}
and the critical parameter
\begin{equation}
\label{eq:renorm30}
h_- = \sup\{ h \in \mathbb{R}; h \text{ is strongly supercritical for $\mu$} \}
\end{equation}
(with the convention $\sup \emptyset = - \infty$). 

\begin{remark} \label{R:h_-}1) It is not hard to see, using \eqref{eq:shift_inv}, \eqref{eq:renorm28} and \eqref{eq:renorm29}, that $\mathcal{S}^{\geqslant h}_\infty$ is non-empty and connected with probability $1$, whenever $h$ is strongly supercritical for $\mu$. In particular, this implies that $ h_- \leq h_+ (< \infty)$, using \eqref{eq:renorm10}. Moreover, using the decoupling inequality \eqref{eq:renorm8}, adapting the arguments in the proof of Theorem \ref{T:subcrit}, one can show that for all sufficiently small $h$, with $A^{<h}_*(x)$ denoting the event that $0$ is connected to $x$ by a $*$-nearest-neighbor path of vertices in $\{y \in \mathbb{Z}^d; \varphi_y < h\}$, the probability $\mu (A^{<h}_*(x))$ decays stretched exponentially in $|x|$. Along with standard duality arguments, this readily yields that such $h$ is in fact strongly supercritical for $\mu$, and $h_- < -\infty$ follows. \\
\noindent 2) By the preceding remark, one has the string of inequalities $-\infty < h_- \leq h_* \leq h_+ < \infty$, where $h_*$ denotes the critical parameter for percolation of $E^{\geqslant h}$. It is an open question to determine whether $h_- =h_* =h_+$. \hfill $\square$
\end{remark}
In what follows we tacitly view $E^{\geqslant h}$ as a graph, by drawing an edge between any two vertices $x,y \in E^{\geqslant h}$ which are neighbors in $\mathbb{Z}^d$, and denote by $\rho(\cdot, \cdot)$ the graph distance on $E^{\geqslant h}$ (with the convention that $\rho(x, y)= \infty$ if $x$, $y$ belong to different connected components of $E^{\geqslant h}$). Theorem \ref{T:dec_ineq2} has many applications, see Remark \ref{R:moreresults} below, but we highlight the following two results.

\begin{thm} $($Chemical distance$)$.
\label{T:chem_dist}

\medskip
\noindent For all $h < h_-$, there exist $c(h), c'(h), c''(h) \in (0,\infty)$ such that
\begin{equation}
\label{eq:chem_dist}
\mu\bigg[ \bigcap_{x,y \in \mathcal{S}_L^{\geqslant h} \cap B(0,L)} \{ \rho(x,y) \leq c L \} \bigg]  \geq 1-c' e^{-c''(\log L)^{1+\Delta(h)}}.
 \end{equation} 
\end{thm}

The next theorem concerns simple random walk on the percolation cluster $\mathcal{S}_{\infty}^{\geqslant h}$, in the regime $h < h_-$ of strong supercriticality. We endow $\mathcal{S}_{\infty}^{\geqslant h}$ with edge weights
$$
\mu_{x,y}= \begin{cases}
1, & \text{if } x,y \in \mathcal{S}_{\infty}^{\geqslant h}, x\sim y\\
0,& \text{else} 
\end{cases}
\qquad \qquad \mu_x = \sum_y \mu_{x,y}  
$$
and let $X= (X_n)_{n\geq 0}$, resp. $Y= (Y_t)_{t\geq 0}$ be the discrete-time, resp. (constant speed) continuous-time simple random walk on $\mathcal{S}_{\infty}^{\geqslant h}$. That is, $X$ is the Markov chain on $\mathcal{S}_{\infty}^{\geqslant h}$ with transition probabilities $\mu_{x,y}/\mu_x$, and $Y$ is the jump process on $\mathcal{S}_{\infty}^{\geqslant h}$ with generator $\mathscr{L}f(x)= \sum_y \frac{\mu_{x,y}}{\mu_x}(f(y)-f(x))$. The (quenched) laws of $X$ and $Y$ with starting point at $x$ will be denoted by $Q_{\mathcal{S}_{\infty}^{\geqslant h}, x}^X$, resp. $Q_{\mathcal{S}_{\infty}^{\geqslant h}, x}^Y$. 

\begin{thm}
\label{T:RW_perc} For all $h < h_-$, the following hold:\\
\noindent 1) (Quenched invariance principle). For any $T \in( 0, \infty)$, the law of $(B_n(t))_{0 \leq  T }$, with $B_n(t)\stackrel{\text{def.}}{=} n^{-1/2}\{X_{\lfloor tn \rfloor} + (tn-\lfloor tn \rfloor)(X_{\lfloor tn \rfloor+1}- X_{\lfloor tn \rfloor})\}$ on $(C[0,T], \mathcal{F}_T)$, the space of continuous functions from $[0,T]$ to $\mathbb{R}^d$, endowed with its canonical $\sigma$-algebra, under $Q_{\mathcal{S}_{\infty}^{\geqslant h}, x}^X$, converges weakly to the law of an isotropic Brownian motion with zero drift and positive diffusion constant.\\
\noindent 2) (Quenched heat kernel estimates). There exist random variables $(T_x(\varphi))_{x \in \mathbb{Z}^d}$ satisfying
$T_x < \infty$, $\mu(\cdot| 0 \in  \mathcal{S}_{\infty}^{\geqslant h})$-a.s, with tails $\mu (T_x \geq r)\leq c(h)e^{-c'(h)(\log r)^{1+ \Delta(h)}}$, $x \in \mathbb{Z}^d$, such that, $\mu(\cdot| 0 \in  \mathcal{S}_{\infty}^{\geqslant h})$-a.s, for all $x, y \in   \mathcal{S}_{\infty}^{\geqslant h} $ and $t \geq T_x$,
\begin{equation}
\begin{split} \label{eq:HKB}
&q_t(x,y)\leq c(h)t^{-d/2}e^{-c'(h)\frac{\rho(x,y)^2}{t}}, \text{ for all $t \geq \rho(x,y)$}\\
&q_t(x,y)\geq c''(h)t^{-d/2}e^{-c'''(h)\frac{\rho(x,y)^2}{t}}, \text{ for all $t \geq \rho(x,y)^{3/2}$}
\end{split}
\end{equation}
where $q_t(x,y)$ stands for either $Q_{\mathcal{S}_{\infty}^{\geqslant h}, x}^Y[Y_t=y]/\mu_y$ or $p_{\lfloor t\rfloor}(x,y) + p_{\lfloor t\rfloor+1}(x,y)$, with $p_n(x,y)= Q_{\mathcal{S}_{\infty}^{\geqslant h}, x}^X[X_n=y]$, $n \geq 0$.
\end{thm}
\textit{Proof of Theorems \ref{T:chem_dist} and \ref{T:RW_perc}.} Theorem \ref{T:chem_dist} follows from Theorem 2.3 in \cite{DRS14.2}, Theorem \ref{T:RW_perc}, 1) from Theorem 1.1 in \cite{PRS15} and 2) from Theorem 1.6 in \cite{Sa14}, provided the conditions \textbf{S1}-\textbf{S2} and \textbf{P1}-\textbf{P3} appearing, for instance, in Section 1.1 of \cite{Sa14}, can be verified for the family $(\mathbb{P}^u)_{u \in (u_0,\infty)}$, for (fixed, but arbitrarily small) $u_0 \in (0,1)$, where $\mathbb{P}^u$ is defined as the law of $E^{\geqslant h_- - u}$ under ${\mu}$.

Condition \textbf{S1} (local uniqueness) follows immediately from the definition of $h_-$ in \eqref{eq:renorm28}-\eqref{eq:renorm30}. Property $\mathbf{P2}$ requires the function $u\in (u_0,\infty) \mapsto \mu( 0 \in \mathcal{S}_{\infty}^{\geqslant h_- -u})$ to be continuous and positive. Positivity is immediate, cf. Remark \ref{R:h_-}, 1) and continuity follows from the same argument as in the Bernoulli case, cf. the proof of Theorem (8.8) in \cite{Gr99}. Conditions \textbf{P1} (ergodicity of $\mathbb{P}^u$ under lattice shifts) and \textbf{P2} (monotonicity of $u \mapsto \mathbb{P}^u(G)$, for increasing, measurable $G\subset \widehat{\Omega}$) follow immediately from the ergodicity of $\mu$ and the fact that $E^{\geqslant h_- - u} \supset E^{\geqslant h_- - u'}$, for $u> u'$. Finally, the crucial property \textbf{P3} (decoupling), requires that 
\begin{equation*}
\begin{split}
&\text{there exist integers $R_{P}, L_{P} \in( 1, \infty) $, and $\varepsilon_{P}, \chi_{P} \in (0,1)$, along with a function}\\
&\text{$f_{P}: \mathbb{Z}_+\to \mathbb{R}$ satisfying $f_{P}(L) \geq e^{(\log L)^{\varepsilon_{P}}}$, for all $L \geq L_{P}$, such that, for all $L \geq 1$,}\\
&\text{and $x_1,x_2 \in \mathbb{Z}^d$ with $|x_1-x_2| \geq R L$, all increasing events $A_i \in \sigma(Y_x , x\in B(x_i,10L))$,}\\
&\text{decreasing events $B_i \in \sigma(Y_x , x\in B(x_i,10L))$, $i=1,2$, all $\hat{u},u \in (u_0,\infty)$ with}\\
&\text{$\hat{u}\leq  u/(1+R^{-\chi_{P}})$, and all $R \geq R_{P},$}\\[0.2em]
&\qquad \mathbb{P}^{\hat{u}}[A_1 \cap A_2] \leq \mathbb{P}^u[A_1]  \mathbb{P}^u[A_2]+e^{-f_{P}(L)}, \text{ and } \mathbb{P}^{{u}}[B_1 \cap B_2] \leq \mathbb{P}^{\hat{u}}[B_1]   \mathbb{P}^{\hat{u}}[B_2]+e^{-f_{P}(L)}. 
\end{split}
\end{equation*} 
This follows from Theorem \ref{T:dec_ineq2} and Remark \ref{R:dec_ineq1}, upon letting $h= h_- -\hat{u}$, for $\hat{u}\in (u_0,\infty)$, $A\equiv A_1$, $f=1_{A_2^{h_{-}-\hat{u}}}$, with $A_1$, $A_2$ as above, setting $R=\varepsilon^{-K}$, letting $\varepsilon$ vary in $(0,(\frac{u_0}{2})^2]$ and definind $R_{P}= R_{P}(u_0)=(4 u_0^{-2})^K$ (so that $R \geq R_{P}$ for all $\varepsilon \leq (\frac{u_0}{2})^2$), choosing $L_{P}$ sufficiently large such that $f_{P}(L)= L^{\alpha}-\log(c_{23})\geq L$, which satisfies $f_{P}(L)> e^{(\log L)^{1/2}}$, for all $L \geq L_{P}$ and suitable $L_P \geq 1$, and setting $\chi_{P}= (2K)^{-1}$, which yields, for all $\hat{u},u \in (u_0,\infty)$ with $\hat{u}\leq  u/(1+R^{-\chi_{P}})$,
 $$
 u- \hat{u}\geq \frac{uR^{-\chi_{P}}}{1+ R^{-\chi_{P}}} \geq u_0 \frac{\sqrt{\varepsilon}}{2} \geq \varepsilon, \text{ whenever $\sqrt{\varepsilon} \leq \frac{u_0}{2}$}
 $$ 
and therefore $h-\varepsilon =  h_- -\hat{u}-\varepsilon \geq h_- - u$, so that $\mu({A_1^{h-\varepsilon}} ) \leq \mu({A_1^{ h_- - u}} )= \mathbb{P}^u[A_1]$, as desired. \hfill $\square$

\begin{remark}\label{R:moreresults} 1) (Shape Theorem)
By standard methods, \eqref{eq:chem_dist} is known to imply that, upon suitable rescaling, large balls in $\mathcal{S}^{\geqslant h}_{\infty}$ (endowed with the metric $\rho$) have an asymptotically deterministic shape; for a precise statement, see for instance Theorem 2.5 in \cite{DRS14.2}.\\
\nolinebreak 2) The heat kernel bounds in Theorem \ref{T:RW_perc} continue to hold if one replaces $\rho$ by the usual $\ell^1$-graph distance on $\mathbb{Z}^d$, and the exponent $3/2$ in \eqref{eq:HKB} can be replaced by $1+ \varepsilon$, for any $\varepsilon > 0$ (with $T_x$ and all constants depending on $\varepsilon$). Further results which follow directly from the fact that the law of $\mathcal{S}_{\infty}^{\geqslant h}$, for $h<h_- $, satisfies the conditions \textbf{S1}-\textbf{S2} and \textbf{P1}-\textbf{P3}, include (quenched) Harnack inequalities on $\mathcal{S}_{\infty}^{\geqslant h}$ and bounds on gradients of the heat kernel, for all $h< h_-$, see \cite{Sa14} Section 1.3. \hfill $\square$.

\end{remark}

\section{Non-convex perturbations} \label{S:nonconvex}

We now explain how to apply our sprinkling technique to a certain class of non-convex potentials. As will soon become clear, most of the necessary work is implicitly contained in the framework chosen in Section \ref{S:1}, see in particular \eqref{eq:H_pots} and \eqref{eq:smooth}-\eqref{eq:elliptic}. Specifically, we consider non-convex modifications of a uniformly convex two-body potential, see \eqref{eq:nonconvex1} and \eqref{eq:noncon2} for precise definitions, and show an analogue of Theorem \ref{T:subcrit} for such Hamiltonians, at sufficiently high temperature, cf. Theorem \ref{T:noncon} below. Accordingly, in what follows, $\beta$, the inverse temperature, will play the role of a perturbative parameter.

Our starting point is ${\mathcal{G}}= (\mathbb{Z}^d, {\mathscr{E}})$, where ${\mathscr{E}}$ refers from now on to the usual nearest-neighbor edge structure. Recall $ {\mathscr{E}}_{\Lambda}$ stands for the set of edges having at least one endpoint in $\Lambda$. We consider the family of measures ${\mu}_{\Lambda,\beta}^0$, for $\beta > 0$, $\Lambda \subset\subset \mathbb{Z}^d$, defined as in \eqref{eq:H_xi}, \eqref{eq:mu_Lambda}, with 
\begin{equation}\label{eq:nonconvex1}
{H}_{\Lambda}(\varphi) = \frac12 \sum_{e\in {\mathscr{E}}_{\Lambda}} V(\nabla \varphi (e)),
\end{equation}
where $V:\mathbb{R}\to \mathbb{R}$ is of the form
\begin{equation}\label{eq:noncon2.0}
V=U+g,
\end{equation}
and 
\begin{equation}\label{eq:noncon2}
\begin{split}
&U,g\in C^{2,\alpha}(\mathbb{R},\mathbb{R}_+), \text{ for some }\alpha>0, \, U, g \text{ are even functions},\\
&c^{-1}\leq U'' \leq c, \, \text{supp}(g)\subset \mathbb{R} \text{ is compact, and } V(\eta) \geq A\eta^2 -B, \eta \in \mathbb{R},
\end{split}
\end{equation}
for some $c\in (0,1)$ and $A>0$, $B \in \mathbb{R}$. 

We now consider the restriction of ${\mu}_{\Lambda,\beta}^0$ to the ``even'' field, to which we will eventually apply the results of Section \ref{S:sprinkle}. Let ${\mathbb{Z}}_e^d= \{ x \in \mathbb{Z}^d; \sum_{1\leq i \leq d}x_i \text{ is even} \}$, be the even sublattice, ${\mathbb{Z}}_o^d= \mathbb{Z}^d \setminus {\mathbb{Z}}_e^d$. We write $\widetilde{\Gamma}=\{ x\in \mathbb{Z}^d; |x|_1 =2 \}$, a subset of ${\mathbb{Z}}_e^d$, and consider the graph $\widetilde{\mathcal{G}}= ({\mathbb{Z}}_e^d, \widetilde{\mathscr{E}})$, whereby $(x,y)\in \widetilde{\mathscr{E}}$ if $x-y \in \widetilde{\Gamma}$. For $\Lambda \subset \mathbb{Z}^d$, we write $\Lambda_e = \Lambda \cap {\mathbb{Z}}_e^d$, $\Lambda_{\text{o}}= \Lambda \setminus \Lambda_e$, and $\tilde{\partial}K$ denotes the outer vertex boundary of $K \subset {\mathbb{Z}}_e^d$ in $\widetilde{\mathcal{G}}$, i.e. $x \in \tilde{\partial}K$ if and only if $x -y \in \widetilde{\Gamma}$ for some $y \in K$.

To keep the exposition simple, we will consider sets $\Lambda \in F^*$, where $F^*$ consists of all finite subsets $\Lambda$ of $\mathbb{Z}^d$ with the property that, if $x \in \Lambda_e$ and $y$ is a neighbor of $x$ in $\mathcal{G}$, i.e. $|y-x|_1=1$, then $x \in \Lambda$ (dropping this assumption on $F^*$ would result in additional one-body potentials in the representation below). Note that, for $\Lambda \in F^*$, we have $\partial \Lambda = \tilde{\partial}\Lambda_e$. As in \eqref{eq:W}, we consider the set ${\mathscr{W}}_{\beta} (V)$ of weak limits of measures ${\mu}_{{\Lambda_n},\beta}^0$, as $\Lambda_n \nearrow \mathbb{Z}^d$, with the additional requirement that $\Lambda_n \in F^*$ for all $n \geq 0$. We will soon argue that ${\mathscr{W}}_{\beta} (V)$ is non-empty at high temperatures, see Lemma \ref{L:tightnoncon} below.

We define a probability measure $\tilde{\mu}_{{\Lambda_e}, \beta}^0$ on $(\Omega_{\Lambda_e }, \mathcal{F}_{\Lambda_e}$), for $\beta > 0$ and $\Lambda \in  F^*$, as
\begin{equation}\label{eq:noncon3.0}
d \tilde{\mu}_{\Lambda_e, \beta}^0 ((\varphi_x)_{x\in \Lambda_e}) = \frac{1}{ \widetilde{Z}_{\Lambda_e, \beta}^{0} } e^{ -  \widetilde{H}_{\Lambda_e, \beta}(\varphi_e)} \prod_{x \in \Lambda_e} d\varphi_x \prod_{y \in \mathbb{Z}_e^d\setminus \Lambda_e}\delta_0(\varphi_y),
\end{equation}
where $\varphi_e=(\varphi_x)_{x \in \mathbb{Z}_e^d}$, 
\begin{equation}\label{eq:noncon3}
\widetilde{H}_{\Lambda_e, \beta}(\varphi_e) =\sum_{x \in\Lambda_e\cup \tilde{\partial}\Lambda_e} \bigg[\frac{1}{2d} \sum_{y: |y-x|_1=1} \widetilde{V}_{\beta}((\varphi_{y+e_i} - \varphi_x)_{i \in \mathbb{U}}) \bigg]
\end{equation}
with $\mathbb{U}=\{-d,\dots,d \}$, $e_i$ the standard unit vectors in $\mathbb{Z}^d$, $e_{-i}=-e_i$, for $1\leq i \leq d$, and
\begin{equation}\label{eq:noncon4}
\widetilde{V}_{\beta}((\eta_i)_{i \in \mathbb{U}}) = -\log \int_{\mathbb{R}}e^{-\beta\sum_{i \in \mathbb{U}} V(\eta_i - s)} ds.
\end{equation}
The assumption \eqref{eq:noncon2} guarantees that the normalizing factor $\widetilde{Z}_{\Lambda_e, \beta}^{0} $ in \eqref{eq:noncon3.0} is finite. Moreover, the term in brackets in \eqref{eq:noncon3} is a function of the (even) gradients $(\varphi_{x+z}-\varphi_x)_{z \in \widetilde{\Gamma}}$. Hence, by associating to $x \in \Lambda_e\cup \tilde{\partial}\Lambda_e$ the set $X=\{ x\} \cup \{ z \in \mathbb{Z}_e^d; \, z-x \in \widetilde{\Gamma} \} \equiv B_{\widetilde{\mathcal{G}}}(x,1)$, the unit ball around $x$ in $\widetilde{\mathcal{G}}$, one sees that $\widetilde{H}_{\Lambda_e, \beta}(\varphi_e)$ is a gradient Hamiltonian of the form \eqref{eq:H_pots}, with
\begin{equation}\label{eq:noncon4.0}
V_X((\varphi_{x+z}-\varphi_x)_{z\in \widetilde{\Gamma}}) = \frac{1}{2d} \sum_{y: |y-x|_1=1} \widetilde{V}_{\beta}((\varphi_{y+e_i} - \varphi_x)_{i \in \mathbb{U}}), \text{ if }X= B_{\widetilde{\mathcal{G}}}(x,1) 
\end{equation}
(and $\mathscr{E}(X)= \{ e \in \widetilde{\mathscr{E}}; \, x(e)= x\}$; we are tacitly identifying the vertex set $\mathbb{Z}_e^d$ with a copy of $\mathbb{Z}^d$). However, as opposed to the potential $V$ entering in \eqref{eq:nonconvex1}, $V_X$ is no longer a two-body interaction, cf. \eqref{eq:noncon4}. The key features of the measure $\tilde{\mu}_{{\Lambda_e}, \beta}^0$ are summarized in the following lemma.

\begin{lem} $(\Lambda \in F^*, \text{$V$ as in \eqref{eq:noncon2.0}, \eqref{eq:noncon2}})$ \label{L:noncon6}

\medskip
\noindent (i) For any $A \in \mathcal{F}_{\Lambda_e} =\sigma(\varphi_x; \, x \in \Lambda_e)$ and $\beta > 0$,
\begin{equation}\label{eq:noncon5}
{\mu}_{\Lambda,\beta}^0 (A)= \tilde{\mu}_{{\Lambda_e}, \beta}^0(A).
\end{equation}
(ii) The function $V_X$ in \eqref{eq:noncon4.0} satisfies \eqref{eq:smooth}, \eqref{eq:sym} and \eqref{eq:trans_inv}. Moreover, there exists $c_{24} > 0$ such that, if $\sqrt{\beta}\Vert g'' \Vert_{L^1(\mathbb{R})}< c_{24}$, then $V_X$ also satisfies \eqref{eq:elliptic}, for suitable $c_0\in [1,\infty)$.
\end{lem} 

\begin{proof} We first show \textit{(i)}. Because $\Lambda \in F^*$, the set of unordered nearest-neighbor edges in $\mathbb{Z}^d$ with at least one endpoint in $\Lambda$ can be written as $\bigcup_{y \in\Lambda_0}\bigcup_{x: |x-y|_1 = 1}\{ x,y\}$, and the union is disjoint. Hence, by symmetry of $H_{\Lambda}(\varphi)$, see \eqref{eq:nonconvex1}, we can write, for $A \in \mathcal{F}_{\Lambda_e}$ and $\beta > 0$,
\begin{equation}\label{eq:noncon5.00}
\begin{split}
{\mu}_{\Lambda,\beta}^0 (A) &= ({Z}_{\Lambda, \beta}^{0})^{-1} \int 1_A(\varphi_e)\exp \Big[ - \beta \sum_{y  \in\Lambda_0} \sum_{x: |x-y|_1 = 1} V(\varphi_x-\varphi_y)\Big] \prod_{x \in \Lambda} d\varphi_x \prod_{z \notin \Lambda} \delta_0(\varphi_z)\\
&\stackrel{\eqref{eq:noncon4}}{=} ({Z}_{\Lambda, \beta}^{0})^{-1} \int 1_A(\varphi_e) \Big[\prod_{y\in \Lambda_0}e^{-\widetilde{V}_{\beta}((\varphi_{y+e_i})_{i \in \mathbb{U}})} \Big]\prod_{x \in \Lambda_e} d\varphi_x \prod_{z \notin \Lambda} \delta_0(\varphi_z).
\end{split}
\end{equation}
But $\widetilde{V}_{\beta}((\varphi_{y+e_i})_{i \in \mathbb{U}}) = \widetilde{V}_{\beta}((\varphi_{y+e_i} -\varphi_x)_{i \in \mathbb{U}})$ for any $x$, which follows from \eqref{eq:noncon4}, and therefore
\begin{equation}\label{eq:noncon5.01}
\sum_{y\in \Lambda_0} \widetilde{V}_{\beta}((\varphi_{y+e_i})_{i \in \mathbb{U}}) = \sum_{y\in \Lambda_0} \frac{1}{2d} \sum_{x: |x-y|_1=1} \widetilde{V}_{\beta}((\varphi_{y+e_i} -\varphi_x)_{i \in \mathbb{U}}) \stackrel{\eqref{eq:noncon3}}{=} \widetilde{H}_{\Lambda_e, \beta}(\varphi_e) - B(\varphi), 
\end{equation}
where 
$$
B(\varphi)= \sum_{x \in \tilde{\partial}\Lambda_e} \frac{1}{2d} \sum_{y \in \mathbb{Z}_o^d\setminus \Lambda_o: |y-x|_1=1} \widetilde{V}_{\beta}((\varphi_{y+e_i} -\varphi_x)_{i \in \mathbb{U}}).
$$
The boundary term $B(\cdot)$ is a function of $\varphi_{|_{\Lambda^c}}$ alone (this is because, if $x \in \tilde{\partial}\Lambda_e =\partial \Lambda$ and $|y-x|_1=1$ but $y \notin \Lambda$, then $\text{dist}_{\ell^1}(y, \Lambda)=2$). Thus, inserting \eqref{eq:noncon5.01} into \eqref{eq:noncon5.00}, the claim \eqref{eq:noncon5} readily follows, in view of \eqref{eq:noncon3.0}.

We now argue that \textit{(ii)} holds. The potential $V_X$ defined in \eqref{eq:noncon4.0} inherits properties \eqref{eq:smooth}, \eqref{eq:sym} from $V$, see \eqref{eq:noncon2} and \eqref{eq:noncon4}, and \eqref{eq:trans_inv} is plain. The last part of the assertion is a direct consequence of Theorem 3.4 and Remark 3.12 in \cite{DC12}, applied with $q=1$. Note in particular that the last part of condition \eqref{eq:elliptic}, which requires $\partial^2_{x,y}V_X =0$ for any $x\neq y$ with $(x,y) \notin \widetilde{\mathscr{E}}$, is trivially satisfied, because $\widetilde{V}_{\beta}((\varphi_{y+e_i} - \varphi_x)_{i \in \mathbb{U}})$, for fixed $y \in \mathbb{Z}_o^d$, and $x$ with $|x-y|_1=1$, cf. \eqref{eq:noncon4.0}, is a function of $(\varphi_{y+e_i})_{i\in \mathbb{U}}$, and any two points $y+e_i$, $y+e_j$, $i \neq j$, are neighbors in $\widetilde{\mathcal{G}}$ by definition, i.e. $e_i - e_j = e_i +e_{-j} \in \widetilde{\Gamma}$.
\end{proof}

\begin{remark} \label{R:gensetup} In view of Lemma \ref{L:noncon6}, the measure $\tilde{\mu}_{{\Lambda_e}, \beta}^0$ can be viewed as the restriction of ${\mu}_{{\Lambda}, \beta}^0$ to the even field, for $\beta > 0$ and $\Lambda \in F^*$. Moreover, when $\beta$ is sufficiently small, it is within the realm of the setup in Section \ref{S:1}. In particular, for such $\beta$, the effective Hamiltonian \eqref{eq:noncon3} yields a random-walk representation for covariances of $\tilde{\mu}_{{\Lambda_e}, \beta}^0$ by means of Lemma \ref{L:HS}. The corresponding quenched walk is a jump process on the graph $\widetilde{\mathcal{G}}$ (which is not isomorphic to $\mathcal{G}$, hence our somewhat general setup). \hfill $\square$
\end{remark}

We first note that Lemma \ref{L:noncon6} yields tightness at high temperatures.
\begin{lem} \label{L:tightnoncon}$(\Lambda \in F^*, \, \text{$V$ as in \eqref{eq:noncon2.0}, \eqref{eq:noncon2}})$
\begin{equation}\label{eq:tight2}
{\mathscr{W}}_{\beta} (V) \neq \emptyset, \text{ for all $\sqrt{\beta}\Vert g'' \Vert_{L^1(\mathbb{R})}< c_{24}$.}
\end{equation}
\end{lem}

\begin{proof}
On account of Lemma \ref{L:noncon6}, \textit{(ii)}, Lemma \ref{L:BL} applies, yielding, for all $\sqrt{\beta}\Vert g'' \Vert_{L^1(\mathbb{R})}< c_{24}$, $\Lambda \in F^*$,
\begin{equation}\label{eq:00final00}
\sup_{\Lambda} \sup_{x \in\Lambda_e} \mathbb{E}_{{\mu}_{\Lambda,\beta}^0}[e^{\varphi_x}] = \sup_{\Lambda} \sup_{x \in\Lambda_e} \mathbb{E}_{\tilde{\mu}_{\Lambda_e,\beta}^0}[e^{\varphi_x}] < \infty.
\end{equation}
For $x \in \Lambda_o$, write
\begin{equation*}
\mathbb{E}_{{\mu}_{\Lambda,\beta}^0}[e^{\varphi_x}] \stackrel{\eqref{eq:CONDEXP}}{\leq} \sup_{|\xi|_{\infty} \leq 1} \mathbb{E}_{{\mu}_{\{0\},\beta}^\xi}[e^{\varphi_0}] + 2d \max_{x \in\Lambda_e} {\mu}_{\Lambda,\beta}^0(|\varphi_x |> 1),
\end{equation*}
the first term on the right-hand side is finite by assumption \eqref{eq:noncon2}, the second one is bounded uniformly in $\Lambda$ using \eqref{eq:00final00}. All in all, $\sup_{\Lambda\in F^*} \sup_{x \in\Lambda} \mathbb{E}_{{\mu}_{\Lambda,\beta}^0}[e^{\varphi_x}]< \infty$, for all $\beta> 0$ such that $\sqrt{\beta}\Vert g'' \Vert_{L^1(\mathbb{R})}< c_{24}$, which implies that the family $\{ {\mu}_{\Lambda,\beta}^0; \Lambda \in F^*\}$ is tight, cf. \eqref{eq:tight}, and \eqref{eq:tight2} follows.
\end{proof}

Finally, we state our main result regarding the connectivity decay of level sets of measures $\mu_{\beta} \in {\mathscr{W}}_{\beta} (V)$ above large levels and at high temperatures, for the class of non-convex potentials $V$ considered above. Recall the definition of $h^+(\cdot)$ in \eqref{eq:renorm10}.

\begin{thm} $(V \, \text{as in }  \eqref{eq:noncon2.0}, \eqref{eq:noncon2}, \, \sqrt{\beta}\Vert g'' \Vert_{L^1(\mathbb{R})}< c_{24})$ \label{T:noncon}
\begin{equation}\label{eq:noncon10}
h^+(\mu_{\beta})< \infty, \text{ for all $\mu_{\beta} \in {\mathscr{W}}_{\beta} (V).$}
\end{equation}
\end{thm}
\begin{proof}
Let $\mu \equiv \mu_{\beta} \in {\mathscr{W}}_{\beta} (V)$. We remove the $\beta$-dependence of all quantities throughout the proof, keeping in mind that $\sqrt{\beta}\Vert g'' \Vert_{L^1(\mathbb{R})}< c_{24}$. Write $\tilde{\mu} = \mu_{\vert_{\widetilde{\mathcal{F}}}}$, where $\widetilde{\mathcal{F}}=\sigma(\varphi_x; x \in \mathbb{Z}_e^d)$.
Observe that, if $\pi$ is a nearest-neighbor path in $\mathcal{G} = ({\mathbb{Z}}^d, {\mathscr{E}})$, the usual Euclidean lattice, then, by definition of $\widetilde{\Gamma}$, its trace on $\widetilde{\mathcal{G}}$ (i.e. the ordered sequence of vertices in $\mathbb{Z}^d_e$ visited by $\pi$) is a nearest neighbor path in $\widetilde{\mathcal{G}}= ({\mathbb{Z}}_e^d, \widetilde{\mathscr{E}})$. Hence, writing $\hat{x}=x$, if $x \in \mathbb{Z}^d_e$, and $\hat{x}= x+e_1$ if $x \in \mathbb{Z}^d_o$, and introducing the events $\widetilde{E}_{L,x}^h=\{ B_{\widetilde{\mathcal{G}}}(\hat{x},L) \stackrel{\geqslant h}{\longleftrightarrow} \partial B_{\widetilde{\mathcal{G}}}(\hat{x},2L) \}$, for $x \in \mathbb{Z}^d$, $L \geq 1$, referring to the existence of a nearest-neighbor path of sites in $\widetilde{\mathcal{G}}$ with field value exceeding $h$
joining $B_{\widetilde{\mathcal{G}}}(\hat{x},L)$, the ball of radius $L$ (in the graph distance for $\widetilde{\mathcal{G}}$) around $\hat x$, to the boundary of a concentric ball of radius $2L$, one obtains
\begin{equation}\label{eq:noncon11}
\mu(x \stackrel{\geqslant h}{\longleftrightarrow} \partial B_{\mathcal{G}}(x,4L+1)) \leq \mu(\widetilde{E}_{L,x}^h)= \tilde{\mu}(\widetilde{E}_{L,x}^h), \text{ for all $L \geq 1$, $x\in \mathbb{Z}^d$}.
\end{equation}
Now, if $(\Lambda_n)_{n \geq 0}$ is a sequence of increasing sets in $F^*$ such that ${\mu}_{\Lambda_n}^0 \stackrel{w}{\rightarrow} \mu$, then, by \eqref{eq:noncon5}, and with a slight abuse of notation, $\tilde{\mu}_{\Lambda_e}^0 \stackrel{w}{\rightarrow} \tilde{\mu}$. Hence, by Lemma \ref{L:noncon6}, \textit{(ii)}, which applies due to our assumption on $\beta$, we infer that Theorem \ref{P:dec_ineq} applies to $\tilde{\mu}$, viewed as a measure on $\Omega^{\widetilde{\mathcal{G}}}$ (in particular, the reference measure $P^*$ appearing in \eqref{eq:eta_error} refers to the law of simple random walk on $\widetilde{\mathcal{G}}$). Thus, mimicking the proof of Theorem \ref{T:subcrit}, using the inequality \eqref{eq:dec_ineq}, applied to $\tilde{\mu}$, we deduce that $\tilde{\mu}(\widetilde{E}_{L,x}^h) \leq ce^{-c'L^{\varepsilon}}$, for some $\varepsilon > 0$ and all $L \geq 1$, if $h$ is chosen sufficiently large. In view of \eqref{eq:noncon11}, the claim follows.
\end{proof}

\begin{remark}\label{R:final}
1) (Relaxing the uniform ellipticity condition). Our main result, Theorem~\ref{P:dec_ineq} crucially hinges on two tools: the availability of the Helffer-Sj\"ostrand representation (in finite volume) to carry out the sprinkling argument of Section~\ref{S:sprinkle}, and the Brascamp-Lieb inequality to bound the arising error term. The latter seems to require the hard lower bound in \eqref{eq:elliptic}. However, one could attempt to replace the uniform upper bound by a suitable moment assumption on the conductances $a_{x,y}^{\xi}(\Phi_t)$, see \eqref{eq:RATES}, in the spirit of \cite{ACDS16}, provided one recovers Lemma~\ref{L:HS}, and adapts the comparison estimates of Section~\ref{S:error}, which requires at the very least an annealed version of \eqref{eq:GF_bounds}. 

\noindent 2) (Comparison with the GFF). In the interpolation argument, Proposition \ref{L:monot}, one may try to replace $\widetilde{\varphi}$ by the actual Gaussian free field, and attempt to compare the two objects directly (see for instance \cite{Mi11} for results in this direction in dimension $2$). In view of \eqref{eq:harmonic_ext}, this leads to the following question (for simplicity, consider $\mu_{\Lambda}^{\xi}$ the uniformly convex nearest-neighbor gradient interface Hamiltonian, with potential as in \eqref{eq:pot2body}): can one write
$$
\mu_{\Lambda}^{\xi}(\varphi \in \cdot) = \mu_{\Lambda}^0(\varphi + \widetilde{\xi}(\varphi)\in \cdot),
$$ 
for a function $ \widetilde{\xi}$ extending $\xi$ to the interior of $\Lambda$, such that $ \widetilde{\xi}$ is close to the discrete harmonic extension of $\xi$, with ``overwhelming'' probability? \hfill $\square$
\end{remark}

\bibliography{rodriguez}

\begin{thebibliography}{10}

\bibitem{ACDS16}
S.~Andres, A.~Chiarini, J.-D. Deuschel, and M.~Slowik.
\newblock Quenched invariance principle for random walks with time-dependent
  ergodic degenerate weights.
\newblock {\em Preprint}, available at ar{X}iv:1602.01760, 2016.

\bibitem{BK07}
M.~Biskup and R.~Koteck{{\'y}}.
\newblock Phase coexistence of gradient {G}ibbs states.
\newblock {\em Probab. Theory Related Fields}, 139(1-2):1--39, 2007.

\bibitem{BS11}
M.~Biskup and H.~Spohn.
\newblock Scaling limit for a class of gradient fields with nonconvex
  potentials.
\newblock {\em Ann. Probab.}, 39(1):224--251, 2011.

\bibitem{BL76}
H.~J. Brascamp and E.~H. Lieb.
\newblock On extensions of the {B}runn-{M}inkowski and {P}r{\'e}kopa-{L}eindler
  theorems, including inequalities for log concave functions, and with an
  application to the diffusion equation.
\newblock {\em J. Functional Analysis}, 22(4):366--389, 1976.

\bibitem{DC12}
C.~Cotar and J.-D. Deuschel.
\newblock Decay of covariances, uniqueness of ergodic component and scaling
  limit for a class of {$\nabla\phi$} systems with non-convex potential.
\newblock {\em Ann. Inst. Henri Poincar{\'e} Probab. Stat.}, 48(3):819--853,
  2012.

\bibitem{DD05}
T.~Delmotte and J.-D. Deuschel.
\newblock On estimating the derivatives of symmetric diffusions in stationary
  random environment, with applications to {$\nabla\phi$} interface model.
\newblock {\em Probab. Theory Related Fields}, 133(3):358--390, 2005.

\bibitem{DGI00}
J.-D. Deuschel, G.~Giacomin, and D.~Ioffe.
\newblock Large deviations and concentration properties for {$\nabla\phi$}
  interface models.
\newblock {\em Probab. Theory Related Fields}, 117(1):49--111, 2000.

\bibitem{DRS14.3}
A.~Drewitz, B.~R{{\'a}}th, and A.~Sapozhnikov.
\newblock Local percolative properties of the vacant set of random
  interlacements with small intensity.
\newblock {\em Ann. Inst. Henri Poincar{\'e} Probab. Stat.}, 50(4):1165--1197,
  2014.

\bibitem{DRS14.2}
A.~Drewitz, B.~R{{\'a}}th, and A.~Sapozhnikov.
\newblock On chemical distances and shape theorems in percolation models with
  long-range correlations.
\newblock {\em J. Math. Phys.}, 55(8):083307, 30, 2014.

\bibitem{Fu11}
M.~Fukushima, Y.~Oshima, and M.~Takeda.
\newblock {\em Dirichlet forms and symmetric {M}arkov processes}, volume~19 of
  {\em de Gruyter Studies in Mathematics}.
\newblock Walter de Gruyter \& Co., Berlin, extended edition, 2011.

\bibitem{Fu05}
T.~Funaki.
\newblock Stochastic interface models.
\newblock In {\em Lectures on probability theory and statistics}, volume 1869
  of {\em Lecture Notes in Math.}, pages 103--274. Springer, Berlin, 2005.

\bibitem{FS97}
T.~Funaki and H.~Spohn.
\newblock Motion by mean curvature from the ginzburg-landau interface model.
\newblock {\em Communications in Mathematical Physics}, 185(1):1--36, 1997.

\bibitem{GOS01}
Giambattista Giacomin, Stefano Olla, and Herbert Spohn.
\newblock Equilibrium fluctuations for {$\nabla\phi$} interface model.
\newblock {\em Ann. Probab.}, 29(3):1138--1172, 2001.

\bibitem{Gr99}
G.~Grimmett.
\newblock {\em Percolation}, volume 321 of {\em Grundlehren der Mathematischen
  Wissenschaften}.
\newblock Springer-Verlag, Berlin, second edition, 1999.

\bibitem{HS94}
B.~Helffer and J.~Sj{{\"o}}strand.
\newblock On the correlation for {K}ac-like models in the convex case.
\newblock {\em J. Statist. Phys.}, 74(1-2):349--409, 1994.

\bibitem{KS91}
I.~Karatzas and S.~E. Shreve.
\newblock {\em Brownian motion and stochastic calculus}, volume 113 of {\em
  Graduate Texts in Mathematics}.
\newblock Springer-Verlag, New York, second edition, 1991.

\bibitem{LL10}
G.~F. Lawler and V.~Limic.
\newblock {\em Random walk: a modern introduction}, volume 123.
\newblock Cambridge University Press, 2010.

\bibitem{Mi11}
J.~Miller.
\newblock Fluctuations for the {G}inzburg-{L}andau {$\nabla\phi$} interface
  model on a bounded domain.
\newblock {\em Comm. Math. Phys.}, 308(3):591--639, 2011.

\bibitem{NS97}
A.~Naddaf and T.~Spencer.
\newblock On homogenization and scaling limit of some gradient perturbations of
  a massless free field.
\newblock {\em Comm. Math. Phys.}, 183(1):55--84, 1997.

\bibitem{PR13}
S.~Popov and B.~R{\'a}th.
\newblock On decoupling inequalities and percolation of excursion sets of the
  gaussian free field.
\newblock {\em J. Stat. Phys.}, 159(2):312--320, 2015.

\bibitem{PT12}
S.~Popov and A.~Teixeira.
\newblock Soft local times and decoupling of random interlacements.
\newblock {\em J. Eur. Math. Soc.}, 17(10):2545--2593, 2015.

\bibitem{PRS15}
E.~B. Procaccia, R.~Rosenthal, and A.~Sapozhnikov.
\newblock Quenched invariance principle for simple random walk on clusters in
  correlated percolation models.
\newblock {\em Probability Theory and Related Fields}, pages 1--39, 2015.

\bibitem{Ro14}
P.-F. Rodriguez.
\newblock Level set percolation for random interlacements and the {G}aussian
  free field.
\newblock {\em Stochastic Process. Appl.}, 124(4):1469--1502, 2014.

\bibitem{RoS13}
P.-F. Rodriguez and A.-S. Sznitman.
\newblock Phase transition and level-set percolation for the {G}aussian free
  field.
\newblock {\em Comm. Math. Phys.}, 320(2):571--601, 2013.

\bibitem{Sa14}
A.~Sapozhnikov.
\newblock Random walks on infinite percolation clusters in models with
  long-range correlations.
\newblock {\em Ann. Probab. (to appear)}, 2014.

\bibitem{Sz10}
A.-S. Sznitman.
\newblock Vacant set of random interlacements and percolation.
\newblock {\em Ann. Math.}, 171(3):2039--2087, 2010.

\bibitem{Sz12a}
A.-S. Sznitman.
\newblock Decoupling inequalities and interlacement percolation on
  {$G\times\Bbb Z$}.
\newblock {\em Invent. Math.}, 187(3):645--706, 2012.

\bibitem{Sz12b}
A.-S. Sznitman.
\newblock {\em Topics in occupation times and {G}aussian free fields}.
\newblock Zurich Lectures in Advanced Mathematics. European Mathematical
  Society (EMS), Z\"urich, 2012.

\bibitem{Ve06}
Y.~Velenik.
\newblock Localization and delocalization of random interfaces.
\newblock {\em Probab. Surv.}, 3:112--169, 2006.

\end{thebibliography}
\bibliographystyle{plain}

\end{document}